\documentclass[11 pt, hidelinks, leqno]{amsart}

\usepackage[margin=2.5cm]{geometry}
\usepackage[english]{babel}
\usepackage{hyperref}
\usepackage{amsmath}
\usepackage{amsthm}
\usepackage{amssymb}
\usepackage{ucs}
\usepackage{mathtools}
\usepackage{enumerate}
\usepackage[utf8x]{inputenc}
\usepackage[T1]{fontenc}
\usepackage{color}
\usepackage{tikz}
\usepackage{diagbox}
\usepackage{tikz-cd}
\usepackage{tkz-graph}

\AtBeginDocument{%
   \def\MR#1{}
}

\theoremstyle{plain}
\newtheorem{theorem}{Theorem}[section]
\newtheorem{theoremA}{Theorem}

\newtheorem{lemma}[theorem]{Lemma}
\newtheorem{proposition}[theorem]{Proposition}
\newtheorem{corollary}[theorem]{Corollary}
\theoremstyle{definition}
\newtheorem{definition}[theorem]{Definition}
\newtheorem{example}[theorem]{Example}

\newtheorem{remark}[theorem]{Remark}
\newtheorem{propdef}[theorem]{Proposition-Definition}

\newcommand{\R}{\mathbb{R}}

\newcommand{\C}{\mathbb{C}}
\newcommand{\N}{\mathbb{N}}
\newcommand{\rL}{{\rm L}}
\newcommand{\G}{\mathcal{G}}
\newcommand{\T}{\mathcal{T}}
\newcommand{\A}{\mathcal{A}}
\newcommand{\B}{\mathcal{B}}
\newcommand{\Z}{\mathbb{Z}}

\newcommand{\Gammahat}{\widehat{\Gamma}}
\newcommand{\w}{\wr_{*}}
\newcommand{\wH}{\wr_{*,H}}
\newcommand{\wreath}{G\w\SN}
\newcommand{\wreathH}{G\wH\SN}
\newcommand{\BS}{\mathrm{BS}}
\newcommand{\wreathhat}{\widehat{\wreath}}
\newcommand{\wreathHhat}{\widehat{\wreathH}}
\newcommand{\Ghat}{\widehat{G}}

\newcommand{\ebar}{\overline{e}}
\newcommand{\vbar}{\overline{v}}

\newcommand{\ot}{\otimes}
\newcommand{\id}{{\rm id}}
\newcommand{\tr}{{\rm tr}}

\newcommand{\Pol}{{\rm Pol}}
\newcommand{\Irr}{{\rm Irr}}
\newcommand{\Linf}{{\rm L}^\infty}
\newcommand{\SN}{S_N^+}
\newcommand{\F}{\mathbb{F}}

\newcommand{\Mor}{\mathrm{Mor}}
\newcommand{\HNN}{\mathrm{HNN}}
\newcommand{\Aut}{\mathrm{Aut}}

\newcommand\scalemath[2]{\scalebox{#1}{\mbox{\ensuremath{\displaystyle #2}}}}

\title{Operator algebras of free wreath products}
\author{Pierre Fima}
\address{Pierre Fima
\newline
Universit\'e Paris Cit\'e and Sorbonne Universit\'e, CNRS, IMJ-PRG, F-75013 Paris, France.}
\email{pierre.fima@imj-prg.fr}
\thanks{P.F. is partially supported by the ANR project ANCG (No. 
 ANR-19-CE40-0002), ANR project AODynG (No. ANR-19-CE40-0008), and Indo-French 
 Centre for the Promotion of Advanced Research - CEFIPRA}
 
 \author{Arthur Troupel}
\address{Arthur Troupel
\newline
Universit\'e Paris Cit\'e and Sorbonne Universit\'e, CNRS, IMJ-PRG, F-75013 Paris, France.}
\email{arthur.troupel@imj-prg.fr}

\begin{document}
\date{}
\maketitle

\begin{abstract}
   We give a description of operator algebras of free wreath products in terms of fundamental algebras of graphs of operator algebras as well as an explicit formula for the Haar state. This allows us to deduce stability properties for certain approximation properties such as exactness, Haagerup property, hyperlinearity and K-amenability. We study qualitative properties of the associated von Neumann algebra: factoriality, fullness, primeness and absence of Cartan subalgebra and we give a formula for Connes' $T$-invariant and $\tau$-invariant. We also study maximal amenable von Neumann subalgebras. Finally, we give some  explicit computations of K-theory groups for C*-algebras of free wreath products. As an application we show that the reduced C*-algebras of quantum reflection groups are pairwise non-isomorphic.
\end{abstract}

\section{Introduction}

\noindent After Woronowicz developed the general theory of compact quantum groups \cite{wor87,wor88,wor98} which covers, in particular, the deeply investigated $q$-deformations of compact Lie groups constructed by Drinfeld and Jimbo \cite{dri86,dri87,jim85}, new examples of compact quantum groups emerged from the works of Wang \cite{Wa95,wan98}. A particularly relevant quantum group for the present work is the quantum permutation group $S_N^+$ (see Section \ref{SectionQperm}) for which the representation theory has been completely understood by Banica \cite{Ba99} and its operator algebras (as well as approximation properties) have been extensively studied by Brannan and Voigt \cite{Br13,Vo17}. In order to find natural quantum subgroups of $S_N^+$, quantum automorphism groups of finite graphs were introduced in \cite{bic03}. Since the classical automorphism group of the graph obtained as the $N$-times disjoint union of a finite connected graph $\mathcal{G}$ is given by the wreath product of the classical automorphism group of $\mathcal{G}$ by $S_N$, it became important to introduce a free version of the wreath product.

\vspace{0.2cm}

\noindent The free wreath product construction was achieved by Bichon in \cite{Bi04}: given a compact quantum group $G$, Bichon constructed and studied the free wreath product $\wreath$ of $G$ by $S_N^+$ (see Section \ref{Sectionfwp}). Some generalisations of free wreath products by $S_N^+$ were considered in \cite{pit14,FP16,TW18,FS18,Fr22}. Concerning the representation theory of $\wreath$, it has been studied in many specific cases before achieving the general result \cite{bv09,lem14,lt14,pit14,FP16,TW18,FS18}. While some approximation properties are known to be preserved by the free wreath product construction, very little is known about the operator algebras associated to $\wreath$. It seems that the main reason why these operator algebras are not already completely understood is that there is no explicit computation of the Haar state. The first result of the present paper is an explicit formula for the Haar state of $\wreath$ (Theorem \ref{propreduced}). While studying free wreath products, we realized that the full C*-algebra $C(\wreath)$ of the free wreath product can actually be written as the full fundamental C*-algebra of a graph of C*-algebras \cite{FF14}. However, in contrast with \cite{FF14}, $\wreath$ is not the fundamental quantum group of a graph of quantum groups. Nevertheless, our description in terms of graph of C*-algebras allows us to produce a specific state, called the \textit{fundamental state} in \cite{FF14}, which is shown to be the Haar state of $C(\wreath)$. This allows a complete description of the reduced C*-algebra $C_r(\wreath)$ as the reduced fundamental C*-algebra of a graph of C*-algebras as well as a description of the von Neumann algebra $\Linf(\wreath)$ as the fundamental von Neumann algebra of a graph of von Neumann algebras.

\vspace{0.2cm}

\noindent Concerning approximation properties, our description in terms of graph of operator algebras of the free wreath product allows us to deduce the following. To state our result, we will the following terminology: we will always denote by $\widehat{G}$ the discrete dual of a compact quantum group $G$, and we say that $\widehat{G}$ has the \textit{Haagerup property} if $\Linf(G)$ has the Haagerup property (with respect to the Haar state see \cite{CS15}) and that $\widehat{G}$ is \textit{exact} when the reduced C*-algebra $C_r(G)$ is an exact C*-algebra. When $G$ is Kac, we say that its dual $\widehat{G}$ is \textit{hyperlinear} if the finite von Neumann algebra $\Linf(G)$ embeds in an ultraproduct of the hyperfinite ${\rm II}_1$-factor. Finally, $G$ is called \textit{co-amenable} when the canonical surjection from its full C*-algebra $C(G)$ to $C_r(G)$ is an isomorphism, and we say that $\widehat{G}$ is \textit{K-amenable} when this canonical surjection is a KK-equivalence.

\begin{theoremA}\label{THMA}Given $N\in\N^*$ and $G$ a compact quantum group. The following holds.
\begin{enumerate}
    \item $\widehat{G}$ is exact if and only if $\wreathhat$ is exact.
    \item $\Ghat$ has the Haagerup property if and only if $\wreathhat$ has the Haagerup property.
    \item If $G$ is Kac then $\Ghat$ is hyperlinear if and only if $\wreathhat$ is hyperlinear.
    \item
    \begin{itemize}\item If $N\geq 5$, $\wreath$ is not co-amenable for any $G$.
    \item If $N\in\{3,4\}$ then $\wreath$ is co-amenable if and only if $G$ is trivial.
    \item $G\wr_*S_{2}^+$ is co-amenable if and only if $G$ is trivial or if $G\simeq\Z/2\Z$.
    \end{itemize}
    \item $\Ghat$ is K-amenable if and only if $\wreathhat$ is K-amenable.
\end{enumerate}
\end{theoremA}

\noindent Let us mention that some of the statements in Theorem A are already known in specific cases but our proofs, based on our description in terms of graphs of operator algebras, are completely different. The stability of exactness under the free wreath product with $S_N^+$ is known \cite{lt14, FP16} for $N\geq 4$ while the stability of the Haagerup property is not known in the form above. Actually, most of the previous proofs of approximation properties were based on a monoidal equivalence argument so that it was not possible to deduce properties which are not stable under monoidal equivalence. To deduce approximation properties, the common trick was to strengthen the definition of these properties so that they become invariant under monoidal equivalence. Hence, what is known in relation with the Haagerup property is actually the stability of the \textit{central ACPAP} \cite{lem14,lt14,FP16,TW18}. One exception is the case $N=2$ which has been completely described by Bichon \cite{bic03}. From Bichon's description, it has been possible to deduce in \cite[Corollary 7.11]{DFSW16} that, when $G$ is Kac, $\widehat{G\wr_*S_2^+}$ has the Haagerup property if and only if $\Ghat$ has the Haagerup property. Concerning hyperlinearity, it is already known that $\widehat{S_N^+}$ (for all $N$) and the discrete duals of the quantum reflection groups $\widehat{\Z_s}\wr_*S_N^+$ (for $1\leq s\leq +\infty\,$ and $N\geq 4$) are residually finite hence hyperlinear \cite{BCF20}. The now classical argument is based on the topological generation method developed by Brannan-Collins-Vergnioux \cite{BCV17}. The stability of hyperlinearity of Theorem A is completely new and our argument is not based on the topological generation method. The stability of K-amenability is also new. It was known that $\widehat{S_N^+}$ is K-amenable \cite{Vo17} and, again by using a monoidal equivalence argument, it was known \cite{FM20} that if $G$ is torsion-free and satisfies the strong Baum-Connes conjecture then $\wreathhat$ is K-amenable. One of the consequence of our result on $K$-amenability which seems to have been unknown is that duals of quantum reflection groups are K-amenable. A more interesting application of our description in terms of graph of operator algebras is about the qualitative properties of the von Neumann algebra $\Linf(\wreath)$ for which nothing is known, outside of the factoriality in some specific cases. It is known that $\Linf(\wreath)$ is a full ${\rm II}_1$-factor whenever $G=\widehat{\Gamma}$ is the dual of a discrete group and $N\geq 8$ \cite{lem14} and it is also a non-Gamma ${\rm II}_1$-factor whenever $G$ is a compact matrix quantum group of Kac type and $N\geq 8$ \cite{Wa14}. Using results about the structure of amalgamated free products von Neumann algebra \cite{Ue11,Ue13,HV13,HI20} we are able to deduce the following. Let $(\sigma_t^G)_t$ be the modular group of the Haar state on $\Linf(G)$ and $T(M)$ the Connes' $T$ invariant of the von Neumann algebra $M$.

\begin{theoremA}\label{THMB} Suppose that $\Irr(G)$ is infinite and $N\geq 4$. Then ${\rm L}^\infty(\wreath)$ is a non-amenable full and prime factor without any Cartan subalgebra. Moreover:
    \begin{itemize}
        \item If $G$ is Kac then ${\rm L}^\infty(\wreath)$ is a type ${\rm II}_1$ factor.
        \item If $G$ is not Kac then ${\rm L}^\infty(\wreath)$ is a type ${\rm III}_\lambda$ factor for some $\lambda\neq 0$ and:
        $$T({\rm L}^\infty(\wreath))=\{t\in\R\,:\,\sigma_t^G=\id\}.$$
        Moreover, the Connes' $\tau$-invariant of the full factor $\Linf(\wreath)$ is the smallest topology on $\R$ for which the map $(t\mapsto\sigma_t^G)$ is continuous.
    \end{itemize}
\end{theoremA}

\noindent To our knowledge, there was no result about the maximal amenable von Neumann subalgebras of $\Linf(\wreath)$ in the literature. A von Neumann subalgebra $A\subset M$ is called \textit{with expectation} if there a normal faithful conditional expectation $P\rightarrow A$. $A\subset P$ is called \textit{maximal amenable} if $A$ is amenable and, for any intermediate von Neumann algebra $A\subset P\subset M$ if $P$ is amenable then $A=P$. $A\subset P$ is called \textit{maximal amenable with expectation} if $A$ is amenable with expectation and, for any $A\subset P\subset M$, if $P$ is amenable with expectation then $A=P$.

\vspace{0.2cm}

\noindent Recall that $\Linf(\wreath)$ is generated by $N$ free copies of $\Linf(G)$ given by $\nu_i\,:\,\Linf(G)\rightarrow\Linf(\wreath)$, $1\leq i\leq N$, and by $\Linf(S_N^+)\subset\Linf(\wreath)$ plus some relations (see Section \ref{Sectionfwp}). Let $u=(u_{ij})_{ij}\in M_N(\C)\ot\Linf(S_N^+)$ be the fundamental representation of $S_N^+$ so that $\Linf(S_N^+)$ is generated by the coefficients $u_{ij}$ of $u$. We use \cite{BH18} to deduce the following.

\begin{theoremA}\label{THMC}
Let $G$ be a compact quantum group and $N\geq2$. The following holds.
\begin{enumerate}
    \item If ${\rm L}^\infty(G)$ is amenable and $\Irr(G)$ is infinite then, for all $1\leq i\leq N$, the von Neumann subalgebra of ${\rm L}^\infty(\wreath)$ generated by $\{\nu_i(a)u_{ij}\,:\, a\in \Linf(G),\,1\leq j\leq N\}$ is maximal amenable with expectation in ${\rm L}^\infty(\wreath)$.
    \item If $G$ is Kac then ${\rm L}^\infty(S_4^+)\subset {\rm L}^\infty(G\wr_*S_4^+)$ is maximal amenable.
\end{enumerate}
\end{theoremA}

\noindent Finally, we compute K-theory groups of the C*-algebras $C(\wreath)$ and $C_r(\wreath)$ (recall that they are KK-equivalent whenever $\Ghat$ is K-amenable). In the next Theorem, we denote by $C_\bullet(G)$ either the full $C^*$-algebra or the reduced $C^*$-algebra of $G$.

\begin{theoremA}\label{THMD}For any compact quantum group $G$ and integer $N\in\N^*$ we have,
\begin{eqnarray*}
K_0(C_\bullet (\wreath)) &\simeq& K_0(C_\bullet(G)) \ot (\Z^{N^2}) \oplus K_0(C(S_N^+))/ (\Z^{N^2}) \\ &\simeq & \left\{\begin{array}{lcl}K_0(C_\bullet(G))^{\oplus N^2}/ \Z^{2N-2} &\text{if}&N\neq 3,\\K_0(C_\bullet(G))^{\oplus N^2}/\Z^3 &\text{if}& N=3 .\end{array}\right.  \text{ and,}\\
K_1(C_\bullet(\wreath))&\simeq&  K_1(C_\bullet(G))^{\oplus N^2} \oplus K_1(C(S_N^+)) \simeq \left\{\begin{array}{lcl}K_1(C_\bullet(G))^{\oplus N^2} \oplus \Z&\text{if}&N\geq 4,\\K_1(C_\bullet(G))^{\oplus N^2}&\text{if}&1\leq N\leq 3.\end{array}\right.\\.
\end{eqnarray*}
In particular, for the quantum reflection groups $H^{s+}_N = \widehat{\Z_s}\w S_N^+$, which have K-amenable duals, if $N\geq 4$,
$$K_0(C_\bullet(H^{s+}_N))\simeq\left\{
\begin{array}{lcl}\Z^{N^2-2N+2}&\text{if}&s=+\infty,\\ \Z^{sN^2-2N+2}&\text{if}&s<+\infty,\end{array}\right.
\quad K_1(C_\bullet(H^{s+}_N))\simeq\left\{
\begin{array}{lcl}\Z^{N^2+1}&\text{if}&s=+\infty,\\ \Z&\text{if}&s<+\infty,\end{array}\right. $$
and if $N\in \lbrace 1,2,3\rbrace,$
$$K_0(C_\bullet(H^{s+}_N))\simeq\left\{
\begin{array}{lcl}\Z^{N!}&\text{if}&s=+\infty,\\ \Z^{(s-1)N^2+N!}&\text{if}&s<+\infty,\end{array}\right.
\quad K_1(C_\bullet(H^{s+}_N))\simeq\left\{
\begin{array}{lcl}\Z^{N^2}&\text{if}&s=+\infty,\\ 0&\text{if}&s<+\infty,\end{array}\right. $$
\end{theoremA}

\noindent It seems that the only attempt to compute the K-theory of free wreath products was done in \cite{FM20} in which they compute K-theory groups of some quantum groups which are not free wreath product with $S_N^+$ but only monoidally equivalent to a free wreath products with $S_N^+$. Actually the K-theory is computed for some quantum groups of the form $G\wr_* SO_q(3)$, where the free wreath product is in the sense of \cite{FP16}. The method of computation in \cite{FM20}, which is based on the ideas of the works of Voigt \cite{Vo11,Vo17} and Voigt-Vergnioux \cite{VV13} is to prove the strong Baum-Connes property and then find an explicit projective resolution of the trivial action. However, they could not do it for free wreath products by $S_N^+$ but they showed that this can be done for $G\wr_* SO_q(3)$ with some specific $G$ such as free products of free orthogonal and free unitary quantum groups as well as duals of free groups. While our method cannot provide results on the stability of the strong Baum-Connes conjecture for free wreath products, it has the advantage to be extremely simple and also applicable to any free wreath product with $S_N^+$. Our K-theory computations are based on our decomposition as fundamental algebras and the results of \cite{FG18}. We also use the result of \cite{Vo17} on the computation of the K-theory of $C_\bullet(S_N^+)$. As mentioned to us by Adam Skalski, it seems unknown whether or not the reduced C*-algebras of quantum reflection groups $H_N^{s+}$ are isomorphic or not for different values of the parameters $N,s$. Our K-theory computation in Theorem \ref{THMD} solves this question: we show in Corollary \ref{CORReflection} that for all $N,M\geq 8$ one has $C_r(H_N^{s+})\simeq C_r(H_M^{t+})\Leftrightarrow (N,s)=(M,t)$. We do have to restrict the parameters to $N,M\geq 8$ since our proof uses the unique trace result of Lemeux \cite{lem14}, which is valid only for $N\geq 8$. We also compute some other K-theory groups of free wreath products, such as $\widehat{\mathbb{F}}_m\wr_*S_N^+$ and deduce that reduced (as well as full) C*-algebras $C_\bullet(\widehat{\mathbb{F}}_m\wr_*S_N^+)$ are pairwise non-isomorphic.

\vspace{0.2cm}

\noindent The paper is organized as follows. After the introduction in Section 1, we have included preliminaries in Section 2 in which we detail all our notations concerning operator algebras and quantum groups. We also prove in this preliminary Section that the von Neumann algebra $\Linf(G)$ of a compact quantum group $G$ is diffuse if and only if it is infinite-dimensional. This result, which will be useful in the study of von Neumann algebras of free wreath product, also solves a question left open in \cite{BCF20} in which it is mentioned that it is unknown if the von Neumann algebra of a quantum reflection group $\Z_s\wr_* S_N^+$ is diffuse when $N\leq 7$. Section $2$ also contains important remarks on free product quantum groups, semi-direct product quantum groups, quantum permutation groups and free wreath product quantum groups, as well as amalgamated free wreath products (in the sense of Freslon \cite{Fr22}). In particular, we identify the amalgamated free wreath product $\wreathH$ at $N=2$ with a semi-direct product. In Section $3$ we explain how operator algebras of amalgamated free wreath products can be described as fundamental algebras of graphs of operator algebras. Section $4$ contains the proof of the amalgamated version of Theorem \ref{THMA}, Section $5$ contains the proof of the amalgamated versions of Theorem \ref{THMB} and \ref{THMC}. In Section $6$, we give a formula describing an amalgamated free wreath product of a fundamental quantum group of a graph of quantum groups by $S_N^+$ as the fundamental quantum group of a graph of quantum groups given by free wreath products. Finally, section $7$ contains the proof of Theorem \ref{THMD} and some other K-theory computations.
\section{Preliminaries}

\subsection{Notations}

All Hilbert spaces, C*-algebras and preduals of von Neumann algebras considered in this paper are assumed to be separable. The inner product on an Hilbert space $H$ is always linear on the right. The C*-algebra of bounded linear maps from $H$ to $H$ is denoted by $\mathcal{B}(H)$ and, given vectors $\xi,\eta\in H$ we denote by $\omega_{\eta,\xi}\in\mathcal{B}(H)^*$ the bounded linear form on $\mathcal{B}(H)$ defined by $(x\mapsto\langle\eta,x\xi\rangle)$. For $T\in\mathcal{B}(H)$, we denote by ${\rm Sp}(T)$ the spectrum of $T$ and by ${\rm Sp}_p(T)$ the point spectrum of $T$. The symbol $\ot$ will denote the tensor product of Hilbert spaces, von Neumann algebras as well as the minimal tensor product of C*-algebras. 

\subsection{Full von Neumann algebras}

Let $M$ be a von Neumann algebra with predual $M_*$. We consider on ${\rm Aut}(M)$ the topology of pointwise convergence in $M_*$ i.e. the smallest topology for which the maps ${\rm Aut}(M)\rightarrow M_*$, $\alpha\mapsto\omega\circ\alpha$ are continuous, for all $\omega\in M_*$, where $M_*$ is equipped with the norm topology. It is well known that ${\rm Aut}(M)$ is a Polish group (since $M_*$ is separable). When $\omega\in M_*$ is a faithful normal state, we may consider the closed subgroup ${\rm Aut}(M,\omega)<{\rm Aut}(M)$ of automorphisms preserving $\omega$. Note that the induced topology from ${\rm Aut}(M)$ to ${\rm Aut}(M,\omega)$ is the smallest topology making the maps ${\rm Aut}(M,\omega)\rightarrow{\rm L}^2(M,\omega)$,  $\alpha\mapsto \alpha(a)\xi_\omega$ continuous, for the norm on ${\rm L}^2(M,\omega)$, where $\xi_\omega$ is the cyclic vector in the GNS construction ${\rm L}^2(M,\omega)$ of $\omega$. We record the following Remark for later use.

\begin{remark}\label{RmkRestriction}
Note that if $(M,\omega)$ is a von Neumann algebra with a faithful normal state $\omega$ and $A\subset M$ is a von Neumann subalgebra with a $\omega$-preserving normal conditional expectation $M\rightarrow A$ then, the subgroup ${\rm Aut}_A(M,\omega)=\{\alpha\in{\rm Aut}(M,\omega)\,:\,\alpha(A)=A\}<{\rm Aut}(M,\omega)$ is closed and the restriction map ${\rm Aut}_A(M,\omega)\rightarrow{\rm Aut}(A,\omega)$, $\alpha\mapsto\alpha\vert_A$ is continuous.
\end{remark}

\vspace{0.2cm}

\noindent Let ${\rm Inn}(M)\subset{\rm Aut}(M)$ be the normal subgroup of inner automorphisms and ${\rm Out}(M):={\rm Aut}(M)/{\rm Inn}(M)$ be the quotient group. Equipped with the quotient topology, ${\rm Out}(M)$ is a topological group and it is Hausdorff if and only if ${\rm Inn}(M)\subset{\rm Aut}(M)$ is closed. In that case (since $M_*$ is assumed to be separable) ${\rm Out}(M):={\rm Aut}(M)/{\rm Inn}(M)$ is actually a metrizable topological group so the convergence of sequences in ${\rm Out}(M)$ completely characterises the topology. Following Connes \cite{Co74}, a von Neumann algebra is called \textit{full} if ${\rm Inn}(M)\subset{\rm Aut}(M)$ is closed. A normal faithful semi-finite weight $\varphi$ on $M$ is called \textit{almost-periodic} if its modular operator $\nabla_\varphi$ has pure point spectrum. Connes defines \cite{Co74} the invariant $Sd(M)$ of a full von Neumann algebra as the intersection of point spectrum of $\nabla_\varphi$ for $\varphi$ a normal faithful semi-finite almost-periodic weight and shows that the closure of $Sd(M)$ is the $S$-invariant $S(M)$.

\vspace{0.2cm}

\noindent The famous noncommutative Radon-Nykodym Theorem of Connes \cite{Co73} shows that, for any pair of normal faithful states (actually semi-finite weights) $\omega_1$ and $\omega_2$ on $M$, their modular groups $\sigma_t^{\omega_1}$ and $\sigma_t^{\omega_2}$ have the same image in ${\rm Out}(M)$, for all $t\in\R$. Hence, there is a well defined homomorphism $\delta\,:\,\R\rightarrow{\rm Out}(M)$, called the \textit{modular homomorphism}, defined by $\delta(t)=[\sigma^\omega_t]$, where $\omega$ is any normal faithful state on $M$ and $[\cdot]$ denotes the class in ${\rm Out}(M)$.

\vspace{0.2cm}

\noindent Connes defines in \cite{Co74}, the \textit{invariant} $\tau(M)$ of a full von Neumann algebra $M$ as the smallest topology on $\R$ for which the modular homomorphism $\delta$ is continuous. When $M$ is full (and $M_*$ is separable), ${\rm Out}(M)$ is clearly a metrizable topological group hence, $(\R,\tau(M))$ is also a metrizable topological group. In particular, the topology $\tau(M)$ is completely characterized by the knowledge of which sequences are converging to zero.

\subsection{Compact quantum groups}\label{SectionCQG}

For a discrete group $\Gamma$ we denote by $C^*(\Gamma)$ its full C*-algebra, $C^*_r(\Gamma)$ its reduced C*-algebra and ${\rm L}(\Gamma)$ its von Neumann algebra. We briefly recall below some elements of the compact quantum group (CQG) theory developed by Woronowicz \cite{wor87,wor88,wor98}.

\vspace{0.2cm}

\noindent For a CQG $G$, we denote by $C(G)$ its \textit{maximal} C*-algebra, which is the enveloping C*-algebra of the unital $*$-algebra $\Pol(G)$ given by the linear span of coefficients of irreducible unitary representations of $G$. The set of equivalence classes of irreducible unitary representations will be denoted by $\Irr(G)$. We will denote by $\varepsilon_G\,:\, C(G)\rightarrow\C$ the counit of $G$ which satisfies ($\id\ot\varepsilon_G)(u)={\rm \id}_H$ for all finite dimensional unitary representations $u\in\mathcal{B}(H)\ot C(G)$.

\vspace{0.2cm}

\noindent Let us recall below the modular ingredients of a CQG. Let us fix a complete set of representatives $u^x\in\mathcal{B}(H_x)\ot C(G)$ for $x\in \Irr(G)$. It is known that for any $x\in\Irr(G)$ there exists a unique $\overline{x}\in\Irr(G)$, called the contragredient of $x$, such that $\Mor(1,x\ot \overline{x})\neq\{0\}$ and $\Mor(1,\overline{x}\ot x)\neq\{0\}$, where $1$ denotes the trivial representation of $G$. Both the spaces $\Mor(1,\overline{x}\ot x)$ and $\Mor(1,x\ot \overline{x})$ are actually one-dimensional. Fix non-zero vectors $s_x\in H_x\otimes H_{\overline{x}}$ and $s_{\overline{x}}\in H_{\overline{x}}\ot H_x$ such that $s_x\in \Mor(1,x\ot \overline{x})$ and $s_{\overline{x}}\in \Mor(1,\overline{x}\ot x)$. Let $J_x\,:\, H_x\rightarrow H_{\overline{x}}$ be the unique invertible antilinear map satisfying $\langle J_x\xi,\eta\rangle=\langle s_x,\xi\ot\eta\rangle$ for all $\xi\in H_x$, $\eta\in H_{\overline{x}}$ and define the positive invertible operator $Q_x=J_x^*J_x\in\mathcal{B}(H_x)$. Then, there exists a unique normalization of $s_x$ and $s_{\overline{x}}$ such that $\Vert s_x\Vert=\Vert s_{\overline{x}}\Vert$ and $J_{\overline{x}}=J_x^{-1}$. With this normalization, $Q_x$ is uniquely determined, we do have ${\rm Tr}(Q_x)={\rm Tr}(Q_x^{-1})=\Vert s_x\Vert^2$, $Q_{\overline{x}}=(J_xJ_x^*)^{-1}$ and ${\rm Sp}(Q_{\overline{x}})={\rm Sp}(Q_x)^{-1}$ (${\rm Tr}$ is the unique trace on $\mathcal{B}(H_x)$ such that ${\rm Tr}(1)={\rm dim}(H_x)$). The number ${\rm Tr}(Q_x)$ is called the \textit{quantum dimension} of $x$ and is denoted by ${\rm dim}_q(x)$. From the orthogonality relations:
$$(\id\ot h_G)((u^x)^*(\xi\eta^*\ot 1)u^y)=\frac{\delta_{x,y}1}{{\rm dim}_q(x)}\langle\eta,Q_x^{-1}\xi\rangle$$
it is not difficult to check that, with ${\rm L}^2(G)=\bigoplus_{x\in\Irr(G)}H_x\ot H_{\overline{x}}$, $\xi_G:=1\in H_1\ot H_{\overline{1}}=\C$ and $\lambda_G\,:\, C(G)\rightarrow\mathcal{B}({\rm L}^2(G))$ the unique unital $*$-homomorphism such that
$$\lambda_G((\omega_{\eta,\xi}\ot\id)(u^x))\xi_G={\rm dim}_q(x)^{-\frac{1}{2}}\xi\ot J_{\overline{x}}^*\eta,$$
$\xi_G$ is $\lambda_G$-cyclic and $h_G(x)=\langle\xi_G,\lambda_G(a)\xi_G\rangle$ $\forall a\in C(G)$. Hence, the triple $({\rm L}^2(G),\lambda_G,\xi_G)$ is an explicit GNS construction for the Haar state $h_G$ on $C(G)$. 

\vspace{0.2cm}

\noindent Let $C_r(G)=\lambda_G(C(G))\subset\mathcal{B}({\rm L}^2(G))$ be the \textit{reduced C*-algebra} of $G$. The surjective unital $*$-homomorphism $\lambda_G\,:\,C(G)\rightarrow C_r(G)$ is called the canonical surjection. Recall that $G$ is called \textit{co-amenable} whenever $\lambda_G$ is injective. The von Neumann algebra of $G$ is denoted by $\Linf(G):=C_r(G)''\subset\mathcal{B}({\rm L}^2(G))$. We will still denote by $h_G$ the Haar state of $G$ on the C*-algebra $C_r(G)$ as well as on the von Neumann algebra ${\rm L}^\infty(G)$ i.e. $h_G=\langle\xi_G,\cdot\xi_G\rangle$ when viewed as a state on $C_r(G)$ or a normal state on $\Linf(G)$. We recall that $h_G$ is faithful on both $C_r(G)$ and $\Linf(G)$. With the explicit GNS construction of $h_G$ given above, it is not difficult to compute the modular ingredients of the normal faithful state $h_G$ on $\Linf(G)$ and we find that the closure $S_G$ of the antilinear operator $x\xi_G\mapsto x^*\xi_G$ has a polar decomposition $S_G=J_G\nabla_G^{\frac{1}{2}}$, where $\nabla_G$ is the positive operator on ${\rm L}^2(G)$ given by $\nabla_G:=\bigoplus_{x\in\Irr(G)} Q_x\ot Q_{\overline{x}}^{-1}$. Hence, the Haar state of a CQG is always almost-periodic and its modular group $(\sigma_t^G)_t$ is the unique one-parameter group of $\Linf(G)$ such that $(\id\ot\sigma_t^G)(u^x)=(Q_x^{it}\ot 1)u^x(Q_x^{it}\ot 1)$. Let us recall that $G$ is said to be of \textit{Kac type} whenever $h_G$ is a trace and it happens if and only if $Q_x=1$ for all $x\in\Irr(G)$. Let us also recall that the \textit{scaling group} $(\tau_t^G)_t$ is the one-parameter group of $\Linf(G)$ given by by the formula $(\id\ot\tau_t^G)(u^x)=(Q_x^{it}\ot 1)u^x(Q_x^{-it}\ot 1)$. It is well known that the scaling group of $G$ is the unique one parameter group $(\tau_t^G)_t$ of $\Linf(G)$ such that $\Delta\circ\sigma_t^G=(\tau_t^G\ot\sigma_t^G)\circ\Delta$ $\forall t\in\R$. Moreover, the scaling group preserves that Haar state: $h_G\circ\tau_t^G=h_G$ (this means that the \textit{scaling constant} of a compact quantum group is $1$).

\noindent Let us also recall the definition of the $T$\textit{-invariant} of a CQG $G$, introduced by S. Vaes in \cite{Va05}:
$$T(G):=\{t\in\R\,:\,\exists u\in\mathcal{U}(\Linf(G))\,:\,\Delta(u)=u\ot u\text{ and }\tau_t^G={\rm Ad}(u)\},$$ where ${\rm Ad}(u)$ is the automorphism of $\Linf(G)$ given by $a\mapsto uau^*$.

\vspace{0.2cm}

\noindent The following result is well known \cite[Proposition 2.2]{Fi10}.

\begin{proposition}
The set ${\rm Mod}(G):=\bigcup_{x\in\Irr(G)}{\rm Sp}(Q_x)$ is a subgroup of $\R^*_+$.
\end{proposition}

\begin{proof}
We already remarked in Section \ref{SectionCQG} that ${\rm Sp}(Q_{\overline{x}})=({\rm Sp}(Q_x))^{-1}$. Hence ${\rm Mod}(G)$ is stable by inverse. The fact that ${\rm Mod}(G)$ is stable by product follows from the relation
\begin{equation*}
SQ_z=(Q_x\ot Q_y)S,
\end{equation*}
for any $x,y,z\in\Irr(G)$, and isometry $S\in\Mor(z,x\ot y)$.\end{proof}

\begin{remark}
For convenience, we will use the following non-standard notations.
\begin{enumerate}
    \item Let $Sd(G):={\rm Sp}_p(\nabla_G)$. From the explicit computation of $\nabla_G$ and since ${\rm Sp}(Q_{\overline{x}})={\rm Sp}(Q_x)^{-1}$ one has $Sd(G)=\cup_{x\in\Irr(G)}{\rm Sp}(Q_x)^2\subset{\rm Mod}(G)$ and $Sd(G)^{-1}=Sd(G)$.
    \item We denote by $\tau(G)$ the smallest topology on $\R$ such that the map $(t\mapsto\sigma^G_t)$ is continuous. Hence $\tau(G)$ is smaller than the usual topology on $\R$. It is not difficult to check that it is the smallest topology on $\R$ for which the maps $f_\lambda\,:\,\R\rightarrow\mathbb{S}^1\,:\,(t\mapsto\lambda^{it})$ are continuous, for all $\lambda\in Sd(G)$, where $\mathbb{S}^1$ has the usual topology. Note also that, for any topology $\tau$ on $\R$, the set of $\lambda>0$ for which $f_\lambda$ is $\tau$-continuous is a closed subgroup of $\R^*_+$ (for the usual topology on $\R^*_+$). Hence, $\tau(G)$ is also the smallest topology on $\R$ for which the maps $f_\lambda$ are continuous, for all $\lambda\in\overline{\langle Sd(G)\rangle}$, the closed subgroup of $\R_+^*$ generated by $Sd(G)$ (for the usual topology on $\R_+^*$). Moreover, the following is easy to check:
    \begin{itemize}
        \item $G$ is Kac $\Leftrightarrow Sd(G)=\{1\}\Leftrightarrow\tau(G)$ is the trivial topology. If $Sd(G)\neq\{1\}$ then $(\R,{\tau}(G))$ is a metrizable topological group so the topology $\tau(G)$ is completely characterized by the knowledge of which sequences are converging to zero.
        \item $\overline{\langle Sd(G)\rangle}=\R^*_+$ if and only if $\tau(G)$ is the usual topology on $\R$.
    \end{itemize}
\end{enumerate}
\end{remark}

\noindent Recall that a von Neumann algebra is diffuse when it has no non-zero minimal projection. We will use the following simple Lemma. While the arguments for its proof are already present in the literature, the precise statement seems to be unknown.

\begin{lemma}\label{LemmaInfinite}
Let $G$ be a compact quantum group. The following are equivalent.
\begin{enumerate}
    \item $\Irr(G)$ is infinite.
    \item $\Linf(G)$ is diffuse.
    \item $C(G)$ is infinite-dimensional.
\end{enumerate}

\end{lemma}

\begin{proof}
The implication $(2)\Rightarrow(1)$ is obvious. Let us show $(1)\Rightarrow(2)$. By the general theory, a von Neumann algebra is diffuse if and only if it has no direct summand of the form $\mathcal{B}(H)$. When $H$ is finite dimensional, we may apply \cite[Theorem 3.4]{CKSS16} (since the action of $G$ on $\Linf(G)$ given by the comultiplication is ergodic) to deduce that $\mathcal{B}(H)$ is not a direct summand of $\Linf(G)$ whenever $\Irr(G)$ is infinite. When $H$ is infinite-dimensional, we may apply \cite[Theorem 6.1]{KS22} to deduce that $\mathcal{B}(H)$ is not a direct summand of $\Linf(G)$ whenever $\Irr(G)$ is infinite. The equivalence between $(1)$ and $(3)$ is clear.\end{proof}

\noindent Let us now recall the notion of a dual quantum subgroup. We are grateful to Kenny De Commer for explaining to us the main argument of the proof of the next Proposition and to Makoto Yamashita for showing to us the reference \cite{Chi14}. Recall that $C_\bullet(G)$ denotes either the reduced or the maximal C*-algebra of $G$. 

\begin{proposition}\label{dualqsg}
Let $G$ and $H$ be CQG. The following data are equivalent.
\begin{itemize}
    \item $\iota\,:\, C(H)\rightarrow C(G)$ is a faithful unital $*$-homomorphism intertwining the comultiplications.
    \item $\iota\,:\,\Pol(H)\rightarrow\Pol(G)$ is a faithful unital $*$-homomorphism intertwining the comultiplications.
    \item $\iota\,:\,C_r(H)\rightarrow C_r(G)$ is a faithful unital $*$-homomorphism intertwining the comultiplications.
\end{itemize} If one of the following equivalent conditions is satisfied, we view $\Pol(H)\subset\Pol(G)$, $C(H)\subset C(G)$ and $C_r(H)\subset C_r(G)$. Then, the unique linear map $E\,:\,\Pol(G)\rightarrow \Pol(H)$ such that
$$(\id\ot E)(u^x)=\left\{\begin{array}{lcl}u^x&\text{if}&x\in\Irr(H),\\0&\text{if}&x\in\Irr(G)\setminus\Irr(H).\end{array}\right.$$
has a unique ucp extension to a map $E_\bullet\,:\,C_\bullet(G)\rightarrow C_\bullet(H)$ which is a Haar-state-preserving conditional expectation onto the subalgebra $C_\bullet(H)\subset C_\bullet(G)$. At the reduced level, $E_r$ is faithful and extends to a Haar-state preserving normal faithful conditional expectation $\Linf(G)\rightarrow\Linf(H)$.
\end{proposition}

\begin{proof}
If $\iota\,:\, C(H)\rightarrow C(G)$ is a faithful unital $*$-homomorphism intertwining the comultiplications then it is clear that its restriction to $\Pol(H)$ has image in $\Pol(G)$ and is still faithful. If now $\iota\,:\,\Pol(H)\rightarrow \Pol(G)$ is defined at the algebraic level then, since it is faithful and intertwines the comultiplications, it also intertwines the Haar states and so extends to a faithful unital $*$-homomorphism $\iota\,:\, C_r(H)\rightarrow C_r(G)$ which is easily seen to intertwine the comultiplications. It is proved in \cite{Ve04} that $E$ extends to a faithful and Haar state preserving conditional expectation at the reduced level as well as at the von Neumann level (which is moreover normal).  Also, if $\iota\,:\, C_r(H)\rightarrow C_r(G)$ is defined at the reduced level, its restriction to $\Pol(H)$ satisfies the second condition. Hence, it suffices to check that if $\Pol(H)\subset\Pol(G)$ is a unital $*$-subalgebra with the inclusion intertwining the comultiplications then, the canonical extension $\iota\,:\, C(H)\rightarrow C(G)$ of the inclusion, which obviously also intertwines the comultiplications, is still faithful. Note that it suffices to show that $E$ extends to a ucp map $E\,:\,C(G)\rightarrow C(H)$. Indeed, if we have such an extension then $E$ has norm $1$ and $E\circ\iota=\id_{C(H)}$, so for all $a\in Pol(H)$ one has $\Vert a\Vert_{C(H)}=\Vert E(\iota(a))\Vert_{C(H)}\leq\Vert \iota(a)\Vert_{C(G)}\leq\Vert a\Vert_{C(H)}$. Then, $\iota$ is an isometry hence faithful. The fact that $E$ extends to a ucp map is proved in \cite[Theorem 3.1]{Chi14}. It is clear that the ucp extension preserves the Haar states since it already does at the algebraic level (by definition of $E$). Now, viewing $C(H)\subset C(G)$, $E$ is ucp and is the identity on $C(H)$ hence, it is a conditional expectation onto $C(H)$.
\end{proof}

\noindent If one of the above equivalent condition is satisfied, we say that $H$ is a \textit{dual quantum subgroup of} $G$ and we will view ${\rm Pol}(H)\subset{\rm Pol}(G)$, $C(H)\subset C(G)$, $C_r(H)\subset C_r(G)$ as well as ${\rm L}^\infty(H)\subset{\rm L}^\infty(G)$. Let us note that, the ucp extension of $E$ at the maximal level is not, in general, faithful and not even GNS-faithful (meaning that the GNS representation morphism may be non injective). We will usually denote $E$ at the algebraic, full, reduced and von Neumann algebraic level by the same symbol $E_H$.

\begin{remark}\label{RmkECH} Let $C(H)\subset C(G)$ be a dual quantum subgroup and define $\Pol(G)^\circ:=\{a\in \Pol(G)\,:\, E_H(a)=0\}$ and $C_\bullet(G)^\circ:=\{a\in C_\bullet(G)^\circ\,:\, E_\bullet(a)=0\}$. Then $\Pol(G)^\circ$ is the linear span of coefficient of irreducible representations $x\in\Irr(G)\setminus\Irr(H)$ , $C_\bullet(G)^\circ$ is the closure in $C_\bullet(G)$ of $\Pol(G)^\circ$ and $\Delta(\Pol(G)^\circ)\subset\Pol(G)^\circ\otimes\Pol(G)^\circ$. All this statements easily follow from the property of $E_H$ stated in the previous proposition.
\end{remark}

\noindent Let us recall that if $\Linf(H)\subset\Linf(G)$ is a dual compact quantum subgroup then $\sigma_t^G\vert_{\Linf(H)}=\sigma_t^H$ and $\tau_t^G\vert_{\Linf(H)}=\tau_t^H$ for all $t\in\R$.

\vspace{0.2cm}

\begin{definition}\label{index}
Following Vergnioux \cite{Ve04}, given a dual quantum subgroup $C(H)\subset C(G)$, we introduce an equivalence relation $\sim_H$ on $\Irr(G)$ by defining $x\sim_H y\Leftrightarrow \exists s\in\Irr(H),\,\Mor(s,\overline{x}\ot y)\neq\{0\}$. Note in particular that $x\nsim_H y\Leftrightarrow (\id\ot E_H)(u^{\overline{x}}\ot u^{y})=0$. We define the \textit{index of $H$ in $G$} by the number of equivalence classes $[G:H]:=\left\vert\Irr(G)/\sim_H\right\vert$.
\end{definition} 

\subsection{Free product quantum group}

We now recall some well known results about free product quantum groups. Given two compact quantum groups $G_1$ and $G_2$ the universal property of the C*-algebra given by the full free product $C(G):=C(G_1)*C(G_2)$ allows to define the unique unital $*$-homomorphism $\Delta\,:\, C(G)\rightarrow C(G)\otimes C(G)$ such that $\Delta\vert_{C(G_k)}=\Delta_{G_k}$ for $k=1,2$. It is easy to check that $G:=(C(G),\Delta)$ is a CQG with maximal C*-algebra $C(G)$, reduced C*-algebra given by the reduced free product with respect to the Haar states $C_r(G)=(C_r(G_1),h_1)*(C_r(G_2),h_2)$, and $\Linf(G)=(\Linf(G_1),h_1)*(\Linf(G_2),h_2)$. Moreover the Haar state on $C_r(G)$ is the free product state $h=h_1*h_2$.

\vspace{0.2cm}

\noindent We collect below some important remarks about free products. Most of them are well known to specialists. Since we could not find any explicit statements in the literature, we include a complete proof.

\begin{proposition}\label{Prop-Freeprod}
Let $G_1,G_2$ be non-trivial CQG. The scaling group of $G:=G_1*G_2$ is the free product $\tau_t^G=\tau_t^{G_1}*\tau_t^{G_2}$ and the Vaes' $T$-invariant of $G_1*G_2$ is given by:
$$T(G_1*G_2)=\{t\in\R\,:\,\tau_t^{G_1}=\tau_t^{G_2}=\id\}.$$
The following are equivalent.
\begin{enumerate}
    \item $G_1*G_2$ is co-amenable.
    \item $\vert\Irr(G_1)\vert=2=\vert\Irr(G_2)\vert$.
    \item $G_1\simeq G_2\simeq \Z/2\Z$.
    \item ${\rm L}^\infty(G_1*G_2)$ is an amenable von Neumann algebra.
\end{enumerate}
Moreover, if one of the previous equivalent conditions does not hold then ${\rm L}^\infty(G_1*G_2)$ is a full and prime factor without any Cartan subalgebras and:
\begin{itemize}
    \item If $G_1$ and $G_2$ are Kac then ${\rm L}^\infty(G_1*G_2)$ is a type ${\rm II}_1$ factor.
    \item If $G_1$ or $G_2$ is not Kac then ${\rm L}^\infty(G_1*G_2)$ is a type ${\rm III}_\lambda$ factor with $\lambda\neq 0$ and its Connes' $T$-invariant is given by $\{t\in\R\,:\,\sigma_t^{h_1}=\id=\sigma_t^{h_2}\}$.
    \item$Sd(\Linf(G_1*G_2))=\langle Sd(G_1),Sd(G_2)\rangle$ and $\tau(\Linf(G_1*G_2))=\langle\tau(G_1),\tau(G_2)\rangle$, where $\langle\,\cdot\,\rangle$ means either the group generated by or the topology generated by.
\end{itemize}
\end{proposition}

\begin{proof}Since $h=h_1*h_2$, we have $\sigma_t=\sigma_t^{G_1}*\sigma_t^{G_2}$, where $\sigma_t$ denotes the modular group of $h$. Since $\tau_t^{G_k}$ is $h_{k}$-invariant, the free product $\tau_t:=\tau_t^{G_1}*\tau_t^{G_2}$ makes sense and it defines a one parameter group of $\Linf(G_1*G_2)$. To show that it is the scaling group, we only have to check that $\Delta\circ\sigma_t=(\tau_t\ot\sigma_t)\circ\Delta$, which is clear. Let $T':=\{t\in\R\,:\,\tau_t^{G_1}=\tau_t^{G_2}=\id\}$. It is clear that $T'\subseteq T(G_1*G_2)$. Let $t\in T(G_1*G_2)$ so that there exists $u\in\Linf(G_1*G_2)$ a unitary such that $\Delta(u)=u\otimes u$ and $\tau_t={\rm Ad}(u)$. It follows that $u$ is a dimension $1$ unitary representation of $G_1*G_2$, hence irreducible. If $u$ is non-trivial, it follows from the classification of irreducible representations of $G_1*G_2$ \cite{Wa95} that $u$ is a product of non-trivial dimension $1$ unitary representations alternating from $\Irr(G_1)$ and $\Irr(G_2)$ i.e. $u=u_1u_2\dots u_n$, where $u_k$ is a unitary in $\Linf(G_{i_k})$ with $\Delta_{G_{i_k}}(u_k)=u_k\ot u_k$, $h_k(u_k)=0$ and $i_k\neq i_{k+1}$ for all $k$. Let $l\in\{1,2\}\setminus\{i_n\}$ and note that, since $G_l$ is non-trivial, there exists a non zero $x\in\Linf(G_l)$ such that $h_l(x)=0$. Then $uxu^*\in\Linf(G_1)*\Linf(G_2)$ is a reduced operator so $E_l(uxu^*)=0$, where $E_l\,:\,\Linf(G_1)*\Linf(G_2)\rightarrow\Linf(G_l)$ denotes the canonical Haar-state-preserving normal and faithful conditional expectation. However, since $\tau_t(x)=\tau_t^{G_l}(x)\in\Linf(G_l)$ one has $\tau_t(x)=E_{l}(\tau_t(x))=E_l(uxu^*)=0$ hence $x=0$, leading to a contradiction. It follows that such a $u$ is always trivial and $T(G_1*G_2)\subseteq T'$.

\vspace{0.2cm}

\noindent$(3)\Rightarrow(2)\Rightarrow(1)\Rightarrow(4)$ are obvious. Also $(2)\Rightarrow(3)$ is easy and well known. Let us repeat however the argument here for the convenience of the reader. Let $G$ be a CQG satisfying $\vert \Irr(G)\vert=2$ and let $u$ be the unique, up to unitary equivalence, non-trivial irreducible representation of $G$ and write $1$ for the trivial representation. Since $u$ is non-trivial, $\overline{u}$ also is hence $\overline{u}\simeq u$. It follows that ${\rm dim}(1,u\ot u)=1$. Hence, $u\ot u=1\oplus du$, where $d={\rm dim}(u,u\ot u)\in\N$. Let us denote by $N\in\N^*$ the dimension of $u$ so that we have $N^2=1+dN$ hence $1\equiv 0 \pmod N$ which implies that $N=1$ and then $d=0$. Since $u$ is of dimension $1$, $u\in C(G)$ is a unitary such that $\Delta(u)=u\otimes u$, $u=u^*$ (hence $u^2=1$) and $C(G)$ is generated by $u$ so $G=\Z/2\Z$.

\vspace{0.2cm}

\noindent Suppose that $(2)$ does not hold so that ${\rm dim}(\Linf(G_1))+{\rm dim}(\Linf(G_2))\geq 5$. It follows from \cite[Theorem 4.1]{Ue10} that there exists a central projection $z$ in ${\rm L}^\infty(G_1*G_2)=({\rm L}^\infty(G_1),h_1)*({\rm L}^\infty(G_2),h_2)$ such that $z\Linf(G_1*G_2)$ is either a full factor of type ${\rm II}_1$ or a full factor of type ${\rm III}_\lambda$, $\lambda\neq 0$, with $T$-invariant given by $\{t\in\R\,:\,\sigma_t^{h_1}=\sigma_t^{h_2}\}$ and $(1-z)\Linf(G_1*G_2)$ is a direct sum of matrix algebras. Hence,  $z\Linf(G_1*G_2)$ is non-amenable (since it is full and not of type ${\rm I}$) so $\Linf(G_1*G_2)$ is non-amenable either and it shows $(4)\Rightarrow(2)$. Moreover, we know from \cite{Wa95} that the set $\Irr(G_1*G_2)$ is infinite hence, by Lemma \ref{LemmaInfinite}, $\Linf(G_1*G_2)$ is diffuse so $z=1$. Both the primeness and absence of Cartan follow now from \cite[Corollary 4.3]{Ue10}. Finally, the $Sd$ and $\tau$ invariants are computed in \cite[Corollary 2.3]{Ue11} and \cite[Theorem 3.2]{Ue11} respectively (recall that our von Neumann algebras are supposed to have separable preduals and that the Haar states on CQG are all almost periodic).\end{proof}

\begin{remark}It is known that, for $G$ a CQG of Kac type, the co-amenability of $G$ is equivalent to the injectivity of the von Neumann algebra $\Linf(G)$ (see Corollary 3.17 in \cite{To06} and the discussion after its proof). However, the equivalence for general CQG is open. Proposition \ref{Prop-Freeprod} shows that in the class of CQG which are nontrivial free products the equivalence between co-amenability and injectivity of the von Neumann algebra is true.\end{remark}

\noindent Let us now recall the amalgamated free product construction \cite{Wa95,Ve04}. Let $G_1,G_2$ be two CQG and $C(H)\subset C(G_k)$ a dual quantum subgroup of both $G_1,G_2$. Let $E_k\,:\, C_r(G_k)\rightarrow C_r(H)$ be the faithful CE. The amalgamated free product is introduced in \cite{Wa95} and its Haar state and reduced C*-algebra is understood in \cite{Ve04}. Following \cite{Wa95}, let us define $C(G):=C(G_1)\underset{C(H)}{*}C(G_2)$ the full amalgamated free product. By universal property there exists a unique unital $*$-homomorphism $\Delta\,:\, C(G)\rightarrow C(G)\ot C(G)$ such that $\Delta\vert_{C(G_k)}=\Delta_{G_k}$ for $k=1,2$ and it is easy to check \cite{Wa95} that the pair $(C(G),\Delta)$ is a compact quantum group, denoted by $G=G_1\underset{H}{*}G_2$. It is shown in \cite{Ve04} that the reduced C*-algebra is the reduced amalgamated free product with respect to the CE $E_k$, $C_r(G)=(C_r(G_1),E_1)\underset{C_r(H)}{*}(C_r(G_2),E_2)$ and the Haar state of $G$ is the free product state $h_{G_1}*h_{G_2}$. To study further amalgamated free products, we will need the following lemma. Let us introduce before some terminology. A unitary representation $u$ of a CQG $G$ is called a \textit{Haar representation} if, for all $k\in\Z^*$ one has $\Mor(1,u^{\ot k})=\{0\}$, where $1$ denotes the trivial representation and, for $k\geq 1$, we define $u^{\ot -k}:=\overline{u}^{\ot k}$. Two unitary representation $u_1,u_2$ of $G$ are called \textit{free} if, for $l\geq 1$, any $(i_1,\dots,i_l)\in\{1,2\}^l$ such that $i_s\neq i_{s+1}$ for all $s$ and any $k_1,\dots,k_l\in\Z^*$, one has $\Mor(1,u_{i_1}^{\ot k_1}\ot u_{i_2}^{\ot k_2}\ot\dots\ot u_{i_l}^{\ot k_l})=\{0\}$.

\begin{lemma}\label{LemAmalgamated}
Let $G$ be a compact quantum group with two Haar representations which are free then there exists $N\geq 1$ such that ${\rm L}(\mathbb{F}_2)\subset M_N(\C)\ot\Linf(G)$ (with a state preserving inclusion). In particular, if $G$ is Kac then $\Linf(G)$ is not amenable. 
\end{lemma}

\begin{proof}
Recall that, for $u\in\mathcal{B}(H)\ot\Linf(G)$ a finite dimensional unitary representation, its contragredient unitary representation $\overline{u}\in\mathcal{B}(\overline{H})\ot\Linf(G)$ satisfies the following: there exists an invertible operator $Q\in\mathcal{B}(\overline{H})$ and an orthonormal basis $(e_i)_i$ of $H$ such that, writing $u=\sum_{ij}e_{ij}\ot u_{ij}$, where $(e_{ij})_{ij}$ are the matrix units associated to $(e_i)_i$, then $u^c:=(Q\ot 1)\overline{u}(Q^{-1}\ot 1)=\sum_{ij}  e'_{ij}\ot u_{ij}^*$, where $e_{ij}'$ are the matrix units associated to $(\overline{e}_i)_i$. Note also that $(u^c)^{\ot k}:=(u^c)_{1,k+1}(u^c)_{2,k+1}\dots(u^c)_{k,k+1}=(Q^{\ot k}\ot 1)\overline{u}^{\ot k}((Q^{-1})^{\ot k}\ot 1)\in\mathcal{B}(\overline{H}^{\ot k})\ot\Linf(G)$ for all $k\geq 1$. Hence, if $\Mor(1,\overline{u}^{\ot k})=\{0\}$ one has $(\id\ot h)((u^{c})^{\ot k})=Q^{\ot k}(\id\ot h)(\overline{u}^{\ot k})(Q^{-1})^{\ot k}=0$, since $(\id\ot h)(\overline{u}^{\ot k})\in\mathcal{B}(\overline{H}^{\ot k})$ is the orthogonal projection onto $\Mor(1,\overline{u}^{\ot k})$. It then follows that $h(x)=0$ for any coefficient $x$ of $(u^c)^{\ot k}$. Since the coefficients of $(u^c)^{\ot k}$ are exactly the products of $k$ adjoints of coefficients of $u$, we deduce that $h(x)=0$ for any product of $k$ adjoints of coefficients of $u$ whenever $\Mor(1,\overline{u}^{\ot k})=\{0\}$. We will use this remark in the rest of the proof.

\vspace{0.2cm}

\noindent Let $u_k\in\mathcal{B}(H_k)\ot\Linf(G)$, $k=1,2$, be two free Haar representations and define $v_1:=(u_1)_{13}$ and $v_2:=(u_2)_{23}$ which are both unitary in $\mathcal{B}(H_1\ot H_2)\ot\Linf(G)$ with respect to $\omega$. Consider the faithful normal state $\omega={\rm Tr}\ot h\in\left(\mathcal{B}(H_1\ot H_2)\ot\Linf(G)\right)_*$. Let us show that $v_1$ and $v_2$ are two free Haar unitaries. Since $(\id\ot h)(u_i^{\ot k})$ ($i=1,2$) is the orthogonal projection onto $\Mor(1,u_i^{\ot k})=\{0\}$, for $k\in\Z^*$, it follows that $h(x)=0$ whenever $x$ is a product of $\vert k\vert$ coefficients of $u_i$ or a product of $\vert k\vert$ adjoints of coefficients of $u_i$, for all $k\in\Z^*$. Let $\mathcal{C}_i$ be the linear span of products of coefficients of $u_i$ and of products of adjoints of coefficients of $u_i$ so that $\omega(\mathcal{C}_i)=\{0\}$. Since $v_1^k=\sum_{ij}e_{ij}\ot 1\ot x_{ij}$ and $v_2^k=\sum_{ij}1\ot e_{ij}\ot y_{ij}$, where $x_{ij}\in\mathcal{C}_1$, $y_{ij}\in\mathcal{C}_2$, for all $k\in\Z^*$, it follows that $\omega(v_i^k)=({\rm Tr}\ot\id)(\id\ot h)(v_i^k)=0$ for all $k\in\Z^*$, $i\in\{1,2\}$. Hence, both $v_1$ and $v_2$ are Haar unitaries with respect to $\omega$. Since $u_1$ and $u_2$ are free representations, the same argument as before shows that $h(\mathcal{C})=\{0\}$, where $\mathcal{C}$ is the linear span of operators $x\in\Linf(G)$ of the form $x=y_1 \dots y_l$, $l\geq 1$, $y_s$ is a product of $\vert k_s\vert$ coefficients of $u_{i_s}$ if $k_s\geq 1$ or adjoints of coefficients of $u_{i_s}$ if $k_s\leq -1$ with $k_s\in\Z^*$ and $i_s\in\{1,2\}$ such that $i_s\neq i_{s+1}$ for all $s$. Let now $l\geq 1$ and $i_1,\dots i_l\in\{1,2\}$ with $i_s\neq i_{s+1}$ and $k_1,\dots, k_l\in\Z^*$. We can write $v_{i_1}^{k_1}\dots v_{i_l}^{k_l}=\sum_{i,j,k,l}e_{ij}\ot e_{kl}\ot x_{ijkl}$, where $x_{ijkl}\in\mathcal{C}$. It follows that $\omega(v_{i_1}^{k_1}\dots v_{i_l}^{k_l})=({\rm Tr}\ot\id)(\id\ot h)(v_{i_1}^{k_1}\dots v_{i_l}^{k_l})=0$. Hence $v_1$ and $v_2$ are free with respect to $\omega$. It follows that there exists a unique normal faithful $*$-homomorphism ${\rm L}(\mathbb{F}_2)\rightarrow \mathcal{B}(H_1\ot H_2)\ot\Linf(G)$ which maps one generator of $\mathbb{F}_2$ onto $v_1$ and the other onto $v_2$. Hence, if $G$ is Kac, $\mathcal{B}(H_1\ot H_2)\ot\Linf(G)$ is non amenable either and it implies that $\Linf(G)$ is also non amenable.\end{proof}

\begin{definition}\label{DefProper}
A dual quantum subgroup $C(H)\subset C(G)$ is called \textit{proper} if there exists an irreducible representation $a$ of $G$ such that $(\id\ot E_H)(a)=0$ and for any $s\in\Irr(H)$, if $s\subset \overline{a}\ot a$ then $s=1$.
\end{definition}

\noindent Note that if $C(H)\subset C(G)$ is proper then $[G:H]\geq 2$.

\begin{remark}
In the case of duals of discrete groups, the notion of proper dual quantum subgroup coincides with the usual notion of a proper subgroup. However, for quantum groups, there are examples of non proper dual quantum subgroup of index $2$. A nice example is the dual quantum subgroup ${\rm Aut}^+(M_N(\C))$ of $O_N^+$. Indeed, the representation category of $O_N^+$ is the category such that all the irreducible are self-adjoint, indexed by $\N$, $(u_n)_{n\in \N}$, with $u_0 = \varepsilon$ and fusion rules $u_n\ot u_m = u_{\vert n-m\vert}\oplus u_{\vert n-m\vert+2}\oplus \dots \oplus u_{n+m}$. Taking the full subcategory of ${\rm Rep}( O_N^+)$ generated by the irreducible representations $(v_n=u_{2n})$, we get a category isomorphic to ${\rm Rep}({\rm Aut}^+(M_N(\C)))$, with irreducible representations indexed by $\N$ and fusion rules $v_n\ot v_m = u_{2n}\ot u_{2m} = v_{\vert n-m\vert} \oplus v_{\vert n-m\vert+1}\oplus\dots \oplus v_{n+m}$. With the fusion rules, we see that it is indeed a full subcategory and that the index of the corresponding subgroup is $2$. However, taking any irreducible $u_n$ of $O_N^+$, with $n\geq 1$, we have that $v_1 = u_2\subset \overline{u_n}\ot u_n = u_n\ot u_n$ so the dual quantum subgroup cannot be proper because $v_1\in \Irr(\Aut^+(M_N(\C)))$.
\end{remark}

\begin{proposition}\label{PropProper}
Let $G=G_1*_HG_2$ be an amalgamated free product with $H$ a proper dual quantum subgroup of $G_1$ and $[G_2:H]\geq 3$ then there exists $N$ such that ${\rm L}(\mathbb{F}_2)\subset M_N(\C)\ot\Linf(G)$.
\end{proposition}

\begin{proof}
There exist an irreducible representation $a$ of $G_1$ satisfying the conditions of Definition \ref{DefProper} and two irreducible representations $b_1,b_2$ of $G_2$ such that $(\id\ot E_H)(b_i)=0$ and $(\id\ot E_H)(\overline{b_1}\ot b_2)=0$ (so we also have $(\id\ot E_H)(\overline{b_2}\ot b_1)=0$). Define $u_i=a\ot b_i\ot\overline{a}\ot\overline{b}_i$ and let $k\geq 1$. Since $\chi(a)\in\Linf(G_1)^\circ,\chi(b_i)\in{\rm L}^\infty(G_2)^\circ$, $\chi(u_i^{\ot k})=\chi(u_i)^k=\left(\chi(a)\chi(b_i)\chi(a)^*\chi(b_i)^*\right)^{k}$ is a reduced operator in the amalgamated free product so:
$$\dim\left(\Mor(1, u_i^{\ot k})\right)=h(\chi(u_i^{\ot k}))=0.$$
Also, $\chi(\overline{u}_i^{\ot k})=\left(\chi(b_i)\chi(a)\chi(b_i)^*\chi(a)^*\right)^{k}$ is reduced so $\dim\left(\Mor(1, \overline{u_i}^{\ot k})\right)=h(\chi(\overline{u_i}^{\ot k}))=0$. Hence, $u_i$ is a Haar representation of $G$ for $i\in\{1,2\}$. Let us show that $u_1$ and $u_2$  are free. It suffices to show that, for any $l\geq 1$, $(i_1,\dots,i_l)\in\{1,2\}^l$ with $i_s\neq i_{s+1}$ and $k_1,\dots,k_l\in\Z^*$ the operator $x:=\chi(u_{i_1}^{\ot k_{i_1}}\ot\dots\ot u_{i_l}^{\ot k_l})$ is in the linear span of reduced operators. The case $l=1$ is clear by the first part of the proof and the general case can be shown by induction by using the same arguments used in the case $l=2$ that we present below. Let $x=\chi(u_{i_1}^{\ot k_1}\ot u_{i_{2}}^{\ot k_{2}})$. It suffices to show that $x$ is in the linear span of reduced operator. If $k_1$ and $k_2$ have the same sign then $x$ is already reduced. If $k_1\geq 1$ and $k_{2}\leq -1$ then,
$$
x=\chi(u_{i_1})^{k_1}\chi(\overline{u}_{i_{2}})^{- k_{2}}=\chi(u_{i_1})^{k_l-1}\chi(a)\chi(b_{i_1})\chi(a)^*\chi(\overline{b}_{i_1}\ot b_{i_{2}})\chi(a)\chi(b_{i_{2}})^*\chi(a)^*\chi(\overline{u}_{i_{2}})^{-k_{2}-1}$$
Since $(\id\ot E_H)(\overline{b}_{i_1}\ot b_{i_{2}})=0$, we have $\chi(b_{i_1}\ot\overline{b}_{i_{2}})\in\Linf(G_2)^\circ$ hence $x$ is reduced. If $k_1\leq -1$ and $k_2\geq 1$ then
\begin{eqnarray*}
x&=&\chi(\overline{u}_{i_1})^{-k_1-1}\chi(b_{i_1})\chi(a)\chi(b_{i_1})^*\chi(\overline{a}\ot a)\chi(b_{i_2})\chi(a)^*\chi(b_{i_2})^*\chi(
u_{i_2})^{k_2-1}\\
&=&\chi(\overline{u}_{i_1})^{-k_1-1}\chi(b_{i_1})\chi(a)\chi(\overline{b}_{i_1}\ot b_{i_2})\chi(a)^*\chi(b_{i_2})^*\chi(u_{i_2})^{k_2-1}\\
&&+\sum_{s\in\Irr(G_1)\setminus\Irr(H), s\subset\overline{a}\ot a}\chi(\overline{u}_{i_1})^{-k_1-1}\chi(b_{i_1})\chi(a)\chi(b_{i_1})^*\chi(s)\chi(b_{i_2})\chi(a)^*\chi(b_{i_2})^*\chi(u_{i_2})^{k_2-1},
\end{eqnarray*}
The right hand side of this equality is in the linear span of reduced operators since $\chi(\overline{b}_{i_1}\ot b_{i_2})\in\Linf(G_2)^\circ$.\end{proof}

\begin{remark}
The index condition is clearly necessary since it is already necessary in the discrete group case. Indeed, if $\Gamma=\Gamma_1\underset{\Sigma}{*}\Gamma_2$ be a non-trivial amalgamated free product of discrete groups (i.e. $\Sigma\neq\Gamma_k$, $k=1,2$). It is well known and easy to check that $\Gamma$ is amenable if and only if $\Sigma$ is amenable and  $[\Gamma_k:\Sigma]= 2$ for all $k\in\{1,2\}$ (actually $\Gamma$ is an extension of $\Sigma$ by $D_\infty=\Z_2*\Z_2$).
\end{remark}

\begin{example}
The dual quantum subgroup $C({\rm Aut^+}(M_2(\C)))\subset C(O_2^+)$ is not proper, has index $2$ and the quantum group $G:=O_2^+\underset{{\rm Aut}^+(M_2(\C))}{*} O_2^+$ is co-amenable. The inclusion of $C({\rm Aut^+}(M_2(\C)))$ in $C(O_2^+)$ is the map which sends the fundamental representation of $\Aut^+(M_2(\C))$ onto $v\ot v$, where $v\in M_2(\C)\ot C(O_2^+)$ is the fundamental representation of $O_2^+$. Hence, writing $v_l$, $l\in\N$, the representatives of the irreducible representations of $O_N^+$ such that $v_0=1$ and $v_1=v$, $C({\rm Aut^+}(M_2(\C)))$ is viewed in $C(O_2^+)$ has the C*-subalgebra generated by the coefficients of representations $v_{l}$ for $l\in 2\N$. Let $\rho\,:\,C(O_2^+)\rightarrow C^*(\Z_2)$ be the unique unital $*$-homomorphism such that $(\id\ot\rho)(v)=1\ot g$, where $g$ is the generator of $\Z_2$. It is not difficult to check that $\rho$ intertwines the comultiplications and,
$$(\id\ot\rho)(v_l)=\left\{\begin{array}{lcl}1&\text{if}&l\in 2\N,\\g&\text{if}&l\in 2\N+1.\end{array}\right.$$
In particular, one has $\rho(x)=\varepsilon(x)1$ for all $x\in C({\rm Aut}^+(M_2(\C)))$.

\vspace{0.2cm}

\noindent It follows from the preceding discussion that, writing $v_{1,l}$ and $v_{2,l}$ the two copies of $v_l$ in $C(G)$, there exists a unique unital $*$-homomorphism $\pi\,:\, C(G)\rightarrow C^*(\Z_2*\Z_2)$ such that $(\id\ot\pi)(v_{i,1})=1\ot g_i$, for $i=1,2$ where $g_1$, $g_2$ are the two copies of $g$ in $\Z_2*\Z_2$. In particular $\pi$ intertwines the comultiplications, $\pi(x)=\varepsilon(x)1$ for all $x\in C({\rm Aut}^+(M_2(\C)))$ and, whenever $u$ is a representation of the form $u=v_{i_1,l_1}\ot\dots\ot v_{i_n,l_n}$ with $i_s\neq i_{s+1}$ and $k_s\in 2\N+1$ for all $s$ one has $(\id\ot\pi)(u)=1\ot g_{i_1}\dots g_{i_n}$. Let us call such a representation a reduced representation and let $\mathcal{C}$ be the linear span of coefficients of reduced representations. By the previous computation, for all $x\in\mathcal{C}$, $\pi(x)$ is a linear combination of reduced words in $\Z_2*\Z_2$ hence, $\tau\circ\pi(\mathcal{C})=\{0\}$, where $\tau$ is the canonical tracial state on $C^*(\Z_2*\Z_2)$. Note also that, since any coefficient of a reduced representation $u$ is a reduced word in the amalgamated free product $C(G)$, one has $E(\mathcal{C})=\{0\}$ where $E\,:\, C(G)\rightarrow C({\rm Aut}^+(M_2(\C)))$ is the canonical conditional expectation. Since $C(G)$ is the closed linear span of $\mathcal{C}$ and $C({\rm Aut}^+(M_2(\C)))$ we deduce that $\tau\circ\pi=\varepsilon\circ E$. It follows that $\ker(\lambda_G)\subset\ker(\pi)$. Indeed, let $E_r\,:\, C_r(G)\rightarrow C_r({\rm Aut}^+(M_2(\C)))$ be the canonical faithful conditional expectation such that $E_r\circ\lambda_G=\lambda_{{\rm Aut}^+(M_2(\C))}\circ E$ and take $x\in\ker(\lambda_G)$. Then, $E_r(\lambda_G(x^*x))=\lambda_{{\rm Aut}^+(M_2(\C))}(E(x^*x))=0$. Since ${\rm Aut}^+(M_2(\C))$ is co-amenable \cite{Ba99} it follows that $E(x^*x)=0$ hence $\varepsilon(E(x^*x))=\tau(\pi(x^*x))=0$. Since $\Z_2*\Z_2$ is amenable, $\tau$ is faithful so $x\in\ker(\pi)$. Hence $\pi\prec\lambda_G$ and the co-amenability follows from \cite[Theorem 3.11 and 3.12]{KKSV20}.
\end{example}

\subsection{Semi-direct product quantum group}\label{semi-direct}

The semi-direct product quantum group is defined and studied in \cite{Wa19}. Let us recall below the basic facts about this construction.

\vspace{0.2cm}

\noindent Let $G$ be a compact quantum group and $\Lambda$ a finite group acting on $C(G)$ by automorphisms of $G$, meaning that we have a group homomorphism $\alpha\,:\,\Lambda\rightarrow{\rm Aut}(C(G))$ such that $\Delta_G\circ\alpha_g=(\alpha_g\ot\alpha_g)\circ\Delta_G$ for all $g\in\Lambda$. Define the C*-algebra $C(G\rtimes\Lambda)=C(G)\otimes C(\Lambda)$ with the comultiplication $\Delta\,:\,C(G\rtimes\Lambda)\rightarrow C(G\rtimes\Lambda)\otimes C(G\rtimes\Lambda)$ such that:
$$\Delta(a\ot\delta_r)=\sum_{s\in\Lambda}\left[(\id\ot\alpha_s)(\Delta_G(a))\right]_{13}(1\ot\delta_s\ot 1\ot\delta_{s^{-1}r}).$$
In particular, the inclusion $C(\Lambda)\subset C(G\rtimes\Lambda)\,:\,x\mapsto 1\ot x$ preserves the comultiplications. It is shown in \cite{Wa19} that the pair $(C(G\rtimes\Lambda),\Delta)$ is a compact quantum group in its maximal version, the Haar measure $h$ is given by $h=h_G\ot {\rm tr}$, where $h_G$ is the Haar state on $C(G)$ and ${\rm tr}$ is the integration with respect to the uniform probability on $\Lambda$ i.e. ${\rm tr}(\delta_r)=\frac{1}{\vert\Lambda\vert}$. Hence the reduced C*-algebra is $C_r(G)\ot C(\Lambda)$, the von Neumann algebra is ${\rm L}^\infty(G)\ot C(\Lambda)$ and the modular group $\sigma_t$ of $G\rtimes\Lambda$ is $\sigma_t=\sigma_t^G\ot\id$. Moreover, the canonical surjection $\lambda\,:\,C(G\rtimes\Lambda)=C(G)\ot C(\Lambda)\rightarrow C_r(G\rtimes\Lambda)=C_r(G)\ot C(V)$ is $\lambda=\lambda_G\ot\id$, where $\lambda_G$ is the canonical surjection $C(G)\rightarrow C_r(G)$. The irreducible representations and the fusion rules of $G\rtimes\Lambda$ are described in \cite{Wa19}. We could use the general classification of irreducible representations of $G\rtimes\Lambda$ from \cite{Wa19} to deduce the one-dimensional representations of $G\rtimes\Lambda$. However, since we only need to understand the one-dimensional representations, we prefer to include a self contained proof in the next Lemma.

\begin{lemma}\label{Lem1Dim}
The one-dimensional unitary representations of $G\rtimes\Lambda$ are of the form $w\ot v\in C(G\rtimes \Lambda)$, where $w\in C(G)$ and $v\in C(\Lambda)$ are one-dimensional unitary representations of $G$ and $\Lambda$ respectively and $\alpha_r(w)=w$ for all $r\in\Lambda$.
\end{lemma}

\begin{proof}
For this proof, we will view $C(G\rtimes\Lambda)=C(\Lambda,C(G))$ and $C(G\rtimes\Lambda)\ot C(G\rtimes\Lambda)=C(\Lambda\times\Lambda,C(G)\ot C(G))$. With this identification, the comultiplication becomes $\Delta(u)(r,s)=(\id\ot\alpha_r)(\Delta_G(u(rs)))$ for all $u\in C(G\rtimes\Lambda)$ and all $r,s\in\Lambda$ and it is then easy to check that the unitary of the form given in the Lemma are indeed one-dimensional unitary representations of $G\rtimes\Lambda$. Conversely, if $u\in C(G\rtimes\Lambda)$ is a unitary such that $\Delta(u)=u\ot u$, then $u(r)\in C(G)$ is a unitary for all $r\in\Lambda$ and, for all $r,s\in\Lambda$, $(\id\ot\alpha_r)(\Delta_G(u(rs))=u(r)\ot u(s)$. It follows that $w:=u(1)$ is a unitary in $C(G)$ such that $\Delta(w)=w\ot w$. Moreover, since $\alpha_r$ intertwines the comultiplication of $G$ one has $\varepsilon_G\circ\alpha_r=\varepsilon_G$, where $\varepsilon_G$ is the counit of $G$, and,
$$(\id\ot\varepsilon_G\circ\alpha_r)(\Delta_G(u(rs)))=(\id\ot\varepsilon_G)(\Delta_G(u(rs)))=u(rs)=u(r)\varepsilon_G(u(s)).$$
It follows that $v:=(r\mapsto v_r)$, where $v_r:=\varepsilon_G(u_r)$, is a one-dimensional unitary representation of $\Lambda$ and, for all $s\in\Lambda$, $u(s)=wv_s$ i.e. $u=w\ot v\in C(G)\ot C(\Lambda)=C(G\rtimes\Lambda)$. Using that $\Delta(u)=u\ot u$, one checks easily that $\alpha_r(w)=w$ for all $r\in\Lambda$.
\end{proof}

\noindent We collect in the following Proposition some easy observations about $G\rtimes\Lambda$ that are not contained in \cite{Wa19}. We use the terminology introduced in the Introduction before the statement of Theorem \ref{THMA} and we denote by $\Lambda_{cb}(G)$ the Cowling-Haagerup constant of the von Neumann algebra $\Linf(G)$.

\begin{proposition}\label{Prop-Semidirect}
The following holds.
\begin{enumerate}
    \item $G\rtimes\Lambda$ is co-amenable if and only if $G$ is co-amenable.
    \item $\widehat{G\rtimes\Lambda}$ has the Haagerup property if and only if $\Ghat$ has the Haagerup property.
    \item $\Lambda_{cb}(\widehat{G\rtimes\Lambda})=\Lambda_{cb}(\Ghat)$.
    \item The scaling group $\tau_t$ of $G\rtimes\Lambda$ is the one parameter group of $\Linf(G\rtimes\Lambda)=\Linf(G)\otimes C(\Lambda)$ defined by $\tau_t=\tau_t^G\ot\id$, where $\tau_t^G$ is the scaling group of $G$.
    \item The Vaes' $T$-invariant $T(G\rtimes\Lambda)$ is:
    $$\{t\in\R:\exists w\in\mathcal{U}(C(G)),\tau_t^G={\rm Ad}(w),\Delta_G(w)=w\ot w,\,\alpha_r(w)=w\,\forall r\in\Lambda\}.$$
    \item $Sd(G\rtimes\Lambda)=Sd(G)$ and $\tau(G\rtimes\Lambda)=\tau(G)$.
\end{enumerate}
\end{proposition}

\begin{proof}
$(1)$ directly follows from the fact that $\lambda=\lambda_G\ot\id$ and $(2)$ and $(3)$ follows from ${\rm L}^\infty(G\rtimes\Lambda)\simeq{\rm L}^\infty(G)\ot\C^K$, where $K=\vert\Lambda\vert$. To prove $(4)$, one can easily check that the one parameter group $\tau_t:=\tau_t^G\ot\id$ satisfies $\Delta\circ\sigma_t=(\tau_t\ot\sigma_t)\circ\Delta$. To prove $(6)$, we note that since $h=h_G\ot{\rm tr}$, the modular operator $\nabla$ of $h$, is the positive operator on ${\rm L}^2(G)\ot l^2(\Lambda)$ given by $\nabla_G\ot\id$, the equality $Sd(G\rtimes\Lambda)=Sd(G)$ follows, while the equality $\tau(G\rtimes\Lambda)=\tau(G)$ is a direct consequence of $\sigma_t=\sigma_t^G\ot\id$. Finally, $(5)$ is a consequence of $(4)$ and Lemma \ref{Lem1Dim}.\end{proof}

\subsection{Quantum permutation group}\label{SectionQperm}

For $N\in\N^*$, we denote by $S_N^+$ the quantum permutation group on $N$ points. We recall that $C(S_N^+)$ is the universal unital C*-algebra generated by $N^2$ orthogonal projections $u_{ij}$, $1\leq i,j\leq N$ with the relations $\sum_{j=1}^Nu_{ij}=1=\sum_{j=1}^Nu_{ji}$ for all $1\leq i\leq N$. In particular $u=(u_{ij})_{ij}\in M_N(\C)\ot C(S_N^+)$ is a unitary. The comultiplication on $C(S_N^+)$ is defined, using the universal property of $C(S_N^+)$, by the relation $\Delta(u_{ij})=\sum_{k=1}^N u_{ik}\ot u_{kj}$ for all $1\leq i,j\leq N$. In particular, $u$ is a unitary representation of $S_N^+$, called the \textit{fundamental representation}. For $1\leq i\leq N$ we write $L_i:={\rm Span}\{u_{ij}\,:\,1\leq j\leq N\}\subset\Pol(\SN)$. Since the family $(u_{ij})_{1\leq j\leq N}$ is a partition of unity, the vector subspace $L_i$ is actually a unital $*$-subalgebra of $\Pol(\SN)$ and the map $\C^N\rightarrow L_i$, $e_j\mapsto u_{ij}$ is a unital $*$-isomorphism of $*$-algebras. Since $L_i$ is finite dimensional, we may view $L_i\subset C(\SN)$ or $L_i\subset C_r(\SN)$ as an abelian finite dimensional C*-subalgebra and also $L_i\subset\Linf(\SN)$ as an abelian finite dimensional von Neumann subalgebra. We use the same symbol $h$ to denote the Haar state of $\SN$ on $C(\SN)$, $C_r(\SN)$ or $\Linf(\SN)$. We also recall that $h(u_{ij})=\frac{1}{N}$ for all $1\leq i,j\leq N$, where $h$ is the Haar state on $C(S_N^+)$.

\vspace{0.2cm}

\noindent The elementary proof of the next Proposition is left to the reader.

\begin{proposition}\label{lemCE}
Let $1\leq i\leq N$. The following holds
\begin{enumerate}
    \item The unique trace preserving conditional expectation $E_i\,:\,\Linf(\SN)\rightarrow L_i$ satisfies :
    $$E_i(x)=N\sum_{j=1}^Nh(xu_{ij})u_{ij}\quad\text{for all }x\in\Linf(\SN).$$
    \item The map $x\mapsto N\sum_{j=1}^Nh(xu_{ij})u_{ij}$ is a conditional expectation from $C(\SN)$ (resp. $C_r(\SN)$) onto $L_i$. All these maps will be denoted by $E_i$.
    \item The conditional expectation $E_i\,:\, C_r(\SN)\rightarrow L_i$ is faithful.
\end{enumerate}
\end{proposition}

\noindent We collect below some elementary computations concerning the conditional expectation onto $L_i$.

\begin{lemma}\label{lemortho2}
Let $a\in C(\SN)$ and $1\leq i\leq N$ be such that $E_i(a)= 0$. Then,
\begin{enumerate}
    \item For any $1\leq s,j \leq N$, we have $(h\ot \id)(\Delta(a) (u_{is}\ot u_{sj}))=0$.
    \item For all $1\leq s\leq N$, we have $(h\ot\id)\left(\Delta(a) (u_{is}\ot 1)\right) = 0$.
\end{enumerate}
\end{lemma}

\begin{proof}
$(1).$ We show that $\forall\omega \in C(\SN)^*$, $1\leq s,j \leq N$, $(h\ot \omega) (\Delta(a) (u_{is}\ot u_{sj})) = 0$. Let $\omega, s,j$ be as above, and define $\mu = \omega(\,\cdot\, u_{sj}) \in C(\SN)^*$.
Using the Sweedler notation,
$$\Delta(a u_{ij}) \left( 1\ot u_{sj}\right) =  \sum_{t=1}^N a_{(1)} u_{it} \ot a_{(2)} u_{tj}u_{sj}=  a_{(1)} u_{is} \ot a_{(2)} u_{sj}=  \Delta(a) \left( u_{is} \ot u_{sj} \right).$$

\noindent Applying $(h\ot \omega)$, we get
$$
(h\ot\omega) \left(\Delta(a)(u_{is}\ot u_{sj})\right)=  (h\ot \omega)\left(\Delta(au_{ij})(1\ot u_{sj})\right) 
 =  (h\ot \mu) (\Delta (au_{ij}))
=  \mu(1) h(au_{ij}) =0,
$$
where we used the invariance of the Haar state $h$ and the fact that $h(au_{ij})=0$ for all $1\leq j\leq N$ since $E_i(a)=0$ and by definition of $E_i$.

\vspace{0.2cm}

\noindent$(2).$ It suffices to sum the relation $(1)$ for $1\leq j \leq N$ and use $\sum_{j=1}^N u_{sj}=1$ $\forall s$.\end{proof}

\noindent We will use the following Lemma, which is an easy consequence of the factoriality of $\Linf(S_N^+)$ when $N\geq 8$ \cite{Br13}, but we will need the next result for all $N\geq 4$.

\begin{lemma}\label{LemmaCommutant}
For all $N\geq 4$ and all $1\leq k\leq N$ one has $\Linf(S_N^+)'\cap L_k=\C1$.
\end{lemma}

\begin{proof}
$\Pol(S_N^+)$ being weakly dense in $\Linf(S_N^+)$ one has $\Linf(S_N^+)'\cap L_k=\Pol(S_N^+)'\cap L_k$. By the universal property of $\Pol(S_N^+)$, for all $\sigma\in S_N$, there exist unique unital $*$-homomorphisms $R_\sigma,C_\sigma\,:\,\Pol(S_N^+)\rightarrow \Pol(S_N^+)$ such that $R_\sigma(u_{ij})=u_{\sigma(i)j}$ and $C_\sigma(u_{ij})=u_{i\sigma(j)}$, for all $1\leq i,j\leq N$. Note that $R_\sigma$ and $C_\sigma$ are $*$-isomorphisms since $R_{\sigma^{-1}}R_\sigma=R_\sigma R_{\sigma^{-1}}=\id$ and $C_{\sigma^{-1}}C_\sigma=C_\sigma C_{\sigma^{-1}}=\id$. Denoting by $(1,k)\in S_N$ the transposition of $1$ and $k$, one has $R_{(1,k)}(L_k)=L_1$ hence, it suffices to show that $\Pol(S_N^+)'\cap L_1=\C1$.

\noindent Fix a Hilbert space $H$ with two non-commuting orthogonal projections  $P,Q\in\mathcal{B}(H)$ and let us denote by $A\in M_4(\mathcal{B}(H))$ the matrix $$A:=\left(\begin{array}{cccc}P&1-P&0&0\\1-P&P&0&0\\0&0&Q&1-Q\\0&0&1-Q&Q\end{array}\right),$$
and by $B\in M_N(\mathcal{B}(H))$ the block matrix $B=\left(\begin{array}{cc}A&0\\0&I\end{array}\right)$, where $I$ is the identity matrix. Note that $B$ is a magic unitary. Hence, writing $B=(b_{ij})$, there exists a unique unital $*$-homomorphism $\pi\,:\,\Pol(S_N^+)\rightarrow\mathcal{B}(H)$ such that $\pi(u_{ij})=b_{ij}$ for all $1\leq i,j\leq N$.

\vspace{0.2cm}

\noindent To show that $\Pol(S_N^+)'\cap L_1=\C1$, it suffices to show that the only orthogonal projections in $L_1$ that commutes with $\Pol(S_N^+)$ are $0$ and $1$. A projection $p\in L_1\setminus\{0,1\}$ is of the form $p=\sum_{j\in I}u_{1j}$ for $I\subset\{1,\dots,N\}$  with $1\leq \vert I\vert\leq N-1$. Let $\sigma\in S_N$ be a permutation such that $\sigma(I)=\{1,\dots,\vert I\vert+1\}\setminus\{2\}$ so that $q:=C_\sigma(p)=u_{11}+u_{13}+u_{14}+\dots+u_{1,\vert I\vert+1}$. It suffices to show that $q$ does not commute with $u_{33}$ and this follows from $\pi(q)=P$ and $\pi(u_{33})=Q$.
\end{proof}

\subsection{Free wreath products}\label{Sectionfwp}

For a compact quantum group $G$ and an integer $N$, the \emph{free wreath product} of $G$ by $\SN$, as defined by Bichon in \cite{Bi04}, is the CQG $\wreath$ with
$$C(\wreath)=C(G)^{*N}*C(S_N^+)/I,$$
where we consider the full free product and $I$ is the two-sided closed ideal generated by:
\begin{equation}\label{idealI}
    \{\nu_i(a)u_{ij}-u_{ij}\nu_i(a)\,:\,a\in C(G),\,1\leq i,j\leq N\}
\end{equation}
and $\nu_i\,:\,C(G)\rightarrow C(G)^{*N}\subset C(G)^{* N}*C(S_N^+)$ is the unital $*$-homomorphism on the $i^{th}$-copy of $C(G)$ in $C(G)^{*N}$, $u_{ij}\in C(S_N^+)$ are the coefficients of the fundamental representation. If $G$ has a \textit{dual quantum subgroup} $H$ i.e. $C(H)\subset C(G)$,
then we define, following \cite{Fr22}, the \emph{free wreath product with amalgamation} $\wreathH$. This is the CQG with $C(\wreathH) = C(G)^{*_H N}*C(S_N^+)/I$, where the full free product is taken amalgamated over $C(H)$ and the ideal $I$ is the same as in (\ref{idealI}). It is easy to check (using both universal properties) that $C(\wreathH)=C(\wreath)/J$, where $J$ is the closed two-sided ideal generated by $\{\nu_i(a)-\nu_j(a)\,:\,a\in C(H),\,1\leq i,j\leq N\}$.  Note that $G$ always admits the trivial group $\{e\}$ as a dual quantum subgroup, and we have, for $H= \{e\}$, $C(\wreathH) \simeq C(\wreath)$.

\begin{remark}\label{RemPi}
The surjective $*$-homomorphism $a\mapsto a+I\,:\, C(G)^{* N}* C(S_N^+)\rightarrow C(\wreath)$ is injective when restricted to $C(G)^{* N}$ as well as to $C(S_N^+)$. By the universal property, there exists a unique unital $*$-homomorphism $\pi\,:\,C(\wreath)\rightarrow C(G)^{* N} \otimes C(S_N^+)$ such that $\pi(x+I)=x\ot 1$ if $x\in C(G)^{*N}$ and $\pi(x+I)=1\ot x$ if $x\in C(S_N^+)$. The composition of $x\mapsto x+I$ and $\pi$ is the map sending $C(G)^{*N}$ and $C(S_N^+)$ to their respective copies in $C(G)^{* N} \otimes C(S_N^+)$, which is injective. The same holds for $C(G)^{*_H N}$ and $C(\SN)$ in the amalgamated case.
\end{remark}

\noindent Following the previous Remark, we will always view $C(G)^{*_H N},\, C(S_N^+)\subset C(\wreathH)$.

\vspace{0.2cm}

\noindent We endow the unital $C^*$-algebra $C(\wreathH)$ with the unique unital $*$-homomorphism $\Delta\,:\,C(\wreathH)\rightarrow C(\wreathH)\ot C(\wreathH)$ satisfying:
$$\Delta(\nu_i(a))=\sum_{j=1}^N(\nu_i\ot\nu_j)(\Delta_G(a))(u_{ij}\ot 1)\text{ and }\Delta(u_{ij})=\sum_{k=1}^N u_{ik}\ot u_{kj}.$$

\begin{remark}\label{RmkSubgroup}
\begin{enumerate}
    \item Both $G^{*_H N}$ and $S_N^+$ are compact quantum subgroups of $\wreathH$ via the maps $(\id\ot\varepsilon_{S_N^+})\circ\pi\,:\,C(\wreathH)\rightarrow C(G)^{*_H N}$ and $(\varepsilon_{G^{* N}}\ot\id)\circ\pi\,:\, C(\wreathH)\rightarrow C(S_N^+)$ (which obviously intertwine the comultiplications), where $\pi$ is the map defined in Remark \ref{RemPi}.
    \item Let $\nu\,:\,C(H)\rightarrow C(\wreathH)$ be the common restriction of the maps $\nu_i$ on $C(H)\subset C(G)$, $1\leq i\leq N$. Then, $\nu$ is faithful and, since $\sum_ju_{ij}=1$ we see that $\nu$ intertwines the comultiplications. Hence, $H$ is a dual compact subgroup of $\wreathH$.
    \end{enumerate}
    \end{remark}

\noindent Let us now describe another specific compact quantum subgroup of $\wreathH$. The only mentions of this subgroup that we could find in the literature are in \cite[Proposition 2.6]{Bi04}, when $N=2$ and $G=\widehat{\Gamma}$ is the dual of a discrete group, and in Gromada's recent work \cite{Gr23} where it appears under the name \textit{wreath product}, in the non-amalgamated case. We would like to thank the referee for showing us this reference.

\vspace{0.2cm}

\noindent Fix $\sigma\in S_N$. By the universal property of $C(G)^{*_H N}$, there exists a unique unital $*$-homomorphism $\alpha_\sigma\,:\, C(G)^{*_H N}\rightarrow C(G)^{*_H N}$ such that $\alpha_\sigma\circ\nu_i=\nu_{\sigma(i)}$ for all $1\leq i\leq N$. Since we clearly have $\alpha_\sigma\alpha_\tau=\alpha_{\sigma\tau}$, for all $\sigma,\tau\in S_N$, $\alpha$ is an action of $S_N$ on the C*-algebra $C(G)^{*_H N}$ by unital $*$-isomorphisms. Moreover, it is easy to see that, for all $\sigma\in\Sigma$, $(\alpha_\sigma\ot\alpha_\sigma)\circ\Delta_{G^{*N}}=\Delta_{G^{*N}}\circ\alpha_\sigma$, where $\Delta_{G^{*N}}$ is the comultiplication on $C(G)^{*_H N}$. Note that this action by automorphisms of $G^{*_H N}$ is actually the restriction of the action of $S_N^+$ on the compact quantum group $G^{*_H N}$ to the compact quantum subgroup $S_N$ which was described in \cite[Proposition 2.1]{Bi04}. Let us now consider the semi-direct product quantum group $G^{*_H N}\rtimes S_N$ associated to this action, as described in Section \ref{semi-direct}. We show below that the quantum group $G^{*_H N}\rtimes S_N$ is actually a compact quantum subgroup of $\wreathH$.

\begin{proposition}\label{Prop-Subgroup}
Given $N\in\N^*$, there exists a unique unital $*$-homomorphism $$\pi\,:\,C(\wreathH)\rightarrow C(G^{*_H N}\rtimes S_N)=C(G)^{*_H N}\otimes C(S_N)\text{ s.t. }\left\{\begin{array}{lcl}\pi(\nu_i(a))&=&\nu_i(a)\otimes 1\\ \pi(u_{ij})&=&1\ot\chi_{ij}\end{array}\right.$$
where $\chi_{ij}\in C(S_N)$ is the characteristic function of $A_{i,j}:=\{\sigma\in S_N\,:\,\sigma(j)=i\}$. Moreover,
\begin{enumerate}
    \item $\pi$ is surjective and intertwines the comultiplications.
    \item $\pi$ is an isomorphism if and only if $N\in\{1,2\}$ or if $N=3$ and the inclusion $C(H)\hookrightarrow C(G)$ is an isomorphism (a special case being when $G$ is the trivial group).
\end{enumerate}
\end{proposition}

\begin{proof}
The existence of $\pi$ is a direct consequence of the universal property of the C*-algebra $C(\wreathH)$ and the fact that the matrix $(\chi_{ij})_{ij}\in M_N(C(S_N))$ is a magic unitary.

\vspace{0.2cm}

\noindent(1). Since $C(S_N)$ is generated by the $\chi_{ij}$, the surjectivity of $\pi$ is clear. Let us check that $\pi$ intertwines the comultiplications. Recall that the comultiplication on $\wreathH$ is denoted by $\Delta$, the one on $G$ by $\Delta_G$, the one on $G^{*_H N}$ by $\Delta_{G^{*N}}$ (it satisfies $\Delta_{G^{*N}}\circ\nu_i=\nu_i\ot\nu_i\circ\Delta_G$) and let us denote the one on $C(G^{*_H N}\rtimes S_N)$ by $\Delta_s$. On the one hand, $\forall 1\leq i,j\leq N$, $a\in C(G)$,
\begin{eqnarray*}
\Delta_s(\pi(u_{ij}))&=&\Delta_s(1\ot\chi_{ij})=\sum_{\sigma\in S_N,\,\sigma(j)=i}\Delta_s(1\ot\delta_\sigma)=\sum_{\sigma\in S_N,\,\sigma(j)=i}\sum_{\tau\in S_N}(1\ot\delta_\tau\ot 1\ot\delta_{\tau^{-1}\sigma}),\\
\Delta_s(\pi(\nu_i(a)))&=&\sum_{\sigma\in S_N}\Delta_s(\nu_i(a)\ot\delta_\sigma)
=\sum_{\tau,\sigma\in S_N}\left[(\id\ot\alpha_\tau)(\Delta_{G^{*N}}(\nu_i(a)))\right]_{13}(1\ot\delta_\tau\ot 1\ot\delta_{\tau^{-1}\sigma})\\
&=&\sum_{\tau\in S_N}\left[(\id\ot\alpha_\tau)((\nu_i\ot\nu_i)(\Delta_G(a)))\right]_{13}(1\ot\delta_\tau\ot1\ot 1).
\end{eqnarray*}
On the other hand, for all $1\leq i,j\leq N$ and all $a\in C(G)$,
\begin{eqnarray*}
(\pi\ot\pi)(\Delta(\nu_i(a)))&=&\sum_{j=1}^N(\pi\ot\pi)(\nu_i\ot\nu_j)(\Delta_G(a))(\pi(u_{ij})\ot 1)\\
&=&\sum_{j=1}^N\left[(\nu_i\ot\nu_j)(\Delta_G(a))\right]_{13}(1\ot\chi_{ij}\ot 1\ot 1)\\
&=&\sum_{j=1}^N\sum_{\tau\in S_N,\,\tau(j)=i}\left[(\id\ot\alpha_\tau)(\nu_i\ot\nu_i)(\Delta_G(a))\right]_{13}(1\ot\delta_\tau\ot 1\ot 1)\\
&=&\sum_{\tau\in S_N}\left[(\id\ot\alpha_\tau)(\nu_i\ot\nu_i)(\Delta_G(a))\right]_{13}(1\ot\delta_\tau\ot 1\ot 1)
\end{eqnarray*}
where, in the last equality, we use the partition $S_N=\bigsqcup_{j=1}^N\{\tau\in S_N\,:\,\tau(j)=i\}$. Moreover,

\begin{eqnarray*}
(\pi\ot\pi)(\Delta(u_{ij}))&=&\sum_{k=1}^N\pi(u_{ik})\ot\pi(u_{kj})=\sum_{k=1}^N1\ot\chi_{ik}\ot 1\ot\chi_{kj}\\
&=&\sum_{k=1}^N\sum_{\substack{\sigma,\tau\in S_N,\\ \tau(k)=i,\, \sigma(j)=k}}1\ot\delta_\tau\ot 1\ot\delta_{\sigma}
=\sum_{\substack{\sigma,\tau\in S_N,\\\sigma(j)=i}}1\ot\delta_\tau\ot 1\ot\delta_{\tau^{-1}\sigma},
\end{eqnarray*}
where we use, in the last equality, the following partition:
$$\{(\tau,\tau^{-1}\sigma)\in S_N^2\,:\,\sigma(j)=i\}=\bigsqcup_{k=1}^N\{(\tau,\sigma)\in S_N^2\,:\,\tau(k)=i\text{ and }\sigma(j)=k\}.$$
\vspace{0.2cm}

\noindent$(2).$ Suppose that $N=2$. In that case, Bichon's proof of \cite[Proposition 2.1]{Bi04} can be directly adapted. It is well known that $C(S_2^+)$ is commutative: the surjection $C(S_2^+)\rightarrow C(S_2)$, $u_{ij}\mapsto\chi_{ij}$ is an isomorphism. Hence, we view $C(S_2^+)=C(S_2)$ and we have $u=\left(\begin{array}{cc}\delta_1&\delta_\tau\\\delta_\tau&\delta_1\end{array}\right)$, where $\tau$ is the unique non-trivial element in $S_2$. Now, for $i\in\{1,2\}$ and $a\in C(G)$, $\nu_i(a)\in C(\wreathH)$ is commuting with the line $i$ of $u$ hence, it is also commuting with the other line. Hence, $C(G)^{*_H2}$ and $C(S_2^+)$ are commuting in $C(\wreathH)$ and it follows that $\pi$ is actually injective.

\vspace{0.2cm}
\noindent Suppose that $N\geq 4$. In that case, it is well known that $C(S_N^+)$ is infinite-dimensional (the classical argument is actually contained in the proof of Lemma \ref{lemortho2}). In particular, $C(S_N^+)\mapsto C(S_N)$, $u_{ij}\mapsto \chi_{ij}$ is not injective which implies that $\pi$ itself is not injective.
\vspace{0.2cm}

\noindent The non injectivity in the case $N=3$ is a consequence of Theorem \ref{propreduced} (the proof of which does not rely on this isomorphism even in the case $N=3$), using the fact that an isomorphism between these quantum group would intertwine the Haar measures. It is then enough, when $G$ is non-trivial to show that there is a \emph{reduced} operator in $C(G\wr_{*,H} S_3^+)$, hence of Haar measure $0$, which is sent to an element of $C(G^{*_H 3}\rtimes S_3)\simeq C(G)^{*_H 3}\otimes \C^6$. Since $H$ is a strict (not necessarily proper) dual quantum subgroup of $G$, then $\mathrm{Rep}(H)$ is a full subcategory of $\mathrm{Rep}(G)$, with a non trivial irreducible representation $v$ of $G$ which is not a representation of $H$, and $a,b\in C(G)\setminus C(H)$ non trivial coefficients of $v$ and $\overline{v}$ respectively such that $ab$ has non zero Haar measure (using the unique intertwiner between the trivial representation and $v\ot \overline{v}$). Then we can consider the element $\nu_1(a) u_{22} \nu_1(b)\in C(G\wr_{*,H} S_3^+)$, which is sent to $\nu_1(a)\nu_1(b) \ot u_{22} = \nu_1(ab)\ot u_{22} \in C(G)^{*_H 3}\otimes \C^6$. However, the word $\nu_1(a) u_{22} \nu_1(b)$ is reduced in the sense of \ref{propreduced}, because $a$ and $b$ are coefficients of a non trivial irreducible representation, and therefore of Haar measure $0$ and the element $u_{22}$ is such that $E_{L_1}(u_{22})$ is equal to $0$ if $N\geq 3$ and to $u_{22}$ if $N\in \lbrace 0,1\rbrace$, while the element $\nu_1(ab) \ot u_{22}$ is of Haar measure $h_G(ab)/3\neq 0$, because the Haar measure on $C(G^{*_H 3}\rtimes S_3)\simeq C(G)^{*_H 3}\otimes \C^6$ is the tensor product of the Haar measures on $C(G)^{*_H 3}$ and on $C(S_3)$. \vspace{0.2cm}

\noindent If $G$ is the trivial group then the surjection is the canonical morphism $C(\SN)\twoheadrightarrow C(S_N)$ which is an isomorphism if and only if $N\leq 3$.
\end{proof}

\begin{remark}
The condition $C(H)\hookrightarrow C(G)$ not being an isomorphism is equivalent to the index $[G:H]$ being greater than $2$, as defined in \ref{index}.
\end{remark}

\noindent Let us show that this semi-direct product agrees with the wreath product by $S_N$ as defined in \cite{Gr23}. We first start by defining a well-known family of compact quantum subgroups of $\wreathH$.

\begin{propdef}\label{subgroupsn}
For any compact quantum subgroup $K_N\leq S_N^+$ with $\mu : C(S_N^+)\twoheadrightarrow C(K_N)$, we have a corresponding compact quantum subgroup $G\wr_{\ast,H} K_N \leq \wreathH$, with $C(G\wr_{\ast,H} K_N) := C(G)^{\ast_{H}N}\ast C(K_N)/I$, where the double-sided closed ideal $I$ is once again defined as in (\ref{idealI}), replacing the elements $u_{ij}$ by their images through $\mu$ in $C(K_N)$, with the comultiplication defined in the same way as for the free wreath product with $S_N^+$. In particular $G\wr_{\ast,H} S_N$ is a compact quantum subgroup of $\wreathH$, where $S_N$ is seen as a compact quantum group of $S_N^+$.
\end{propdef}

\begin{proof}
It is enough to show that the quotient map $$ \psi : C(\wreathH)\rightarrow C(G\wr_{\ast,H} K_N)\text{ s.t. }\left\{\begin{array}{lcl}\psi(\nu_i(a))&=&\nu_i(a)\\ \psi(u_{ij})&=&\mu(u_{ij})\end{array}\right.,$$
intertwines the comultiplications but this is obvious by the definitions.
\end{proof}

\begin{remark}
Note that the considered dimension of the generating matrix of the compact quantum group $K_N$ matters a lot, as it changes the numbers of free copies of $C(G)$ in the definition of the free wreath product. Take $H$ to be the trivial group and $K_N$ and $K_M$ for $N\neq M$ two presentations of the trivial group (seen as acting on $\C^N$ and $\C^M$ respectively), then we have that $G\wr_{\ast,H} K_N = G^{\ast N}$ and $G\wr_{\ast,H} K_M = G^{\ast M}$, which are oftentimes not isomorphic.
\end{remark}

\begin{proposition}
The wreath product $G\wr S_N$ from \cite{Gr23} is isomorphic to the semi-direct product $G^{*N}\rtimes S_N$.
\end{proposition}

\begin{proof}
It is enough to prove that the map $\pi : C(\wreath) \twoheadrightarrow C(G^{\ast N}\rtimes S_N) \simeq C(G)^{*N}\otimes C(S_N)$ defined in proposition \ref{Prop-Subgroup} is the composition of the maps $C(\wreath)\twoheadrightarrow C(G\wr_\ast S_N)$ and $C(G\wr_\ast S_N) \twoheadrightarrow C(G^{*N}\rtimes S_N)$. Let $v$ be the matrix $(\chi_{ij})_{1\leq i,j\leq N}$, with coefficients generating $C(S_N)$. Then the map $C(G\wr_\ast S_N) \twoheadrightarrow C(G\wr S_N)$ as defined in \cite{Gr23} sends $\nu_i(g)$ to $\nu_i(g) \ot 1$ and $\chi_{ij}$ to $1\ot \chi_{ij}$, so the composition with the previous quotient coincides exactly with the map $\pi$.
\end{proof}

\begin{remark}
For many quantum groups $G$ we can also use K-theory to distinguish between the free wreath product and the semi-direct product. Indeed, the K-theory of the algebra of the semi-direct product $G^{*3}\rtimes S_3$ is just the K-theory of the tensor product $C(G)^{*3}\ot \C^6$ which is
\begin{align*}
    K_0(C(G)^{*3}\ot \C^6) \simeq (K_0(C(G))^3/\Z^2)^3, \text{ and, } K_1(C(G)^{*3}\ot \C^6) \simeq K_1(C(G))^{18},
\end{align*}
which can only coincide with the K-theory of the free wreath product algebra if $K_0(C(G))\simeq \Z$ and $K_1(C(G))=0$, using the computations of Theorem \ref{THMD} which are independent from Proposition \ref{Prop-Subgroup}
\end{remark}

\noindent In \cite[Section 3.4]{BCF20} it is mentioned that the question whether the von Neumann algebra of a quantum reflection group $H_N^{s+}:=\Z_s\wr_*S_N^+$ is diffuse is open when $N\leq 7$. We add the following Proposition in order to provide a complete answer to this question.

\begin{proposition}
The von Neumann algebra $\Linf(\wreathH)$ is diffuse if and only if at least one of the following conditions hold:
\begin{itemize}
    \item $N\geq 4$,
    \item $\Irr(G)$ is infinite,
    \item $[G:H]\geq 2$ (i.e. $C(H)\hookrightarrow C(G)$ is not surjective) and $N\geq 2$.
\end{itemize}
\end{proposition}

\begin{proof}
If $G$ is trivial, it directly follows from Lemma \ref{LemmaInfinite} since $S_N^+\simeq S_N$ for $N\leq 3$ and $C(S_N^+)$ is infinite-dimensional for $N\geq 4$.
Suppose that $G$ is non-trivial and let $N\geq 3$. Let us denote by $u=(u_{ij})$ and $v=(v_{ij})$ the fundamental representations of $S_N^+$ and $S_{N-1}^+$ respectively. By the universal property, there exists a unique unital $*$-homomorphism $\rho\,:\,C(\wreathH)\rightarrow C(G\wr_{*,H} S_{N-1}^+)$ such that $\rho(u_{ij})=v_{ij}$ if $1\leq i,j\leq N-1$, $\rho(u_{ij})=\delta_{i,j}$ if $i$ or $j$ is equal to $N$ and $\rho(\nu_i(a))=\nu_i(a)$ if $1\leq i\leq N-1$, $\rho(\nu_N(a))=\nu_{N-1}(a)$ for all $a\in C(G)$. Since $\rho$ is clearly surjective a direct induction implies that, for all $N\geq 2$, $C(\wreathH)$ is infinite-dimensional (hence $\Linf(\wreathH)$ is diffuse by Lemma \ref{LemmaInfinite}) whenever $C(G\wr_{*,H} S_2^+)$ is. By Proposition \ref{Prop-Subgroup}, $C(G\wr_{*,H} S_2^+)\simeq C(G*_H G)\ot\C^2$. Hence, it is infinite-dimensional as soon as $C(G*_H G)$ is. There is an embedding $C(G)\hookrightarrow C(G*_H G)$, so the case $\vert\Irr(G)\vert=\infty$ is a direct consequence of Lemma \ref{LemmaInfinite}. If $C(H)\hookrightarrow C(G)$ is not surjective, then there is a non zero element $a\in C(G)$ such that $E_H(G)=0$ (for example, take a coefficient of an irreducible representation of $G$ which is not in the representation category of $H$). Then the words $b_k = \nu_1(a)\nu_2(a)\dots \nu_{\bar{k}}(a)$ taking a product of $k$ elements, where $\bar{k}$ is $1$ if $k$ is odd and $2$ if $k$ is even, form a family of linearly independent reduced elements in $C(G*_HG)$, which is therefore infinite dimensional.
\end{proof}

\subsection{Graphs of operator algebras}\label{graphsop}

We recall below some notions and results from \cite{FF14,FG18}. If $\G$ is a graph in the sense of \cite[Def 2.1]{Se77}, its vertex set will be denoted $V(\G)$ and its edge set will be denoted $E(\G)$. We will always assume that $\G$ is at most countable. For $e\in E(\G)$ we denote by $s(e)$ and $r(e)$ respectively the source and range of $e$ and by $\overline{e}$ the inverse edge of $e$. An \emph{orientation} of $\G$ is a partition $E(\G) = E^{+}(\G)\sqcup E^{-}(\G)$ such that $e\in E^{+}(\G)\Leftrightarrow\overline{e}\in E^{-}(\G)$.

\vspace{0.2cm}

\noindent The data $(\G, (A_{q})_{q\in V(\G)}, (B_{e})_{e\in E(\G)}, (s_{e})_{e\in E(\G)})$ will be called a graph of C*-algebras if:
\begin{itemize}
\item $\G$ is a connected graph.
\item For every $q\in V(\G)$ and every $e\in E(\G)$, $A_{q}$ and $B_{e}$ are unital C*-algebras.
\item For every $e\in E(\G)$, $B_{\overline{e}} = B_{e}$.
\item For every $e\in E(\G)$, $s_{e}: B_e \rightarrow A_{s(e)}$ is a unital  faithful $*$-homomorphism.
\end{itemize}
Define $r_e=s_{\overline{e}}: B_e \rightarrow A_{s(e)}$.

\vspace{0.2cm}

\noindent Fix a maximal subtree $\mathcal{T}\subset\G$ and define $P= \pi_1(\G,A_q,B_e,\T)$ to be the maximal fundamental C*-algebra of the graph of C*-algebras $(\G,A_q,B_e,s_e)$ with respect to the maximal subtree $\mathcal{T}$. This means that $P$ is the universal unital C*-algebra generated by $A_q$, for $q\in V(\G)$ and unitaries $u_e$, for $e\in E(\G)$ with the following relations:

\begin{itemize}
\item For every $e\in E(\G)$, $u_{\overline{e}} = u_{e}^{*}$.
\item For every $e\in E(\G)$ and every $b\in B_e$, $u_{\overline{e}}s_{e}(b)u_{e}=r_{e}(b)$.
\item For every $e\in E(\mathcal{T})$, $u_{e}=1$.
\end{itemize}

\noindent It is known that $P$ is non-zero and that the canonical unital $*$-homomorphisms $A_q\rightarrow P$ are all faithful \cite[Remark 2.2]{FG18}. Hence, we will always view $A_q\subset P$ for all $q\in V(\G)$.

\vspace{0.2cm}

\noindent Assume now that there exists, for all $e\in E(\G)$ a conditional expectation $E_e^s\,:\,A_{s(e)}\rightarrow s_e(B_e)$. For $p\in V(\G)$, an element $a\in P$ will be called a \textit{reduced operator} from $p$ to $p$ if it is of the form $a=a_0u_{e_1}a_1\dots u_{e_n} a_n$ where $n\geq1$, $(e_1,\dots,e_n)$ is a path in $\G$ from $p$ to $p$ (i.e $e_k\in E(\G)$ are such that $r(e_k)=s(e_{k+1})$ and $s(e_1)=r(e_n)=p$), $a_0\in A_p$, $a_k\in A_{r(e_k)}$ and, for all $1\leq k\leq n-1$, if $e_{k+1}=\overline{e}_k$ then $E^s_{e_{k+1}}(a_k)=0$. Then one can construct \cite{FG18}, the unital C*-algebra $P_r$ called the \textit{vertex reduced fundamental algebra} which is the unique (up to canonical isomorphism) quotient $\lambda\,:\,P\rightarrow P_r$ of $P$ satisfying the following:

\begin{enumerate}
    \item There exists, for all $p\in V(\G)$, a ucp map $E_p\,:\,P_r\rightarrow A_p$ such that $E_p(\lambda(a))=a$ for all $a\in A_p$ and $E_p(\lambda(a))=0$ for all reduced operators $a\in P$ from $p$ to $p$. Moreover, the family $\{E_p\,:\,\in V(\G)\}$ is GNS-faithful.
    \item For any unital C*-algebra $C$ with a surjective unital $*$-homomorphism $\rho\,:\,P\rightarrow C$ and a GNS-faithful family of ucp map $\varphi_p\,:\,C\rightarrow A_p$, $p\in V(\G)$, such that $\varphi_p(\rho(a))=a$ for all $a\in A_p$ and $\varphi_p(\rho(a))=0$ for all reduced operators $a\in P$ from $p$ to $p$ there exists a unique unital $*$-isomorphism $\nu\,:\,P_r\rightarrow C$ such that $\nu\circ\lambda=\rho$.
\end{enumerate}

\noindent We recall that, given unital C*-algebras $A,B_i$, $i\in I$, a family of ucp maps $\varphi_i\,:\ A\rightarrow B_i$ is called \textit{GNS-faithful} if $\bigcap_{i\in I}{\rm Ker}(\pi_i)=\{0\}$, where $(H_i,\pi_i,\xi_i)$ is the GNS-construction of $\varphi_i$.

\vspace{0.2cm}

\noindent Note that a ucp map satisfying $(1)$ is necessarily unique and property $(1)$ implies that $\lambda\,:\, P\rightarrow P_r$ is faithful on $A_p$, for all $p\in V(\G)$ so that we may and will view $A_p\subset P_r$ for all $p\in V(\G)$ and the ucp maps $E_p\,:\,P_r\rightarrow A_p$ become conditional expectations under this identification. If all the ucp maps $E_e^s$ are supposed to be GNS-faithful then the vertex reduced fundamental algebra $P_r$ is the same as the reduced fundamental algebra constructed in \cite{FF14} and the conditional expectations $E_p\,:\,P_r\rightarrow A_p$ are all GNS-faithful.

\section{Free wreath products as fundamental algebras}

\noindent Let $N\in\N^*$ and $\mathcal{T}_N$ be the rooted tree with $N+1$ vertices $p_0,\dots,p_N$, $p_0$ being the root and $2N$ edges $v_1,\dots,v_N, \vbar_1,\dots\vbar_N$, source maps $s(v_k)=p_0$ and range maps $r(v_k)=p_k$ for all $1\leq k\leq N$.
\subsection{The full version}\label{SectionGraphFull}

Consider the graph of C*-algebras over $\T_N$ given by $\A_{p_0}=C(H) \ot C(\SN)$, $\A_{p_k}:=C(G)\otimes\C^N$, $\B_{v_k}=\B_{\vbar_k}=C(H)\ot \C^N$ ($1\leq k\leq N$) with source map $s_{v_k}\,:\, C(H) \ot \C^N\rightarrow C(H) \ot L_k\subset C(H)\ot C(\SN)$, $h\ot e_j\mapsto h\ot  u_{kj}$ and range map $r_{v_k}\,:\,C(H) \ot \C^N\rightarrow C(G)\ot\C^N$ being the canonical inclusion. Note that our graph of C*-algebras has conditional expectations $\id\ot E_k\,:\,C(H) \ot C(S_N^+)\rightarrow C(H) \ot L_k$ (Proposition \ref{lemCE}) which can be non GNS-faithful and $E_H\ot\id\,:\, C(G)\ot\C^N\rightarrow C(H) \ot \C^N$, coming from proposition \ref{dualqsg}, which can also be non-GNS faithful.
\vspace{0.2cm}

\noindent Let us denote by $\mathcal{A}$ the maximal fundamental C*-algebra of this graph of C*-algebras relative to the unique maximal subtree $\T_N$ itself so that $\mathcal{A}$ is the universal unital C*-algebra generated by $\mathcal{A}_{p_k}$ for $0\leq k\leq N$ with the relations $s_{v_k}(a)=r_{v_k}(a)$ for all $a\in C(H) \ot \C^N$ and all $1\leq k\leq N$. Recall that $\nu_i\,:\, C(G)\rightarrow C(G)^{*_H N}\subset C(\wreathH)$ denotes the $i^{th}$-copy of $C(G)$ in $C(G)^{*_H N}$, we also denote by $\nu$ the common restriction of the $\nu_i$'s to $C(H)$

\begin{proposition}\label{PropFull}
There is a unique isomorphism $\pi\,:\,\mathcal{A}\rightarrow C(\wreathH)$ such that $\pi(h\ot u_{ij})= \nu (h) u_{ij}$ ($h\ot u_{ij}\in\A_{p_0}=C(H)\ot C(\SN)$) and, $\pi(a\ot e_j)=\nu_i(a)u_{ij}$ for $a\ot e_j\in\A_{p_i}=C(G)\ot\C^N$, $1\leq i,j\leq N$.
\end{proposition}

\begin{proof}
Let us show the existence of $\pi$. Note first that, since $\nu_i(a)$ and $u_{ij}$ commute in $C(\wreathH)$, there exists a unique unital $*$-homomorphism $\pi_i\,:\,\A_{p_i}=C(G)\ot\C^N\rightarrow C(\wreathH)$ such that $\pi_i(a\ot e_j)=\nu_i(a)u_{ij}$, for all $a\in C(G)$ and $1\leq j\leq N$. We also define $\pi_0\,:\,\mathcal{A}_{p_0}=C(H)\ot C(S_N^+)\rightarrow C(\wreathH)$, $\pi_0(h\ot u_{ij})=\nu(h)u_{ij}$. Next, for all $1\leq i,j\leq N,$ and $h\in C(H)$, one has $\pi_0(s_{v_i}(h\ot e_j))=\pi_0(h\ot u_{ij})=\nu(h) u_{ij}$ and $\pi_i(r_{v_i}(h\ot e_j))=\pi_i(h \ot e_j)=\nu(h) u_{ij}$. By the universal property of $\mathcal{A}$, there exists a unique unital $*$-homomorphism $\pi\,:\,\mathcal{A}\rightarrow C(\wreathH)$ satisfying the properties of the proposition. Moreover, since the image of $\pi$ contains $u_{ij}$ for all $i,j$ and $\sum_{j=1}^N\pi_i(a\ot e_j)=\sum_{j=1}^N\nu_i(a)u_{ij}=\nu_i(a)$ for all $a\in A$ and $1\leq i\leq N$, it follows that $\pi$ is surjective.

\vspace{0.2cm}

\noindent To show that it is an isomorphism, we will give an inverse, using this time the universal property of $C(\wreathH)$. By the universal property of the full free product, there exists a unique unital $*$-homomorphism $\mu\,:\,C(G)^{* N}* C(S_N^+)\rightarrow\mathcal{A}$ such that, for all $ 1\leq i\leq N$, $a\in C(G)$, $b\in C(S_N^+)$,
$$\mu(\nu_i(a))=a\ot 1\in C(G)\ot\C^N=\mathcal{A}_{p_i}\subset\mathcal{A}\text{ and }\mu(b)=1\ot b\in C(H) \ot C(S_N^+)=\mathcal{A}_{p_0}\subset\mathcal{A}.$$
Recall that $C(\wreathH)=C(G)^{*_H N}*C(S_N^+)/I$, where $I$ is defined in Equation $(\ref{idealI})$. Note that
\begin{eqnarray*}
\mu(\nu_i(a)u_{ij})&=&(a\ot 1)(1\ot u_{ij})=\mu(\nu_i(a))s_{v_i}(1\ot e_j)=\mu(\nu_i(a))r_{v_i}(1\ot e_j)=(a\ot 1)(1\ot e_j)\\
&=&(1\ot e_j)(a\ot 1)
=r_{v_i}(1\ot e_j)\mu(\nu_i(a))=s_{v_i}(1\ot e_j)\mu(\nu_i(a))\\&=&(1\ot u_{ij})\mu(\nu_i(a))
=\mu(u_{ij}\nu_i(a)).
\end{eqnarray*}
Moreover, we have that for every $h\in C(H)$, and $1\leq i,j\leq N$,
\begin{eqnarray*}
\mu(\nu_i(\iota(h))) &=& (\iota(h)\ot 1) = s_{v_i}(h\ot 1) = r_{v_i}(h\ot 1) = (h\ot 1)\\ &=& r_{v_j}(h\ot 1)= s_{v_j}(h\ot 1)=(\iota(h)\ot 1) =\mu(\nu_j(\iota(h))),
\end{eqnarray*}
Hence, the images of $C(H)$ through $\mu$ coincide and $I\subset{\rm ker}(\mu)$ so there exists a unique unital $*$-homomorphism $\rho\,:\,C(\wreathH)\rightarrow\mathcal{A}$ which factorizes $\mu$. It is clear that $\rho$ is surjective and it is easy to check that $\rho$ is the inverse of $\pi$.\end{proof}

\subsection{The reduced version}\label{SectionGraphReduced}

\noindent Consider the graph of C*-algebras over $\T_N$ given by $A_{p_0}=C_r(H) \ot C_r(\SN)$, $A_{p_k}:=C_r(G)\otimes\C^N$, $B_{v_k}=B_{\vbar_k}=C_r(H)\ot \C^N$ for all $1\leq k\leq N$ with source map $s_{v_k}\,:\,C_r(H)\ot \C^N\rightarrow C_r(H)\ot L_k\subset C_r(H)\ot C_r(\SN)$, $h\ot e_j\mapsto h\ot u_{kj}$ and range map $r_{v_k}\,:\,C_r(H)\ot \C^N\rightarrow C_r(G)\ot\C^N$, being the canonical inclusion. Note that our graph of C*-algebras has faithful conditional expectations $\id\ot E_k\,:\,C_r(H)\ot C_r(S_N^+)\rightarrow C_r(H) \ot L_k$ (Proposition \ref{lemCE}) and $E_H\ot\id\,:\, C_r(G)\ot\C^N\rightarrow C_r(H)\ot \C^N$, thanks to Proposition \ref{dualqsg}. Let $A$ the vertex reduced fundamental C*-algebra of this graph of C*-algebras with faithful conditional expectations and view $A_{p_k}\subset A$ for all $0\leq k\leq N$. By the universal property of $\mathcal{A}$ defined in Section \ref{SectionGraphFull}, there exists a unique unital $*$-homomorphism (which is surjective) $\lambda'\,:\, \mathcal{A}\rightarrow A$ such that $\lambda'\vert_{\mathcal{A}_{p_0}}=\lambda_H \ot \lambda_{S_N^+}$ and $\lambda'\vert_{\mathcal{A}_{p_k}}=\lambda_{G}\ot{\rm id}_{\C^N}$ for all $1\leq k\leq N$. Let $E\,:\, A\rightarrow C_r(H)\ot C_r(S_N^+)$ the GNS-faithful conditional expectation and define $\omega:=h_H \ot h_{S_N^+}\circ E$ so that $\omega\vert_{A_{p_0}}=h_H\ot h_{S_N^+}$ and $\omega(c)=0$ for $c\in A$ a reduced operator. The state $\omega\in A^*$ is called the \textit{fundamental state}. Let $\lambda\,:\,C(\wreathH)\rightarrow C_r(\wreathH)$ be the canonical surjection.

\begin{theorem}\label{propreduced}
The Haar state $h\in C(\wreathH)^*$ is the unique state such that:
\begin{equation*}
    h\left(a_0 \nu_{i_1}(b_1) a_1 \nu_{i_2}(b_2)\dots \nu_{i_{n}}(b_{n}) a_n\right) =0
\end{equation*}
whenever $a_k\in C(S_N^+)\subset C(\wreathH)$, and $b_k\in C(G)$ are such that $E_{H}(b_k)=0$ for all $k$ and if $i_k = i_{k+1}$, then $E_{i_k}(a_k) = 0$.
\vspace{0.2cm}

\noindent There exists a unique unital $*$-isomorphism $\pi_r\,:\,A\rightarrow C_r(\wreathH)$ such that $\lambda\circ\pi=\pi_r\circ\lambda'$, where $\pi\,:\,\mathcal{A}\rightarrow C(\wreathH)$ is the isomorphism of Proposition \ref{PropFull}. Moreover, $\pi$ intertwines the Haar state on $C_r(\wreathH)$ and the fundamental state $\omega\in A^*$.
\end{theorem}

\noindent An element of the form $a_0 \nu_{i_1}(b_1) a_1 \nu_{i_2}(b_2)\dots \nu_{i_{n}}(b_{n}) a_n\in C(\wreathH)$ and satisfying the conditions of Theorem \ref{propreduced} will be called a \textit{reduced operator}.
\begin{proof}
Consider the state $\widetilde{\omega}:=\omega\circ\lambda_N\circ\mu=h_H \ot h_{S_N^+}\circ E\circ\lambda_N\circ\mu\in C(\wreathH)^*$, where $\mu:=\pi^{-1}\,:\,C(\wreathH)\rightarrow\mathcal{A}$ has been constructed in the proof of Proposition \ref{PropFull}. Let $\mathcal{C}\subset\mathcal{A}$ be the linear span of $C(S_N^+)$, $\nu(C(H))$ and all reduced operators in $C(\wreathH)$. Note that $\mathcal{C}$ is dense in $\mathcal{A}$. By construction, the state $\widetilde{\omega}$ satisfies $\widetilde{\omega}\vert_{C(S_N^+)}=h_{S_N^+}$, $\widetilde{\omega}\circ\nu=h_H$ and $\widetilde{\omega}(c)=0$ for $c\in C(\wreathH)$ a reduced operator so $\widetilde{\omega}$ satisfies the property of the state $h$ stated in the Theorem hence $h=\widetilde{\omega}$ by density of $\mathcal{C}$. We will show that $\widetilde{\omega}$ is $\Delta$-invariant, which will imply that it is the Haar state, and thus (since $E$ is GNS faithful and $h_H\ot h_{S_N^+}$ is faithful on $C_r(H)\ot C_r(S_N^+)$) that $A$ is isomorphic to the algebra obtained through the GNS-construction applied to the Haar state, namely the reduced C*-algebra $C^\ast_r(\wreathH)$, which will complete the proof of the Theorem. It is enough to show that $\widetilde{\omega}$ is $\Delta$ invariant when restricted to $\mathcal{C}$. An element in $\mathcal{C}$ is the sum of an element $x_0\in C(S_N^+)\ot C(H)$ and elements of the form $x = a_0 \nu_{i_1} (b_1) a_1 \dots a_{n-1} \nu_{i_n}(b_n) a_n$ with $n\geq 1$, $a_k\in C(\SN)$, $b_k \in C(G)^\circ$, where $C(G)^\circ:=\lbrace b \in C(G),\, E_H(b)=0\rbrace$, for all $0\leq k\leq N$, and if there is $k$ such that $i_k = i_{k+1}$, then $E_{i_k}(a_k) = 0$. It suffices to show $\Delta(x)\in\mathcal{C}\odot C(\wreathH)$ (where $\odot$ is the algebraic tensor product). Using the Sweedler notation,
{\small\begin{eqnarray*}
\Delta(x) &=& \sum (a_0)_{(1)} \nu_{i_1}((b_1)_{(1)})u_{i_1,s_1}(a_1)_{(1)} \nu_{i_2}((b_2)_{(1)})u_{i_2,s_2}\dots (a_n)_{(1)}\\
 & &\ot (a_0)_{(2)} \nu_{s_1}((b_1)_{(2)})(a_1)_{(1)} \nu_{s_2}((b_2)_{(2)})\dots (a_n)_{(2)}.
\end{eqnarray*}}
By Remark \ref{RmkECH} one has $E_H((b_k)_{(1)})=0$ hence, the left part of every term of the sum is a reduced operator whenever there is no $k$ such that $i_k = i_{k+1}$ and $E_{i_k}((a_{k})_{(1)}) \neq 0$, we need only to take care of the remaining cases. Assume that we have $k$ such that $i_k=i_{k+1}$, and $E_{i_k}((a_{k})_{(1)}) \neq 0$, with $k$ being the smallest such integer. Then we can write $(a_{k})_{(1)} =  E_{i_k}((a_{k})_{(1)}) + (a_{k})_{(1)}^\circ$, with $ E_{i_k}((a_{k})_{(1)}^\circ)=0$. Fixing $s =(s_1,s_2,\dots,s_n) \in N^n$, we can write
{\small\begin{align*}
    (a_0)_{(1)} &\nu_{i_1}((b_1)_{(1)})u_{i_1,s_1}(a_1)_{(1)}\dots (a_n)_{(1)}\otimes (a_0)_{(2)} \nu_{s_1}((b_1)_{(2)})(a_1)_{(1)} \nu_{s_2}((b_2)_{(2)})\dots (a_n)_{(2)}\\
    = & (a_0)_{(1)} \nu_{i_1}((b_1)_{(1)})u_{i_1,s_1}(a_1)_{(1)}  \dots (a_k)_{(1)}^\circ \dots (a_n)_{(1)} \ot (a_0)_{(2)} \dots (a_n)_{(2)}\\
     &+(a_0)_{(1)} \nu_{i_1}((b_1)_{(1)})u_{i_1,s_1}(a_1)_{(1)} \dots E_{i_k}((a_k)_{(1)}) \dots (a_n)_{(1)} \ot (a_0)_{(2)} \dots (a_n)_{(2)},
\end{align*}}
we see that the first term of the decomposition is itself reduced up to the $k$-th term, and we can iterate the process with the next term such that $i_{k'}=i_{k'+1}$. In the end we get a sum of tensor such that the first term is always reduced, and we only have to show that the second vanishes. The second term is always of the form
$$
    \alpha_{k,s} = (a_0)_{(1)} \nu_{i_1}((b_1)_{(1)})u_{i_1,s_1}(a_1)_{(1)} \dots E_{i_k}((a_k)_{(1)}) \dots (a_n)_{(1)} \ot (a_0)_{(2)} \dots (a_n)_{(2)},
$$
with maybe other conditional expectations appearing after rank $k$. The definition of the conditional expectations $(E_i)$ gives $E_{i_k}((a_k)_{(1)}) = N \sum_{t=1}^N  h((a_k)_{(1)}u_{i_k t})u_{i_k t}$ hence,
{\small\begin{eqnarray*}
\alpha_{k,s}&=&N \sum_{t=1}^N h((a_k)_{(1)}u_{i_k t})(a_0)_{(1)} \nu_{i_1}((b_1)_{(1)})u_{i_1,s_1}(a_1)_{(1)}\dots\nu_{i_k}((b_k)_{(1)})u_{i_k,t}\nu_{i_k}((b_{k+1})_{(1)}) \dots (a_n)_{(1)}\\
& & \ot (a_0)_{(2)} \dots (a_n)_{(2)}\\
&= & N \sum_{t=1}^N h((a_k)_{(1)}u_{i_k t})(a_0)_{(1)} \nu_{i_1}((b_1)_{(1)})u_{i_1,s_1}(a_1)_{(1)} \dots \nu_{i_k}((b_k)_{(1)}(b_{k+1})_{(1)})u_{i_k,t} \dots (a_n)_{(1)}\\
& & \ot (a_0)_{(2)} \dots (a_n)_{(2)}\\
&=& N \sum_{t=1}^N h((a_k)_{(1)'}u_{i_k t})(a_0)_{(1)} \nu_{i_1}((b_1)_{(1)})u_{i_1,s_1}(a_1)_{(1)} \dots \nu_{i_k}((b_k)_{(1)}(b_{k+1})_{(1)})u_{i_k,t} \dots (a_n)_{(1)}\\
& &    \ot (a_0)_{(2)}\dots (a_k)_{(1)'} \dots (a_n)_{(2)}
\end{eqnarray*}}

\noindent The symbols $(1)'$ and $(2)'$ are used to differentiate between the summation indices of the scalars we will now move to the right side of the tensor product and the others summation indices:
     {\small\begin{align*}
     \alpha_{k,s} =
        & N \sum_{t=1}^N(a_0)_{(1)} \nu_{i_1}((b_1)_{(1)})u_{i_1,s_1}(a_1)_{(1)} \dots \nu_{i_k}((b_k)_{(1)}(b_{k+1})_{(1)})u_{i_k,t} \dots (a_n)_{(1)}\\
        &    \ot (a_0)_{(2)} \dots h((a_k)_{(1)'}u_{i_k t})(a_k)_{(2)'}\dots (a_n)_{(2)}.\\
\end{align*}}

\noindent Now, by Lemma \ref{lemortho2}, we get $\sum_{(1)',(2)'} h((a_k)_{(1)'}u_{i_k t})(a_k)_{(2)'} = 0$. This is true for every value of $t$, hence $\alpha_{k}$ vanishes whenever we considered $k$ as before. This proves that $\widetilde{\omega}$ is $\Delta$-invariant.\end{proof}

\noindent The explicit computation of the Haar state can be used to show the following property of dual quantum subgroups in free wreath products.

\begin{proposition}\label{quantumsub}
Let $C(H)\subset C(G)$ be a dual quantum subgroup and consider, by the universal property of $C(H\wr_* S_N^+)$, the unique unital $*$-homomorphism $\iota\,:\,C(H\wr_* S_N^+)\rightarrow C(\wreath)$ such that $\iota\circ\nu_i=\nu_i\,:\, C(H)\rightarrow C(\wreath)$ for all $1\leq i\leq N$ and $\iota\vert_{C(S_N^+)}$ is the inclusion $C(S_N^+)\subset C(\wreath)$. Then, $\iota$ is faithful so $H\wr_*S_N^+$ is a dual quantum subgroup of $G\wr_*S_N^+$.
\end{proposition}

\begin{proof}
It is clear that $\iota$ intertwines the comultiplications and, by Theorem \ref{propreduced}, it also intertwines the Haar states so the restriction $\iota\vert_{\Pol(H\w \SN)} : \Pol(H\w \SN) \rightarrow \Pol (\wreath)$ is injective (by faithfulness of the Haar states) and still intertwines the comultiplication. Hence,  $H\wr_*S_N^+$ is a dual quantum subgroup of $G\wr_*S_N^+$ and $\iota$ itself faithful by Proposition \ref{dualqsg}.\end{proof}

\subsection{The von Neumann version}

Consider the graph of von Neumann algebras over $\T_N$ given by $M_{p_0}=\Linf(H)\ot \Linf(\SN)$, $M_{p_k}:=\rL^\infty(G)\otimes\C^N$, $N_{e_k}=N_{\ebar_k}=\Linf(H) \ot \C^N$ for all $1\leq k\leq N$ with source map $s_{e_k}\,:\,\Linf(H)\ot \C^N\rightarrow \Linf(H) \ot L_k\subset \Linf(H) \ot \Linf(\SN)$, $h\ot e_j\mapsto h\ot u_{kj}$ and range map $r_{e_k}\,:\,\Linf(H) \ot \C^N\rightarrow \Linf(G)\ot\C^N$ being the canonical inclusion. We also consider family of faithful normal states $\omega_{p_0}=h_H\ot h_{S_N^+}$, $\omega_{p_k}=h_G\ot\tr$ (where $\tr$ is the normalized uniform trace on $\C^N$ i.e. $\tr(e_j)=\frac{1}{N}$) and $\omega_{e_k}=h_H\ot \tr$ for all $1\leq k\leq N$. Hence, we get a graph of von Neumann algebras as defined in \cite[Definition A.1]{FF14}. Let us denote by $M$ the fundamental von Neumann algebra at $p_0$ and view $M_{p_0}=\Linf(H) \ot \Linf(S_N^+)\subset M$. Let $\varphi$ the associated fundamental normal faithful state which is a trace if and only if $G$ is Kac from the results of \cite{FF14}, as $H$ is automatically Kac in that case. As a direct consequence of Proposition \ref{propreduced} and \cite[Section A.5]{FF14} we have the following. Note that, by Theorem \ref{propreduced}, the unital faithful $*$-homomorphism $\nu_i\,:\,C(G)\rightarrow C(\wreathH)$ preserves the Haar states hence, it induces a unital normal faithful $*$-homomorphism $\nu_i\,:\,\Linf(G)\rightarrow\Linf(\wreathH)$.

\begin{proposition}\label{propvn}
There exists a unique unital normal $*$-isomorphism $\pi_r''\,:\,M\rightarrow \Linf(\wreathH)$ extending the isomorphism $\pi_r\,:\, A\rightarrow C_r(\wreathH)$ from Theorem \ref{propreduced}. Moreover, $h\circ\pi_r''=\varphi$, where $h$ also denotes the Haar state on ${\rm L}^\infty(\wreathH)$.
\end{proposition}

\noindent Recall that $\mathrm{Aut}_{\Linf(H)}(\Linf(G),h_G) = \lbrace \beta \in \mathrm{Aut}(\Linf(G),h_G),\,\beta(\Linf(H)) = \Linf(H)\rbrace$.

\begin{proposition}
For all $\alpha\in \mathrm{Aut}_{\Linf(H)}(\Linf(G),h_G)$ there exists a unique $h$-preserving automorphism $\psi(\alpha)$ of $\Linf(\wreathH)$ such that
\begin{equation}\label{MorphismPsi}
\psi(\alpha)\vert_{\Linf(S_N^+)}=\id\quad\text{and}\quad \psi(\alpha)\circ\nu_i=\nu_i\circ\alpha\quad\forall1\leq i\leq N.
\end{equation}
Moreover, $\psi\,:\,{\rm Aut}_{\Linf(H)}(\Linf(G),h_G)\rightarrow{\rm Aut}(\Linf(\wreathH),h)$ is a continuous group homomorphism.
\end{proposition}

\begin{proof}
Let $\alpha\in\mathrm{Aut}_{\Linf(H)}(\Linf(G),h_G)$ and write $({\rm L}^2(\wreathH),\lambda,\xi)$ the GNS of the Haar state. To show that $\psi(\alpha)$ exists, it suffices to show that there exists a well defined unitary $U_\alpha\in\mathcal{B}({\rm L}^2(\wreathH))$ such that, for all reduced operator $a_0 \nu_{i_1}(b_1) \dots \nu_{i_{n}}(b_{n}) a_n\in \Linf(\wreathH)$,
$$U_\alpha(a_0 \nu_{i_1}(b_1)\dots \nu_{i_{n}}(b_{n}) a_n\xi)=a_0 \nu_{i_1}(\alpha(b_1))\dots \nu_{i_{n}}(\alpha(b_{n})) a_n\xi$$
Indeed, if such a unitary is constructed, then $\psi(\alpha)(x):=U_\alpha x U_\alpha^*$ does the job. To show that such a $U_\alpha$ exists it suffices to check that, for any reduced operator $x=a_0 \nu_{i_1}(b_1) \dots \nu_{i_{n}}(b_{n}) a_n\in \Linf(\wreathH)$ one has $h(x^*x)=h(y^*y)$, where $y:=a_0 \nu_{i_1}(\alpha(b_1))\dots \nu_{i_{n}}(\alpha(b_{n})) a_n$. This can be shown by induction on $n$, by first observing that $E_H\circ\alpha=\alpha\circ E_H$, where $E_H\,:\,\Linf(G)\rightarrow\Linf(H)$ is the canonical normal faithful conditional expectation. Indeed, since $\alpha^{-1}\circ E_H\circ\alpha$ is a normal $h$-preserving conditional expectation onto $\Linf(H)$, we have, by uniqueness in Takesaki's Theorem, that $\alpha^{-1}\circ E_H\circ\alpha=E_H$. It follows that the element $y$ is reduced whenever $x$ is. Then, one can easily use the Haar state formula in Theorem \ref{propreduced} to prove our claim by induction. The uniqueness of $\psi(\alpha)$ and the fact that it preserves the Haar state are clear.

\vspace{0.2cm}

\noindent The fact that it is a group homomorphism follows by uniqueness. To prove continuity, it suffices to check that the map $\psi_x\,:\,{\rm Aut}_{\Linf(H)}(\Linf(G),h_G))\rightarrow{\rm L}^2(\wreathH,h)$, $(\alpha\mapsto\psi(\alpha)(x)\xi)$, is continuous for all $x\in\Linf(\wreathH)$, where $\xi\in{\rm L}^2(\wreathH,h)$ is the canonical cyclic vector. Define $\mathcal{C}\subset\Linf(\wreath)$ to be the linear span of $\Linf(S_N^+)$ and the reduced operators and observe that $\mathcal{C}$ is a $\sigma$-strongly dense unital $*$-subalgebra of $\Linf(\wreathH)$. Note that the subset:
$$\{x\in\Linf(\wreathH)\,:\,\psi_x\text{ is continuous }\}\subseteq \Linf(\wreathH)$$
is a clearly a subspace and it is $\sigma$-strongly closed since, $\psi(\alpha)$ being $h$-preserving, we have, $$\Vert\psi_x(\alpha)-\psi_y(\alpha)\Vert^2=\Vert\psi(\alpha)(x)\xi-\psi(\alpha)(y)\xi\Vert^2=\Vert\psi(\alpha)(x-y)\xi\Vert^2=h((x-y)^*(x-y))$$
for all $\alpha\in{\rm Aut}_{\Linf(H)}(\Linf(G),h_G)$ and $x,y\in\Linf(\wreathH)$. Hence, it suffices to show that $\psi_x$ is continuous for all $x\in\Linf(S_N^+)$ or for $x$ a reduced operator. When $x\in\Linf(S_N^+)$ the map $\psi_x$ is the constant map equals to $x\xi$. When $x=a_0 \nu_{i_1}(b_1) \dots \nu_{i_{n}}(b_{n}) a_n\in\Linf(\wreathH)$ is a reduced operator, we have, by induction on $n$, $\psi_x(\alpha)-\psi_x({\rm id})=\sum_{k=1}^n\varphi_k(\alpha)$, where:
$$\varphi_k(\alpha):=a_0\nu_{i_1}(b_1)a_1\dots\nu_{i_{k-1}}(b_k)a_k\nu_{i_k}(\alpha(b_k)-b_k)a_k\nu_{i_{k+1}}(\alpha(b_{k+1}))a_{k+1}\dots\nu_{i_n}(\alpha(b_n))a_n\xi$$
with the natural conventions so that $\varphi_0(\alpha)$ and $\varphi_n(\alpha)$ make sense.
Using the formula for the Haar state given in Theorem \ref{propreduced} and the fact that $\alpha$ is $h_G$ preserving, a direct computation gives:
$$\Vert\varphi_k(\alpha)\Vert^2=\left(\prod_{l=0}^n\Vert a_l\Vert_{2,S_N^+}^2\right)\left(\prod_{1\leq l\leq n,l\neq k}\Vert b_l\Vert_{2,G}^2\right)\Vert\alpha(b_k)-b_k\Vert_{2,G}^2$$
where, for a CQG $G$, $\Vert\cdot\Vert_{2,G}$ is the $2$-norm given by the Haar state on $\Linf(G)$. It follows that if $\alpha_s\rightarrow_s{\rm id}$ in ${\rm Aut}_{\Linf(H)}(\Linf(G),h_G)$ then, for all $1\leq k\leq n$, $\varphi_k(\alpha_s)\rightarrow_s 0$ in ${\rm L}^2(\wreathH)$ for the norm. Hence, $\psi_x$ is continuous at ${\rm id}$ so it is continuous since it is a group homomorphism.\end{proof}

\subsection{An inductive construction for free wreath products algebras}

We define inductively the C*-algebra $\mathcal{A}_i$ ($1\leq i\leq N$) containing $C(H)\ot C(S_N^+)$ with a state $\omega_i\in \mathcal{A}_i^*$ by $\mathcal{A}_0=C(H)\ot C(S_N^+)$ and $\omega_0=h_H\ot h_{S_N^+}$ the Haar state on $\mathcal{A}_0$. Let $A_0:=C_r(H)\ot C_r(S_N^+)$ be the GNS-construction of $\omega_0$ and $M_0:=\Linf(H)\ot \Linf(S_N^+)=A_0''$ the von Neumann algebra generated in the GNS construction. Write $\lambda_0\,:\,\mathcal{A}_0\rightarrow A_0$ the GNS morphism. We still denote by $\omega_0$ the associated faithful state on $A_0$ (resp. faithful normal state on $M_0$). Suppose that $(\mathcal{A}_i,\omega_i)$ is constructed and let $A_i$ be the C*-algebra of the GNS construction of $\omega$ with GNS morphism $\lambda_i\,:\,\mathcal{A}_i\rightarrow A_i$ and $M_i=A_i''$ the von Neumann algebra generated in the GNS construction. Define the full free product $\mathcal{A}_{i+1}:=(C(G)\ot \C^N)\underset{C(H)\ot \C^N}{*}\mathcal{A}_i$, where the amalgamation is with respect to the faithful normal unital $*$-homomorphisms $r_{e_{i+1}}\,:\,C(H)\ot \C^N\rightarrow C(G)\ot\C^N$ and $s_{e_{i+1}}\,:\,C(H)\ot \C^N\mapsto C(H)\ot L_{i+1}\subset C(H)\ot C(S_N^+)\subset \mathcal{A}_i$. Let $A_{i+1}:=(C_r(G)\ot\C^N,h_G\ot{\rm tr})\underset{C_r(H)\ot \C^N}{*}(A_i,\omega_i)$ be the reduced free product with amalgamation with respect to the same maps but at the reduced level. Let $\lambda_{i+1}\,:\,\mathcal{A}_{i+1}\rightarrow A_i$ the unique surjective unital $*$-homomorphism such that $\lambda_{i+1}(a\ot x)=\lambda_G(a)\ot x$ for all $a\in C(G)$, $x\in \C^N$ and $\lambda_{i+1}\vert_{\mathcal{A}_i}=\lambda_i$. Define $\omega_{i+1}:=((h_G\ot{\rm tr})*\omega_i)\circ\lambda_{i+1}$ and note that $A_{i+1}$ is the C*-algebra of the GNS construction of $\omega_{i+1}$  and $\lambda_{i+1}\,:\,\mathcal{A}_{i+1}\rightarrow A_{i+1}$ is the GNS morphism. It follows that $M_{i+1}=A_{i+1}''=(\Linf(G)\ot \C^N)\underset{\Linf(H) \ot \C^N}{*}M_i$, where the amalgamated free product von Neumann algebra is with respect to the faithful normal unital $*$-homomorphisms $r_{e_i}\,:\,\Linf(H)\ot \C^N\rightarrow \Linf(G)\ot\C^N$ and $s_{e_i}\,:\,\Linf(H)\ot \C^N\mapsto \Linf(H)\ot L_{i+1}\subset \Linf(H)\ot \Linf(S_N^+)\subset M_i$  and with respect to the faithful normal states $h_G\ot{\rm tr}$ and $\omega_i$. We still denote by $\omega_{i+1}$ the faithful free product state on $A_{i+1}$ and the faithful normal free product state on $M_{i+1}$.

\begin{proposition}\label{InductiveConstruction}
There exists unique state preserving $*$-isomorphisms $$\rho\,:\, (C(\wreathH),h)\rightarrow(\mathcal{A}_N,\omega_N)\quad\text{and}\quad \rho_r\,:\, (C_r(\wreathH),h)\rightarrow(A_N,\omega_N)$$ and a normal $*$-isomorphism $\rho_r''\,:\,(\Linf(\wreathH),h)\rightarrow(M_N,\omega_N)$, where $h$ the Haar state on $\wreathH$, such that, writing $\mathcal{A}_N=\cup_i^\uparrow\mathcal{A}_i$, $\rho$ maps $C(S_N^+)$ onto $\mathcal{A}_0=C(H)\ot C(S_N^+)$ via $x\mapsto 1\ot x$ and, for all $1\leq i\leq N$, $\rho(\nu_i(a))=a\ot 1\in C(G)\otimes\C^N\subset\mathcal{A}_i=(C(G)\ot \C^N)\underset{C(H)\ot \C^N=C(H)\ot L_i}{*}\mathcal{A}_{i-1}$ and, $\lambda\circ\rho=\rho_r\circ\lambda_N$, $\lambda\circ\rho_r''=\rho_r''\circ\lambda_N$, where $\lambda\,:\,C(\wreathH)\rightarrow C_r(\wreathH)$ is the canonical surjection.
\end{proposition}

\begin{proof}
Uniqueness being obvious, it suffices to show the existence and that $\rho$ intertwines the states. The existence of $\rho_r$ and $\rho_r''$ follows from that. It is easy to see, by the universal property of $C(\wreathH)=C(G)^{*_H N}*C(S_N^+)/I$ that there exists a unital $*$-homomorphism $\rho$ satisfying the conditions of the statement. To show that $\rho$ is an isomorphism, we construct an inverse $\rho'\,:\,\mathcal{A}_N\rightarrow C(\wreathH)$. To do so, we construct inductively unital $*$-homomorphisms $\rho'_i\,:\,\mathcal{A}_i\rightarrow C(\wreathH)$ by letting $\rho'_0\,:\,\mathcal{A}_0=C(H)\ot C(S_N^+)\rightarrow C(\wreathH)$, $a\ot b\mapsto\nu(a)b$ and, if $\rho'_{i-1}\,:\,\mathcal{A}_{i-1}\rightarrow C(\wreathH)$ is defined, we use the universal property of the full amalgamated free product $\mathcal{A}_{i}=(C(G)\ot\C^N)\underset{C(H)\ot \C^N=C(H) \ot L_i}{*}\mathcal{A}_{i-1}$ to define $\rho'_{i}\,:\,\mathcal{A}_{i}\rightarrow C(\wreathH)$ to be the unique unital $*$-homomorphism such that $\rho'_{i}\vert_{\mathcal{A}_{i-1}}=\rho'_{i-1}$ and $\rho_{i}'(a\ot e_j)=\nu_i(a)u_{ij}$ for all $a\ot e_j\in C(G)\ot\C^N$, where $(e_j)_{1\leq j\leq N}$ is the canonical orthonormal basis of $\C^N$. Then, since $\rho'_i\vert_{\mathcal{A}_{i-1}}=\rho'_{i-1}$, there exists a unique unital $*$-homomorphism $\rho'\,:\,\mathcal{A}_N=\cup_i^\uparrow\mathcal{A}_i\rightarrow C(\wreathH)$ such that $\rho'\vert_{\mathcal{A}_i}=\rho'_i$. It is then easy to check that $\rho'$ is the inverse of $\rho$. It remains to show that $\rho$ intertwines the states. By the uniqueness property of the state $h$ stated in Theorem \ref{propreduced}, it suffices to show that $\omega_N(\rho(c))=0$ for  $c\in C(\wreathH)$ a reduced operator in the sense of Theorem \ref{propreduced}. So let $c=a_0\nu_{i_1}(b_1)a_1\dots\nu_{i_n}(b_n)a_n\in C(\wreathH)$ be a reduced operator and note that, if at least one the $i_k$ is equal to $N$ then $\rho(c)$ is a word with letters alternating from $C(G)\ot\C^N\ominus C(H)\ot \C^N$, and $\mathcal{A}_{N-1}\ominus C(H)\ot L_N$ (where, when, for $C\subset B$ with conditional expectation, $B\ominus C$ denotes the kernel of the conditional expectation onto $C$) so, by definition of the free product state, $\omega_N(\rho(c))=0$. If for all $k$, $i_k\leq N-1$, then we can view $\rho(c)\in\mathcal{A}_{N-1}$ and since $\omega_N\vert_{\mathcal{A}_{N-1}}=\omega_{N-1}$, we can repeat the argument inductively and eventually deduce that $\omega_N(\rho(c))=0$.\end{proof}

\noindent In the setting of von Neumann algebras, the inductive construction implies the following.

\begin{proposition}\label{PropModular}
The following holds for $\wreathH$, with $\psi$ defined in Equation $(\ref{MorphismPsi})$,
\begin{enumerate}
\item The modular group of the Haar state $h$ is $\sigma_t:=\psi(\sigma_t^G)$, for all $t\in\R$, where $(\sigma_t^G)_{t\in\R}$ is the modular group of the Haar state of $G$ on ${\rm L}^\infty(G)$.
    \item The scaling group is $\tau_t:=\psi(\tau_t^G)$, $\forall t\in\R$, where $\tau_t^G$ is the scaling group of $G$.
\item$T(\wreath)=\{t\in\R\,:\,\tau_t^G=\id\}$ $\forall N\geq 2$ and $N\neq 3$ and $\tau(\wreathH)=\tau(G)$ $\forall N\in\N^*$.
\end{enumerate}
\end{proposition}

\begin{proof}
\noindent$(1)$. We first note that, for all $t\in\R$, $\sigma_t^G\in{\rm Aut}_{\Linf(H)}(\Linf(G),h_G)$. From \cite[Theorem 2.6]{Ue99}, the modular group of the free product state $\omega_1$ on the amalgamated free product $M_1=(\Linf(G)\ot \C^N)\underset{\Linf(H)\ot\C^N=\Linf(H)\ot L_1}{*}(\Linf(H)\ot\Linf(S_N^+))$ satisfies $\sigma_t^{\omega_1}\vert_{\Linf(G)}=\sigma_t^{h_{S_N^+}}=\id$ and $\sigma_t(x\ot 1)=\sigma_t^G(x)\ot 1$ for all $x\in\Linf(G)$. Identifying $(M_N,\omega_N)\simeq(\Linf(\wreathH),h)$ with the isomorphism of Proposition \ref{InductiveConstruction} and applying inductively \cite[Theorem 2.6]{Ue99} to our finite sequence of von Neumann algebra $(M_i)_{1\leq i\leq N}$ gives the result.

\vspace{0.2cm}

\noindent$(2)$.  Note that, for all $t\in\R$, $\tau_t^G\in{\rm Aut}_{\Linf(H)}(\Linf(G),h_G)$. To show that $\tau_t:=\psi(\tau_t^G)$ is the scaling group we have to show that $\Delta\sigma_t=(\tau_t\ot\sigma_t)\Delta$. Using the properties of $\sigma_t$, $\tau_t$ and the definition of $\Delta$ we have, for all $1\leq i\leq N$ and all $a\in \Pol(G)$,
\begin{eqnarray*}
\Delta(\sigma_t(\nu_i(a)))&=&\Delta(\nu_i(\sigma_t^G(a)))=\sum_{j=1}^N(\nu_i\ot\nu_j)(\Delta_G(\sigma_t^G(a)))(u_{ij}\ot 1)\\
&=&\sum_{j=1}^N(\nu_i\ot\nu_j)((\tau^G_t\ot\sigma^G_t)\Delta_G(a))(u_{ij}\ot 1)\\
&=&\sum_{j=1}^N(\tau_t\ot\sigma_t)(\nu_i\ot\nu_j)(\Delta_G(a))(u_{ij}\ot 1)=(\tau_t\ot\sigma_t)\Delta(\nu_i(a)),
\end{eqnarray*}
where, in the last equality, we used that $\tau_t(u_{ij})=u_{ij}$. Since the modular group and scaling group act as the identity on $\Linf(S_N^+)$, the equality $\Delta\sigma_t(u_{ij})=(\tau_t\ot\sigma_t)\Delta(u_{ij})$ is clear.

\vspace{0.2cm}

\noindent$(3)$. Let us compute the $T$-invariant in the non-amalgamated case. For $N\geq 4$, the irreducible representations of $\wreath$ are completely classified in \cite{FP16} and it follows that the only dimension $1$ irreducible representation is the trivial representation. Hence, $T(\wreath)=\{t\in\R\,:\,\tau_t=\id\}$ and we conclude the proof using assertion $(1)$. For $N=2$, we know from Proposition \ref{Prop-Subgroup} that $G\wr_*S_2^+\simeq G^{* 2}\rtimes S_2$ and the $T$-invariant of such quantum groups is computed in Proposition \ref{Prop-Semidirect}. Using also Proposition \ref{Prop-Freeprod} concludes the computation.

\vspace{0.2cm}

\noindent To compute the $\tau$-invariant, we may and will assume that $N\geq 2$. Since $\psi$ is continuous, one has $\tau(\wreathH)\subseteq\tau(G)$. Let us show that $\tau(G)\subseteq\tau(\wreathH)$. Applying Remark \ref{RmkRestriction} with $M=\Linf(\wreathH)=\left(\Linf(G)\ot\C^N\right)\underset{\Linf(H)\ot\C^N=\Linf(H)\ot L_N}{*}M_{N-1}$, $\omega$ the Haar state and $A=\Linf(G)\ot\C^N$ we see that the map $\R\rightarrow{\rm Aut}(\Linf(G))$, $t\mapsto\sigma_t\vert_{A}=\sigma_t^G\ot\id$ is $\tau(\wreathH)$-continuous. Hence, $(t\mapsto\sigma_t^G)$ is $\tau(\wreathH)$-continuous so $\tau(G)\subseteq\tau(\wreathH)$.\end{proof}

\section{Approximation properties}

\noindent Theorem \ref{THMA} is a consequence of the following more general statement.

\begin{theorem}\label{THMAprime}
For $G$ a CQG with $H$ a dual quantum subgroup and $N\geq 1$, the following holds
\begin{enumerate}
    \item $\Ghat$ is exact if and only if $\wreathHhat$ is exact.
    \item If $\Irr (H)$ is finite, then $\Ghat$ has the Haagerup property if and only if $\wreathHhat$ has the Haagerup property.
    \item If $G$ is Kac and $H$ is co-amenable then $\Ghat$ is hyperlinear if and only if $\wreathHhat$ is hyperlinear.
    \item
    \begin{itemize}\item If $N\geq 5$, $\wreathH$ is not co-amenable for all $G$ and all $H$.
    \item If $N\in\{3,4\}$ and $H$ is a proper dual quantum subgroup of $G$ (\ref{DefProper}) then $\wreathH$ is not co-amenable whenever $G$ is Kac.
    \item $G\wr_{*,H}S_{2}^+$ is co-amenable if and only if $G^{*_H 2}$ is co-amenable.
    \end{itemize}
    \item $\Ghat$ is K-amenable if and only if $\wreathHhat$ is K-amenable.
\end{enumerate}
\end{theorem}

\begin{proof}$(1).$ It is known from \cite{Br13} that $C_r(S_N^+)$ is exact. Hence, the result follows from \cite[Corollary 3.30]{FF14}.

\vspace{0.2cm}

\noindent$(2)$. Note that, at $N\leq 4$, $S_N^+$ is a co-amenable Kac CQG \cite{Ba99} so that $\Linf(S_N^+)$ has the Haagerup property and, at $N\geq 5$, $\Linf(S_N^+)$ has the Haagerup property by \cite{Br13}. We show the result by using the inductive construction of $\Linf(\wreathH)=M_N$, with $M_0=\Linf(H)\ot \Linf(S_N^+)$ and $M_{i+1}=(\Linf(G)\ot\C^N)\underset{\Linf(H)\ot\C^N}{*}M_i$. At $i=0$, $M_0=\Linf(H)\ot\Linf(S_N^+)$ has the Haagerup property, because $\Linf(H)$ is finite-dimensional. Assume that $M_i$ has the Haagerup property. It follows from \cite[Corollary 8.2]{CKSVW23} that $M_{i+1}=(\Linf(G)\ot\C^N)\underset{\Linf(H)\ot\C^N}{*}M_i$ also has the Haagerup property whenever $\Ghat$ has the Haagerup property. By induction, $M_N=\Linf(\wreathH)$ has the Haagerup property. Suppose now that $\wreathHhat$ has the Haagerup property. It is easy to check, using the definition of Haagerup property of \cite{CS15} with respect to a faithful normal state (f.n.s.) that if $(M,\omega)$, $\omega\in M_*$ a f.n.s., has the Haagerup property and $P\subset M$ is a von Neumann subalgebra with a normal and state-invariant conditional expectation $E\,:\, M\rightarrow P$ then, $(P,\omega_{\vert P})$ also has the Haagerup property. Hence, since by construction of amalgamated free products we do have such conditional expectations on each leg of the amalgamated free product we first deduce from  the inductive construction state-preserving isomorphism $$(\Linf(\wreathH),h)\simeq\left((\Linf(G)\ot\C^N)\underset{\Linf(H)\ot \C^N}{*}M_{N-1},\omega_N\right)$$ that $\Linf(G)\ot\C^N$ has the Haagerup property which implies that $\Linf(G)$ also has it.

\vspace{0.2cm}

\noindent$(3)$. By Proposition \cite[Corollary 3.7]{BCF20}, $\widehat{S_N^+}$ is hyperlinear for all $N\in\N^*$. Assuming that $G$ is Kac (so that $\wreathH$ also is), and that $H$ is co-amenable so that $\Linf(H)$ is amenable, we can apply the strategy of the proof of $(2)$ by using \cite[Corollary 4.5]{BDJ08} to deduce the fact that if $\Ghat$ is hyperlinear then so is $\wreathHhat$. The converse is clear.

\vspace{0.2cm}

\noindent$(4)$. The case $N=2$ is a consequence of Propositions \ref{Prop-Subgroup} and \ref{Prop-Semidirect}. Suppose that $N\geq 3$ and $\wreathH$ is co-amenable. Since $G^{*_HN}$ and $S_N^+$ are both compact quantum subgroup of $\wreathH$ by Remark \ref{RmkSubgroup}, it follows that $G^{*_HN}$ and $S_N^+$ are both co-amenable \cite{Sa10}. It implies that $N\leq 4$, since $S_N^+$ is not co-amenable for all $N\geq 5$ \cite{Ba99}. Moreover, since $N\geq 3$, $G^{*_H N}$ is not co-amenable whenever $H$ is proper by Proposition \ref{PropProper} since $H$ has infinite index in $G*_HG$.

\vspace{0.2cm}

\noindent(5). By \cite{Vo17} $\widehat{S_N^+}$ is K-amenable hence, if $\Ghat$ is K-amenable, so is $\widehat{H}$ and, a direct application of \cite[Theorem 5.1 and 5.2]{FG18}, implies that $\wreathHhat$ is K-amenable. Let us prove the converse. To ease the notations we write $\mathcal{A}=C(\wreathH)$ and $A=C_r(\wreathH)$ during this proof. Let $\lambda\,:\, \mathcal{A}\rightarrow A$ be the canonical surjection and assume that $\wreathHhat$ is K-amenable i.e. there exists $\alpha\in KK(A, \C)$ such that $[\lambda]\ot\alpha=[\varepsilon]\in KK(\A,\C)$, where $\ot$ denotes here the Kasparov product and $\varepsilon\,:\,\mathcal{A}\rightarrow\C$ the counit of the free wreath product. Let us show that $\widehat{G^{*_H N}}$ is K-amenable (and so $\widehat{G}$ also is). Consider the canonical inclusion $\pi\,:\,C(G)^{*_H N}\rightarrow \A$. Note that $\pi$ does not intertwine the comultiplications so that we can not deduce the $K$-amenability of $\widehat{G^{*_H N}}$ by using the stability of $K$-amenability by dual quantum subgroup. However, by Theorem \ref{propreduced}, the canonical inclusion $\pi\,:\,C(G)^{*_H N}\rightarrow \A$ intertwines the Haar states hence, there exists a unital $*$-homomorphism $\pi_r\,:\,C_r(G^{*_H N})\rightarrow A$ such that $\pi_r\circ\lambda_G=\lambda\circ\pi$, where $\lambda_G\,:\,C(G^{*_HN})\rightarrow C_r(G^{*_HN})$ is the canonical surjection. Define $\beta:=[\pi_r]\ot\alpha\in KK(C_r(G^{*_HN}),\C)$. One has:
\begin{eqnarray*}
[\lambda_G]\ot\beta&=&[\lambda_G]\ot[\pi_r]\ot\alpha=[\pi_r\circ\lambda_G]\ot\alpha
=[\lambda\circ\pi]\ot\alpha\\
&=&[\pi]\ot[\lambda]\ot\alpha=[\pi]\ot[\varepsilon]
=[\varepsilon\circ\pi]=[\varepsilon_G].
\end{eqnarray*}
This shows that $\widehat{G^{*_HN}}$ is K-amenable.
\end{proof}

\begin{remark}
To deduce Theorem \ref{THMA} it remains to deduce the co-amenability statement for $N=2$ which follows from Proposition \ref{Prop-Freeprod}.
\end{remark}

\begin{remark} As explained in the Introduction, except for $N=2$, the stability of the Haagerup property under the free wreath product construction was open. In the case where $G=\Gammahat$ is the dual of a discrete group, and $H$ is a finite subgroup, then the stability of the Haagerup property and of the weak amenability with constant $1$ has been proved in \cite{Fr22}. In the general but non-amalgamated case, only the stability of the central ACPAP, which is a strengthening of both the Haagerup property and the weak amenability with constant $1$ was known. However, it follows from Proposition \ref{Prop-Semidirect} that $\Lambda_{cb}(\widehat{G\wr_*S_2^+})=\Lambda_{cb}(\Ghat)$. For $N\geq 3$, we believe that weak amenability with constant $1$ is also stable under the free wreath product construction and one obvious way to do the proof would be by induction, using our inductive construction of the von Neumann algebra as well as an amalgamated version (over a finite dimensional subalgebra) of the result of Xu-Ricard \cite{RX06}. Let us mention that for classical wreath products it has been proved by Ozawa \cite{Oz12} that for any non-trivial discrete group $\Lambda$ and any non-amenable discrete group $\Gamma$, the classical wreath product $\Lambda\wr\Gamma$ is not weakly amenable.
\end{remark}

\section{The von Neumann algebra of a free wreath product}

\noindent This section contains the proofs of Theorems \ref{THMB} and \ref{THMC}. Recall that $(\sigma_t^G)_{t\in\R}$ denotes the modular group of the Haar state on $\Linf(G)$.

\vspace{0.2cm}

\noindent Theorem \ref{THMB} is a direct consequence of the following more general statement.

\begin{theorem}\label{THMBprime}
Suppose that $\Irr(G)$ is infinite, $\Irr(H)$ is finite, ${\rm L}^\infty(G)'\cap{\rm L}^\infty(H)=\C1$ and $N\geq 4$. Then ${\rm L}^\infty(\wreathH)$ is a non-amenable full and prime factor without any Cartan subalgebra and the following holds.
    \begin{itemize}
        \item If $G$ is Kac then ${\rm L}^\infty(\wreathH)$ is a type ${\rm II}_1$ factor.
        \item If $G$ is not Kac then ${\rm L}^\infty(\wreathH)$ is a type ${\rm III}_\lambda$ factor for some $\lambda\neq 0$ and:
        $$T({\rm L}^\infty(\wreathH))=\{t\in\R\,:\exists u\in\mathcal{U}(\Linf(H)),\,\,\sigma_t^G={\rm Ad}(u)\}.$$
    \end{itemize}
Moreover, in the non-amalgamated case, we have $\tau(\Linf(\wreath))=\tau(G)$.
\end{theorem}

\noindent\textit{Proof of Theorem \ref{THMB}.} First observe that, by Lemma \ref{LemmaInfinite}, both $\Linf(G)$ and $\Linf(S_N^+)$ are diffuse (since $N\geq 4$). Hence, $\Linf(G)\ot\C^N$ and $\Linf(H)\ot\Linf(S_N^+)$ are also diffuse and the inclusions $\Linf(H)\ot\C^N\subset\Linf(G)\otimes\C^N$ and $\Linf(H)\ot L_k\subset\Linf(H)\ot\Linf(S_N^+)$ are without trivial corner for all $1\leq k\leq N$, in the sense of \cite{HV13} (since $\Linf(H)\ot\C^N$ is finite dimensional hence purely atomic, see \cite[Lemma 5.2]{HV13}).

\vspace{0.2cm}

\noindent We show the result by using the inductive construction of $\Linf(\wreathH)=M_N$, with $M_0=\Linf(H)\ot\Linf(S_N^+)$ and $M_{i+1}=(\Linf(G)\ot\C^N)\underset{\Linf(H)\ot\C^N}{*}M_i$. At $i=0$, we know that $M_0$ and $\Linf(G)\ot\C^N$ are diffuse so the inclusions $B=\Linf(H)\ot\C^N\subset \Linf(G)\ot\C^N$ and $B=\Linf(H)\ot L_1\subset M_0$ are without trivial corner. Moreover, it follows from Lemma \ref{LemmaCommutant} that:
$$M_0'\cap(\Linf(G)\ot\C^N)'\cap B\subset (\Linf(G)'\cap\Linf(H))\ot (\Linf(S_N^+)'\cap L_1)=\C.$$
Hence, we may apply \cite[Theorem E]{HV13} to deduce that $M_1$ is a non-amenable prime factor. In particular $M_1$ is diffuse and a direct induction shows that $M_k$ is a non-amenable prime factor for all $1\leq k\leq N$, in particular, $M_N=\Linf(\wreathH)$ is a non-amenable prime factor.

\vspace{0.2cm}

\noindent Suppose that $A\subset M:=\Linf(\wreathH)=\left(\Linf(G)\ot\C^N\right)\underset{\Linf(H)\ot\C^N}{*}M_{N-1}$ is a Cartan subalgebra. By the previous discussion, $M_{N-1}$ is a non-amenable factor. Hence, it has no amenable direct summand and we may apply \cite[Theorem B]{BHR14} to deduce that $A\prec_M\Linf(H)\ot\C^N$. Since $\Linf(H)\ot\C^N$ is finite dimensional, this contradicts the fact that $A$, being maximal abelian in the diffuse von Neumann algebra $M$, is itself diffuse.

\vspace{0.2cm}

\noindent To see that $\Linf(\wreathH)$ is full, we may now apply \cite[Theorem 4.10]{Ue13} (since $\Linf(G)\ot\C^N$ is diffuse and $\Linf(H)\ot\C^N$ is finite dimensional so the inclusion $\Linf(H)\ot\C^N\subset\Linf(G)\ot\C^N$ is entirely non trivial and, as shown before, $M_{N-1}$ is diffuse).

\vspace{0.2cm}

\noindent When $G$ is Kac, $\wreathH$ is also Kac hence, $\Linf(\wreathH)$ is a ${\rm II}_1$-factor. When $G$ is not Kac, $\Linf(\wreathH)$ is a non-amenable factor by the first part of the proof but not a finite factor by \cite[Theorem 8]{Fi07}. Hence, it is a type ${\rm III}_\lambda$ factor with $\lambda\neq 0$ since it is a full factor \cite[Proposition 3.9]{Co74}. To compute the $T$-invariant, we view again  $M:=\Linf(\wreathH)=\left(\Linf(G)\ot\C^N\right)\underset{\Linf(H)\ot\C^N}{*}M_{N-1}$ and we also note that, since $M_{N-1}$ is diffuse and the amalgam $B:=\Linf(H)\ot L_N=\Linf(H)\ot\C^N$ is finite dimensional, we have $M_{N-1}\nprec_{M}B$. We can now use \cite{HI20} to deduce that $T:=T({\rm L}^\infty(\wreathH))=\{t\in\R\,:\,\exists u\in\mathcal{U}(B)\,\,\,\sigma_t={\rm Ad}(u)\}$,
where $(\sigma_t)_{t\in\R}$ is the modular group of the Haar state on $\Linf(\wreathH)$, which coincides with the free product state. In particular one has $\sigma_t\vert_{\Linf(G)\ot\C^N}=\sigma_t^G\ot\id$. Let $T':=\{t\in\R\,:\,\sigma_t^G=\id\}$. It is clear that $T'\subseteq T$. Let $t\in T$ and $u\in\mathcal{U}(B)$ such that $\sigma_t={\rm Ad}(u)$. Write $u=\sum_{k=1}^Nu_k\ot e_k\in\Linf(H)\ot\C^N$; where $u_k\in\Linf(H)$ is unitary for all $1\leq k\leq N$. For all $b\in\Linf(G)$ and all $1\leq k\leq N$ we find $\sigma_t(b\ot e_k)=\sigma_t^G(b)\ot e_k=u(b\ot e_k)u^*=u_kbu_k^*\ot e_k$. Hence, $\sigma_t^G={\rm Ad}(u_1)$ which implies that $t\in T'$. Note that we could also have computed the $T$-invariant by using the results of \cite{Ue13}. Let us now compute Connes' $\tau$-invariant in the non-amalgamated case. It suffices to show that, for any sequence $(t_n)_n$ of real numbers, one has $t_n\rightarrow 0$ for $\tau(\wreath)$ if and only if $t_n\rightarrow 0$ for $\tau(G)$. Assume that $t_n\rightarrow 0$ for $\tau(\Linf(\wreath))$. Then, viewing again $\Linf(\wreath)=\left(\Linf(G)\ot\C^N\right)\underset{B}{*}M_{N-1}$, where this time $B=\C^N= L_N$, we may apply \cite[Theorem B]{HI20}, to deduce that there exists a sequence of unitaries $v_n\in\mathcal{U}(B)$ such that $\alpha_n:={\rm Ad}(v_n)\circ\sigma_{t_n}\rightarrow \id$. To conclude the proof, we apply the restriction argument from Remark \ref{RmkRestriction} to deduce that $\alpha_n\vert_{\Linf(G)\ot\C^N}\rightarrow\id$ in ${\rm Aut}(\Linf(G)\ot\C^N)$. Note that, since $B$ is in the center of $\Linf(G)\ot\C^N$ one has $\alpha_n\vert_{\Linf(G)\ot\C^N}=\sigma_{t_n}^G\ot\id$ and we deduce that $\sigma_{t_n}^G\rightarrow\id$ i.e. $t_n\rightarrow 0$ in $\tau(G)$. The converse statement follows from the continuity of the group homomorphism $\psi$ defined by Equation $(\ref{MorphismPsi})$.$\hfill\square$

\begin{remark}
The assumption $N\neq 2$ in Theorem \ref{THMB} is necessary. Indeed, when $N=2$, we know from Proposition \ref{Prop-Subgroup} and Section \ref{semi-direct} that $\Linf(G\wr_*S_2^+)=\left(\Linf(G^{* 2})\right)\ot\C^2$. In particular, $\Linf(G\wr_*S_2^+)$ is never a factor. For $N=3$, our proof does not work ($\Linf(S_3^+) \simeq \C^6$ is not diffuse) but we don't know if $\Linf(G\w S_3^+)$ can nonetheless be a factor. However, whenever $\vert\Irr(G)\vert\geq 3$ if $N=2$ or whenever $G$ is non-trivial if $N=3$, Proposition \ref{Prop-Freeprod} shows that $\Linf(G^{*N})$ is a full and prime factor without any Cartan subalgebras and of type ${\rm II}_1$ if $G$ is Kac or of type ${\rm III}_\lambda$, $\lambda\neq 0$, with $T$ invariant given by $\{t\in\R\,:\,\sigma_t^G=\id\}$, if $G$ is not Kac. Finally, let us mention that the assumption $\vert\Irr(G)\vert=\infty$ in Theorem \ref{THMB} does not seems necessary but is useful in our proof.\end{remark}

\noindent Theorem \ref{THMC} is a direct consequence of the following.

\begin{theorem}\label{THMCprime}
Suppose that $\Irr(G)$ is infinite, $\Irr(H)$ is finite and $N\geq 2$. The following holds.
\begin{enumerate}
    \item If ${\rm L}^\infty(G)$ is amenable then, $\forall 1\leq i\leq N$, the von Neumann subalgebra of ${\rm L}^\infty(\wreathH)$ generated by $\{\nu_i(a)u_{ij}\,:\, a\in \Linf(G),\,1\leq j\leq N\}$ is maximal amenable with expectation.
    \item If $G$ is Kac then $\Linf(H)\ot\Linf(S_4^+)\simeq\left(\nu(\Linf(H))\cup\Linf(S_4^+)\right)''\subset {\rm L}^\infty(G\wr_*S_4^+)$ is maximal amenable.
\end{enumerate}
\end{theorem}

\begin{proof}$(1).$ We use Serre's \textit{d\'evissage} in the following way. Fix $1\leq i\leq N$ and Let $\mathcal{T}'_N$ be the graph obtained from $\mathcal{T}_N$ by removing the edge $v_i$ as well as its inverse edge $\overline{v}_i$. This graph is still connected. Let us denote by $M_2$ the fundamental von Neumann algebra of our graph of von Neumann algebras restricted to $\mathcal{T}'_N$ so that we have $\Linf(H)\ot \Linf(S_N^+)\subset M_2$. It follows from \cite{FF14} that $\Linf(\wreathH)$ is canonically isomorphic to $M_1\underset{B}{*}M_2$, where the amalgamation is $B=\Linf(H)\otimes\C^N\subset M_1:=\Linf(G)\ot\C^N$ and $B=\Linf(H)\ot L_i\subset\Linf(H)\ot\Linf(S_N^+)\subset M_2$. Note that the von Neumann algebra generated by $\{\nu_i(a)u_{ij}\,:\, a\in C(G),\,1\leq j\leq N\}$ is then identified with $M_1\subset M_1\underset{B}{*}M_2$ and $M_1\simeq\Linf(G)\ot\C^N\simeq\Linf(G\times\widehat{\Z/N\Z})$ is diffuse, by Lemma \ref{LemmaInfinite}. Hence $M_1\nprec_{M_1}B$ and since $M_1$ is amenable, we may apply \cite[Main Theorem]{BH18} to deduce that $M_1$ is maximal amenable.

\vspace{0.2cm}

\noindent$(2).$ Consider the inductive construction $\Linf(H)\ot\Linf(S_4^+)=M_0\subset M_1\subset\dots\subset M_N=\Linf(G\wr_{*,H} S_4^+)$, where $M_{k+1}=\left(\Linf(G)\ot\C^N\right)\underset{\Linf(H)\ot\C^N}{*} M_k$. Recall that $\Linf(S_4^+)$ is diffuse and amenable and so is $\Linf(H)\ot\Linf(S_4^+)=M_0$. Let $M_0\subset Q\subset M_N$ be an amenable von Neumann algebra. Then $M_0\subset Q\cap M_{N-1}$ so $Q\cap M_{N-1}\underset{M_{N-1}}{\nprec}\Linf(H)\ot\C^N$ since $M_0$ is diffuse. Hence, we can apply \cite[Main Theorem]{BH18} to the amalgamated free product $M_N=\left(\Linf(G)\ot\C^N\right)\underset{\Linf(H)\ot\C^N}{*}M_{N-1}$ and deduce that $Q\subset M_{N-1}$. By a direct induction we find that $Q\subset M_0$. Note that the Kac assumption on $G$ is to ensure that each $M_k$ is finite so the inclusion $Q\cap M_k\subset M_k$ is always with expectation. We do not know if this result still holds in the non Kac case.\end{proof}

\section{Free wreath product of a fundamental quantum group}\label{SectionGraphCQG}

\noindent Let $(\G, (G_{q})_{q\in V(\G)}, (G_{e})_{e\in E(\G)}, (s_{e})_{e\in E(\G)})$ be \textit{a graph of CQG} over the connected graph $\G$  i.e.:
\begin{itemize}
\item For every $q\in V(\G)$ and every $e\in E(\G)$, $G_{q}$ and $G_{e}$ are CQG.
\item For all $e\in E(\G)$, $G_{\overline{e}}=G_e$.
\item For every $e\in E(\G)$, $s_e\,:\,C(G_e)\rightarrow C(G_{s(e)})$ is a faithful unital $*$-homomorphism intertwining the comultiplications (so $G_e$ is a dual quantum subgroup of $G_{s(e)}$).
\end{itemize}

\noindent Consider the graph of $C^*$-algebras $(\mathcal{G},C(G_q),C(G_e),s_e)$ and fix a maximal subtree $\mathcal{T}\subset\mathcal{G}$. Define $C(G)$ as the maximal fundamental C*-algebra of the graph of C*-algebras $(\mathcal{G},C(G_q),C(G_e),s_e)$ relative to the maximal subtree $\mathcal{T}$. By the universal property of $C(G)$, there exists a unique unital $*$-homomorphism $\Delta\,:\, C(G)\rightarrow C(G)\ot C(G)$ such that $\Delta\vert_{C(G_q)}=\Delta_{C(G_q)}$ and $\Delta(u_e)=u_e\ot u_e$ for all $q\in V(\G)$ and all $e\in E(\G)$. The pair $G:=(C(G),\Delta)$ is a CQG, known as the fundamental quantum group and studied in \cite{FF14}. We will denote this CQG by $G=\pi_1(\G,G_p,G_e,\mathcal{T})$. It is known from \cite{FF14} that $C(G)$ is indeed the full C*-algebra of $G$. Let us note that, since $s_e$ identifies $G_e$ as a dual quantum subgroup of $G_{s(e)}$, we have conditional expectations $E_e^s\,:\,C(G_{s(e)})\rightarrow s_e(C(G_e))$ which are non-necessary GNS-faithful.

\vspace{0.2cm}

\noindent Since $G_e\subset G_{s(e)}$ is a dual quantum subgroup $s_e$ factorizes to a faithful unital $*$-homomorphism $s_e\,:\, C_r(G_e)\rightarrow C_r(G_{s(e)})$ intertwining the comultiplications and we have Haar-state-preserving faithful conditional expectations $C_r(G_{s(e)})\rightarrow s_e(C_r(G_e))$. Hence we get a graph of C*-algebras $(\G,C_r(G_q),C_r(G_e),s_e)$ with faithful conditional expectations whose vertex-reduced fundamental C*-algebra is the reduced C*-algebra $C_r(G)$ of $G$ \cite{FF14}.

\begin{definition}
The \textit{loop subgroup} of $G$ is the group $\Gamma\subset\mathcal{U}(C(G))$ generated by $\{u_e\,:\,e\in E(\G)\}$.
\end{definition}

\noindent Note that the inclusion $\Gamma\subset\mathcal{U}(C(G))$ extends to a unital $*$-homomorphism $\pi\,:\,C^*(\Gamma)\rightarrow C(G)$.

\begin{proposition}
For $G=\pi_1(\G,G_p,G_e,\mathcal{T})$ with loop subgroup $\Gamma$ the following holds.
\begin{enumerate}
    \item $\pi\,:\,C^*(\Gamma)\rightarrow C(G)$ is faithful and intertwines the comultiplications.
    \item For all $e\in E(\G)$, $N\in\N^*$, there exists a unique unital $*$-homomorphism
    $$s_e^N\,:\,C\left(G_e\wr_*S_N^+\right)\rightarrow C\left(G_{s(e)}\wr_*S_N^+\right)\text{ such that }s_e^N\circ\nu_i=\nu_i\circ s_e,\,\, s_e^N\vert_{C(S_N^+)}=\id,\,\,\forall i.$$
\end{enumerate}
    Moreover, $s_e^N$ intertwines the comultiplications and restricts to a faithful map $\Pol(G_e\wr_* S_N^+)\hookrightarrow\Pol(G_{s(e)}\wr_* S_N^+)$ i.e. $G_e\wr_* S_N^+$ is a dual quantum subgroup of $G_{s(e)}\wr_* S_N^+$.

\end{proposition}

\begin{proof}
$(1).$ Let $C_\Gamma$ be the image of $\pi$ i.e. the C*-algebra generated by $\{u_e\,:\, e\in E(\G)\}$ in $C(G)$. To show that that $\pi$ is faithful, it suffices to check that $C_\Gamma$ satisfies the universal property of $C^*(\Gamma)$. Recall that $\varepsilon_{s(e)}\circ s_e(b)=\varepsilon_e(b)$, where, for $p\in V(\G)$, $\varepsilon_p$ is the counit on $C(G_p)$ and, for $e\in E(\G)$, $\varepsilon_e$ is the counit on $C(G_e)$. Let now $\rho\,:\,\Gamma\rightarrow\mathcal{U}(H)$ be a unitary representation of $\Gamma$. By the universal property of $C(G)=\pi_1(\G,C(G_p),C(G_e),\mathcal{T})$, there exists a unique unital $*$-homomorphism $\widetilde{\rho}\,:\, C(G)\rightarrow \mathcal{B}(H)$ such that $\widetilde{\rho}(u_e)=\rho(u_e)$ and $\widetilde{\rho}(a)=\varepsilon_p(a)\id_H$, for all $a\in C(G_p)$ and all $p\in V(\G)$. Hence, $\widetilde{\rho}\vert_{C_\Gamma}\,:\, C_\Gamma\rightarrow\mathcal{B}(H)$ is a unital $*$-homomorphism such that $\widetilde{\rho}(g)=\rho(g)$, for all $g\in\Gamma$. It follows that $\pi$ is faithful. The fact that $\pi$ intertwines the comultiplications follows from the equality $\Delta(u_e)=u_e\ot u_e$ for all $e\in E(\G)$.

\vspace{0.2cm}
\noindent$(2).$ The existence and the properties of the morphism $s_e^N$ follows from Proposition \ref{quantumsub}.
\end{proof}

\noindent By the previous Proposition, we will always view $C^*(\Gamma)=C(\widehat{\Gamma})\subset C(G)$ as the C*-algebra generated by $\{u_e\,:\,e\in E(\G)\}$ and such that the inclusion intertwines the comultiplications. In particular, $\widehat{\Gamma}$ is a dual quantum subgroup of $G$. Moreover, the previous Proposition shows that $(\G, (G_{q}\wr_* S_N^+)_{q\in V(\G)}, (G_{e}\wr_* S_N^+)_{e\in E(\G)}, (s^N_{e})_{e\in E(\G)})$ is a graph of quantum groups. Its fundamental quantum group is determined in the following Theorem.

\begin{theorem}\label{THMGraphCQG2}
If $G=\pi_1(\G,G_p,G_e,\mathcal{T})$ with loop subgroup $\Gamma$ then $$G\wr_{*,\widehat{\Gamma}}S_N^+\simeq\pi_1\left(\G,G_p\wr_*S_N^+,G_e\wr_*S_N^+,\mathcal{T}\right).$$
\end{theorem}

\begin{proof}
Define $\mathcal{A}:=C\left(\pi_1\left(\G,G_p\wr_*S_N^+,G_e\wr_*S_N^+,\mathcal{T}\right)\right)$ and $\B:=C\left (G\wr_{*,\widehat{\Gamma}}S_N^+\right)$. For all $p\in V(\G)$, we have a unital $*$-homomorphism $\varphi_p\,:\, C(G_p\w\SN) \rightarrow \B$ coming from the inclusion $C(G_p) \subset C(G)$. Moreover, $\forall e\in E(\G)\setminus E(\T)$, there is a unitary $u_e\in \B$ such that
\begin{equation*}
    u_e^*\varphi_{s(e)}\circ s_e(a) u_e = \varphi_{r(e)}\circ r_e(a)\quad\forall a\in C(G_e\w\SN),
\end{equation*}
the unitary being the one in $C^*(\Gamma)$, which is unique because of the amalgamation and works for every copy of $C(G_e)$ in $C(G_e\w\SN)$ by construction of $G$. Hence, the universal property of $\A$ gives a morphism $\pi\,:\,\A \rightarrow \B$ which clearly intertwines the comultiplications. We will give an inverse of this map, using the universal property of $\B$. For $1\leq i \leq N$, there is for any $p\in V(\G)$ a map $\nu_i \,:\,C(G_p) \rightarrow \A$ sending $C(G_p)$ on its $i$-th copy in the free product defining $C(G_p\w \SN)$ in $\A$. There is also for every edge $e\in E(\G)\setminus E(\T)$ a unitary $u_e\in \A$, which satisfies the relations of the universal property of $\B$, and is the same for every one of the $N$ copies of $C(G_e)$. Therefore, we get a map $\psi_0\,:\,C(\wreath)\rightarrow \A$ which factors through $\B$, because the unitaries $u_e$ are independent of the copy of $C(G_e)$, thus we can factor the map through the quotient amalgamating the algebras generated by the $u_e^i$ for $1\leq i \leq N$, which is exactly $\B$. We get a morphism $\psi : \B \rightarrow\A$ for which it is easy to see that it intertwines the comultiplications. The maps are inverse of each other, because they send the $N$ copies of $C(G)$ to $N$ corresponding copies of $C(G)$ and respect the unitaries from the fundamental algebra construction.
\end{proof}

\begin{example}\label{ExGraphCQG}
Suppose that $\G$ has two edges $e$ and $\overline{e}$. We have two cases.
\begin{enumerate}
    \item If $s(e)\neq r(e)$ then $\G$ is a tree so $\mathcal{T}=\mathcal{G}$ and the loop subgroup is trivial. We are in the situation of an amalgamated free product $G=G_1\underset{H}{*} G_2$ and, by Theorem \ref{THMGraphCQG2},
    $$G\wr_* S_N^+\simeq G_1\wr_* S_N^+\underset{H\wr_* S_N^+}{*} G_2\wr_*S_N^+.$$
    \item If $s(e)=r(e)$ then $\mathcal{T}$ has no edges and the loop subgroup is $\Gamma=\langle u_e\rangle\simeq\Z$. We are in the situation of an HNN extension i.e. $H$ and $\Sigma$ are CQG such that $C(\Sigma)\subset C(H)$ is a dual quantum subgroup and $\theta\,:\,C(\Sigma)\hookrightarrow C(H)$ is a faithful unital $*$-homomorphism intertwining the comultiplications then the HNN extension $G:={\rm HNN}(H,\Sigma,\theta)$ is the CQG with $C(G)$ the universal unital C*-algebra generated by $C(H)$ and a unitary $u\in C(G)$ with the relations $\theta(b)=ubu^*$ for all $b\in C(\Sigma)\subset C(H)$ and comultiplication given by $\Delta(u)=u\ot u$ and $\Delta\vert_{C(H)}=\Delta_H$. The loop subgroup $\Gamma=\Z$ satisfies $C^*(\Gamma)=\langle u\rangle\subset C(G)$ and, by Theorem \ref{THMGraphCQG2}, $\HNN(H,\Sigma,\theta)\wr_{*,\hat{\Gamma}}\SN \simeq \HNN(H\w\SN,\Sigma\w\SN,\tilde{\theta})$.
\end{enumerate}
\end{example}

\section{K-theory}\label{SectionKtheory}

\noindent In \cite{FG18} is obtained a 6-term exact sequences for the KK-theory of the reduced and the full fundamental algebras of any graph of $C^*$-algebras $(\G,A_p,B_e,s_e)$ with (non-necessary GNS-faithful) conditional expectations. It is shown in \cite{FG18} that the canonical surjection from the full fundamental C*-algebra $P$ to the vertex-reduced fundamental C*-algebra $P_r$ is a KK-equivalence and, denoting by $P_\bullet$ either $P$ or $P_r$ one has the following exact sequences, for any C*-algebra $C$.

\begin{equation}\label{Sequence1}
\begin{tikzcd}
\bigoplus_{e\in E^+(\G)} KK^0(C,B_e) \arrow[r,"\sum s_e^*-r_e^*"] & \bigoplus_{p\in V(\G)} KK^0(C,A_p) \arrow[r] & KK^0(C,P_\bullet) \arrow[d]\\
KK^1(C,P_\bullet) \arrow[u] & \bigoplus_{p\in V(\G)} KK^1(C,A_p) \arrow[l] & \bigoplus_{e\in E^+(\G)} KK^1(C,B_e) \arrow[l,"\sum s_e^*-r_e^*"],
\end{tikzcd}
\end{equation}

\begin{equation}\label{Sequence2}
\begin{tikzcd}
\bigoplus_{e\in E^+(\G)} KK^0(B_e,C) \arrow[d] & \bigoplus_{p\in V(\G)} KK^0(A_p,C) \arrow[l,"\sum {s_e}_*-{r_e}_*"] & KK^0(P_\bullet,C) \arrow[l]\\
KK^1(P_\bullet,C) \arrow[r] & \bigoplus_{p\in V(\G)} KK^1(A_p,C) \arrow[r,"\sum {s_e}_*-{r_e}_*"] & \bigoplus_{e\in E^+(\G)} KK^1(B_e,C) \arrow[u].
\end{tikzcd}
\end{equation}

\noindent Recall that, given a CQG $G$, $C_\bullet(G)$ denotes either $C(G)$ or $C_r(G)$.

\vspace{0.2cm}

\noindent\textit{Proof of Theorem \ref{THMD}.} We can use the exact sequence $(\ref{Sequence1})$ with $C= \mathbb{C}$ and with the graph of C*-algebra over $\mathcal{T}_N$ from Section \ref{SectionGraphFull} in the case of $C(G)$ and from Section \ref{SectionGraphReduced} in the case of $C_r(G)$ since both graphs of $C^*$-algebras have injective connecting maps and conditional expectations.
{\small\begin{equation*}
\begin{tikzcd}[cramped]
\underset{v\in E^+}{\bigoplus} K_0(C_\bullet(H)\ot \C^N) \arrow[r,"\sum s_v^*-r_v^*"] & \underset{{p\in V\setminus\lbrace p_0\rbrace}}{\bigoplus} K_0(C_\bullet(G)\otimes \mathbb{C}^N) \oplus K_0(C_\bullet(S_N^+)) \arrow[r] & K_0(C_\bullet(\wreathH)) \arrow[d]\\
K_1(C_\bullet(\wreathH)) \arrow[u] & \underset{p\in V\setminus\lbrace p_0\rbrace}{\bigoplus} K_1(C_\bullet(G)\otimes\C^N) \oplus K_1(C_\bullet(S_N^+))\arrow[l] & \underset{v\in E^+}{\bigoplus} K_1(C_\bullet(H)\ot \C^N) \arrow[l,"\sum s_v^*-r_v^*"].
\end{tikzcd}
\end{equation*}
Using the compatibility of the K-theory with direct sums, we get the following:
\begin{equation*}
\begin{tikzcd}
 K_0(C_\bullet(H))\ot \Z^{N^2} \arrow[r,"\sum s_v^*-r_v^*"] &   \left(K_0(C_\bullet(G))\ot \Z^{N^2}\right)  \oplus K_0(C_\bullet(S_N^+)) \arrow[r] & K_0(C_\bullet(\wreathH)) \arrow[d]\\
K_1(C_\bullet(\wreathH)) \arrow[u] & \left( K_1(C_\bullet(G))\ot \Z^{N^2} \right) \oplus K_1(C_\bullet(S_N^+))\arrow[l] &  K_1(C_\bullet(H))\ot \Z^{N^2} \arrow[l,"\sum s_v^*-r_v^*"].
\end{tikzcd}
\end{equation*}
In what follows we restrict to the cases where $H$ is the trivial group, so that the maps $\sum s_v^*-r_v^*$ are always injective, as shown just below. We will however come back to the above sequence in some special cases.
Let $e_j\in\C^N$, $1\leq j\leq N$ the canonical basis made of pairwise orthogonal rank one projections so that $\{[e_j]\,:\,1\leq j\leq N\}$ is a basis of $K_0(\C^N)\simeq\Z^N$. One checks easily that $\{[1\ot e_j]\,:\,1\leq j\leq N\}$ is linearly independent in  $K_0(C_\bullet(G)\ot\C^N)$ hence for every $1\leq i\leq N$, the elements
$$[1 \otimes e_j] - [u_{ji}] \in K_0((C(G)\ot \C^N))\oplus K_0(C_\bullet(S_N^+)),\,\,1\leq j\leq N$$
are also linearly independent and the map
$$s_{v_i}^*-r_{v_i}^*\,:\,K_0(\C^N)=\Z^N\rightarrow K_0((C(G)\ot \C^N)) \oplus K_0(C_\bullet(S_N^+)),\,\,
    (s_{v_i}^* - r_{v_i}^*)  [e_j]  = [1 \otimes e_j] - [u_{ji}]$$
is injective, for all $1\leq i\leq N$. Now, let us denote by $[e_{ij}]\in\left(\Z^N\right)^{\oplus N}$ the element $[e_j]$ viewed in the $i^{th}$-copy of $\Z^N$ so that $\{[e_{ij}]\,:\,1\leq i,j\leq N\}$ is a basis of $\left(\Z^N\right)^{\oplus N}$ and, similarly, we denote by $[1\ot e_{ij}]\in K_0(C_\bullet(G)\ot\C^N)^{\oplus N}$ the element $[1\ot e_j]$ viewed in the $i^{th}$ copy of $K_0(C_\bullet(G)\ot\C^N)$ so that the map $\psi:=\sum_{v\in E^+}s_v^*-r_v^*\,:\,\left(\Z^N\right)^{\oplus N}\rightarrow K_0(C_\bullet(G)\ot\C^N)^{\oplus N}\bigoplus K_0(C_\bullet(S_N^+))$ is given by $\psi([e_{ij}])=[1\ot e_{ij}]-[u_{ji}]$. As before, since $\{[1\ot e_{ij}]\,:\,1\leq i,j\leq N\}$ is linearly independent in $K_0(C_\bullet(G)\ot\C^N)^{\oplus N}$, it follows that $\{[1\ot e_{ij}]-[u_{ji}]\,:\,1\leq i,j\leq N\}$ is also linearly independent in $K_0(C_\bullet(G)\ot\C^N)^{\oplus N}\bigoplus K_0(C_\bullet(S_N^+))$ so $\psi$ is injective.

\vspace{0.2cm}

\noindent From this we get the isomorphism for $K_1$ and, for $K_0$, we have:
\begin{equation*}
    K_0(C_\bullet(\wreath)) \simeq \left(K_0(C_\bullet(G)\ot \C^N)^{\oplus N} \oplus K_0(C_\bullet(S_N^+)) \right)/ \text{Im}(\psi)
\end{equation*}

\noindent The K-theory of $C_\bullet(\SN)$, where the two algebras are KK-equivalent because $\widehat{S_N^+}$ is K-amenable, is computed by Voigt in \cite{Vo17} for $N\geq 4$ and, for $1\leq N\leq 3$ we have $C_\bullet(S_N^+)=C(S_N)=\C^{N!}$ hence,
$$
K_0(C_\bullet(\SN)) \simeq \left\{\begin{array}{lcl} \Z^{N^2-2N+2}&\text{if}&N\geq 4,\\\Z^{N!}&\text{if}&1\leq N\leq 3.\end{array}\right.\quad\text{and}\quad
K_1(C_\bullet(\SN))\simeq\left\{\begin{array}{lcl} \Z&\text{if}&N\geq 4,\\0&\text{if}&1\leq N\leq 3.\end{array}\right.
$$

\noindent When $N\geq 4$, a family of generators of $K_0(C_\bullet(\SN))$ is given by the classes of $(N-1)^2$ coefficients of the fundamental representation of $\SN$, $[u_{ij}]$, $1\leq i,j\leq N-1$, and the class $[1]$ of the unit.
\begin{align*}
    K_0(C_\bullet(\wreath)) \simeq &\left(K_0(C_\bullet(G)\ot \C^N))^{\oplus N} \oplus K_0(C_\bullet(S_N^+)) \right)/ \text{Im}(\psi)\\
    \simeq&  \left( K_0(C_\bullet(G))\ot \Z^{N^2}) \oplus \Z^{N^2-2N+2} \right)/ \left< ([1]\ot[e_{ij}])-[u_{ji}],1\leq i,j\leq N\right>\\
    \simeq & \left( K_0(C_\bullet(G))\ot \Z^{N^2} \right)/ ([1]\ot \Z^{2N-2}).
\end{align*}
The last quotient is obtained after identifying the classes $[1\ot e_{ij}]$ with $[u_{ji}]$. This gives the first part of the statement of Theorem \ref{THMD} for $N\geq 4$. For the computation of the K-theory of quantum reflection groups $H^{s+}_N = \widehat{\Z_s}\w S_N^+$, for $1\leq s \leq +\infty$, which have K-amenable duals, note that, for $s<\infty$, since $C^*(\Z_s)\simeq C(\widehat{\Z_s}) \simeq \C^s$, we have $K_0(C^*(\Z_s)) \simeq \Z^s$ and $K_1(C^*(\Z_s)) \simeq \lbrace 0\rbrace$. Hence,
$$
    K_0(C_\bullet(H^{s+}_N)) \simeq \left( K_0(C^*(\Z_s))\ot \Z^{N^2} \right)/ ([1]\ot \Z^{2N-2})
     \simeq (\Z^{s}\ot \Z^{N^2})/([1]\ot \Z^{2N-2})
     \simeq \Z^{sN^2-2N+2}.
$$
The same computation works when $s=+\infty$ with $K_0(C^*(\Z)) \simeq \Z$, for $K_0$, but for $K_1$, as $K_0(C^*(\Z)) \simeq \Z$. $\hfill\square$

\vspace{0.2cm}

\noindent We are grateful to Adam Skalski for suggesting to us to use the uniqueness of the trace from \cite{lem14} to deduce the following from Theorem \ref{THMD}.

\begin{corollary}\label{CORReflection}
If $M,N\geq 8$ and $s,t\geq 1$, then $C_r(H_N^{s+})\simeq C_r(H_M^{t+})\Leftrightarrow(N,s)=(M,t)$.
\end{corollary}

\begin{proof}
The dimensions of the K-theory groups may happen to be equal for different pairs of integers, so we need to go a bit further to differentiate between them. We use the value of the Haar state applied to the generators of $K_0(C_r (H_N^{s+}))\simeq \Z^{sN^2-2N+2}$. Thanks to the computation, we know that there are, in addition to the class of the unit $[1]$, $N^2-2N+1$ generators coming from the ones of $K_0(C_r(S_N^+))$, namely $[1 \ot u_{ij}]$ for $1\leq i,j\leq N-1$, they are equal to $[u_{ij}]$ in $K_0(C_r(H_N^{s+}))$. The trace of such elements is the same as the trace of the corresponding element in $C_r(\SN)$, thanks to \ref{propreduced}, which is equal to $1/N$. The $(s-1)N^2$ remaining generators are the ones coming from the $N$ copies of $K_0(C^*(\Z_s)\ot \C^N)$ and are of the form $[\delta_k\ot e_j]$, for $k\in \Z_s$, $k\neq 0$, and $1\leq j\leq N$. Such a class, in the $i$-th copy, is sent to $[\nu_i(\delta_k)u_{ij}]$, which is of trace $1/(sN)$. Thus, using the uniqueness of the trace of these algebras, as proved in \cite{lem14}, we get that the pair $(N,s)$ can be retrieved from the data of the K-theory, and it allows to discriminate the algebras of the different quantum reflection groups.
\end{proof}

\begin{corollary}\label{CORKthFree}
For all $m\geq 1$, and $N\geq 4$, the dual of $\widehat{\mathbb{F}}_m\w S_N^+$ is K-amenable and we have: $$K_0(C_\bullet(\widehat{\mathbb{F}}_m\w S_N^+)) \simeq \Z^{N^2-2N+2}\quad\text{and} \quad K_1(C_\bullet(\widehat{\mathbb{F}}_m\w S_N^+)) \simeq \Z^{N^2m+1}.$$
If $1\leq N\leq 3$, then the dual of $\widehat{\mathbb{F}}_m\w S_N^+$ is also K-amenable, with $$K_0(C_\bullet(\widehat{\mathbb{F}}_m\w S_N^+)) \simeq \Z^{N!}\quad\text{and} \quad K_1(C_\bullet(\widehat{\mathbb{F}}_m\w S_N^+)) \simeq \Z^{N^2m} .$$
In particular, for all $n,m\geq 1$ and $N,M\geq 1$, $C_\bullet(\widehat{\mathbb{F}}_n\w S_N^+)\simeq C_\bullet(\widehat{\mathbb{F}}_m\w S_M^+)\Leftrightarrow (n,N)=(m,M)$.
\end{corollary}

\begin{proof}
K-amenability of free groups has been proved by Cuntz in \cite{Cun83}, when he introduced the notion of K-amenability for discrete groups. The K-theory for the maximal C*-algebra was initially computed in \cite{Cun82} and for the reduced C*-algebra in \cite{PV82}. The result is $K_0(C^*(\F_m)) \simeq \Z$ and $K_1(C^*(\F_m)) \simeq \Z^m$. Using this and Theorem \ref{THMD}, we get the first result. For the last statement, we first use equality of the $K_0$-groups to deduce that $N=M$ and then, equality of the $K_1$-groups to deduce that $n=m$.\end{proof}

\noindent We can use the results of Section \ref{SectionGraphCQG} to compute the KK-theory of $C^*$-algebras of a free wreath product of a fundamental quantum group of graph of CQG. The main theorem is the following, using the notations of Section \ref{SectionGraphCQG} and denoting by $E^+$ and $V$ the positive edges and vertices of the connected graph $\G$.

\begin{theorem}\label{THMExactSeqGraphCQG}
If $G=\pi_1(\G,G_p,G_e,\mathcal{T})$ for any $C^*$-algebra $A$, there is a cyclic exact sequences:
\begin{equation*}
\begin{tikzcd}[cramped]
 \underset{e\in E^+}{\bigoplus} \scalemath{0.9}{KK^0(A, C_\bullet(G_e\wr_*S_N^+))} \arrow[r,"\sum s_e^*-r_e^*"] & \underset{p\in V}{\bigoplus} \scalemath{0.9}{KK^0(A,C_\bullet(G_p\wr_*S_N^+))} \arrow[r] & \scalemath{0.9}{KK^0(A,C_\bullet(G\wr_{*,\Gammahat}\SN))} \arrow[d]\\
\scalemath{0.9}{KK^1(A,C_\bullet(G\wr_{*,\Gammahat}\SN))} \arrow[u] & \underset{p\in V}{\bigoplus} \scalemath{0.9}{KK^1(A,C_\bullet(G_p\wr_*S_N^+))} \arrow[l] & \underset{e\in E^+}{\bigoplus} \scalemath{0.9}{KK^1(A,C_\bullet(G_e\wr_*S_N^+))} \arrow[l,"\sum s_e^*-r_e^*"]
\end{tikzcd}
\end{equation*}

\begin{equation*}
\begin{tikzcd}[cramped]
\underset{e\in E^+}{\bigoplus} \scalemath{0.9}{KK^0(C_\bullet(G_e\wr_*S_N^+),A)}\,\,\, \arrow[d] & \underset{\,\,p\in V}{\bigoplus} \scalemath{0.9}{KK^0(C_\bullet(G_p\wr_*S_N^+),A)} \arrow[l,"\sum {s_e}_*-{r_e}_*\,\,"] & \scalemath{0.9}{KK^0(C_\bullet(G\wr_{*,\Gammahat}\SN),A)} \arrow[l]\\
\scalemath{0.9}{KK^1(C_\bullet(G\wr_{*,\Gammahat}\SN),A)} \arrow[r] & \underset{p\in V}{\bigoplus} \scalemath{0.9}{KK^1(C_\bullet(G_p\wr_*S_N^+),A)}\,\,\, \arrow[r,"\sum {s_e}_*-{r_e}_*"] & \underset{e\in E^+}{\bigoplus} \scalemath{0.9}{KK^1(C_\bullet(G_e\wr_*S_N^+),A)} \arrow[u]
\end{tikzcd}
\end{equation*}
\end{theorem}

\begin{proof}
As observed in Section \ref{SectionGraphCQG} we have, at the level of full as well as reduced C*-algebras, a graph of C*-algebras $(\G,C_\bullet(G_q\wr_* S_N^+),C_\bullet(G_q\wr_* S_N^+),s_E^N)$ with conditional expectations (which are GNS-faithful only at the reduced level). By Theorem \ref{THMGraphCQG2} the full/reduced C*-algebra $C_\bullet(G\wr_{*,\widehat{\Gamma}}S_N^+)$ is the full/vertex-reduced fundamental algebra of $(\G,C_\bullet(G_q\wr_* S_N^+),C_\bullet(G_r\wr_* S_N^+),s_E^N)$. Hence, we may apply the exact sequences $(\ref{Sequence1})$ and $(\ref{Sequence2})$.\end{proof}

\noindent We compute now K-theory groups of some free wreath products of an amalgamated free product with $S_N^+$.

\begin{corollary}\label{freeprodK}
If $G = G_1 *_H G_2$ then there is a cyclic sequence of K-theory groups:
\begin{equation*}
\begin{tikzcd}
 K_0(C_\bullet(H\w \SN)) \arrow[r] & K_0(C_\bullet(G_1\w \SN))\oplus K_0(C_\bullet(G_2\w\SN)) \arrow[r] & K_0(C_\bullet(\wreath)) \arrow[d]\\
 K_1(C_\bullet(\wreath)) \arrow[u] & K_1(C_\bullet(G_1\w \SN))\oplus K_1(C_\bullet(G_2\w\SN)) \arrow[l] &  K_1(C_\bullet(H\w \SN))\arrow[l] .
\end{tikzcd}
\end{equation*}
\end{corollary}

\begin{proof}
As observed in Example \ref{ExGraphCQG} the loop subgroup is trivial in the case of an amalgamated free product. The proof follows from the first exact sequence in Theorem \ref{THMExactSeqGraphCQG} with $A=\C$.\end{proof}

\begin{corollary}
For the compact quantum group $\widehat{{\rm SL}_2(\Z)}\w \SN$, which has K-amenable dual, we have for $N\geq 4$:
\begin{equation*}
    K_0(C_\bullet(\widehat{{\rm SL}_2(\Z)}\w \SN)) \simeq \Z^{8N^2-2N+2}\quad\text{and}\quad K_1(C_\bullet(\widehat{{\rm SL}_2(\Z)}\w \SN))\simeq\Z.
\end{equation*}
For $1\leq N\leq 3$, we have:
\begin{equation*}
    K_0(C_\bullet(\widehat{{\rm SL}_2(\Z)}\w \SN))  \simeq\left\{
\begin{array}{lcl}\Z^{69}&\text{if}&N=3,\\ \Z^{8N^2-2N+2}&\text{if}&N\in \lbrace1,2\rbrace.\end{array}\right. \quad\text{and}\quad K_1(C_\bullet(\widehat{{\rm SL}_2(\Z)}\w \SN))\simeq 0.
\end{equation*}
\end{corollary}

\begin{proof}
We use the well known isomorphism ${\rm SL}_2(\Z) \simeq \Z_6 *_{\Z_2} \Z_4$ which implies the K-amenability of ${\rm SL}_2(\Z)$ hence of the dual of $\widehat{{\rm SL}_2(\Z)}\w \SN$ as well by Theorem \ref{THMA}. Moreover, by Example \ref{ExGraphCQG}, we have $\widehat{{\rm SL}_2(\Z)}\w \SN\simeq H_N^{6+}*_{H_N^{2+}} H_N^{4+}$. We may use the K-theory of the quantum reflection groups computed in Theorem \ref{THMD} to deduce the one for the group $\widehat{{\rm SL}_2(\Z)}\w \SN$. Applying Corollary \ref{freeprodK} and $K$-amenability, we get the following cyclic exact sequence
\begin{equation*}
\begin{tikzcd}
 K_0(C_\bullet(H_N^{2+})) \arrow[r,"\psi"] & K_0(C_\bullet(H_N^{6+}))\oplus K_0(C_\bullet(H_N^{4+})) \arrow[r] &  K_0(C_\bullet(\widehat{{\rm SL}_2(\Z)}\w \SN))\arrow[d]\\
K_1(C_\bullet(\widehat{{\rm SL}_2(\Z)}\w \SN)) \arrow[u] & K_1(C_\bullet(H_N^{6+}))\oplus K_1(C_\bullet(H_N^{4+})) \arrow[l] &  K_1(C_\bullet(H_N^{2+}))\arrow[l] .
\end{tikzcd}
\end{equation*}
The maps in the bottom line are the diagonal embedding and the map $(x,y)\mapsto x-y$, which are respectively injective and surjective $\Z \rightarrow\Z\oplus \Z \rightarrow \Z$. Hence, the top line is a short exact sequence and the  map $\psi$ is injective, and we have $$K_0(C_\bullet(\widehat{{\rm SL}_2(\Z)}\w \SN)) \simeq \Z^{6N^2-2N+2}\oplus\Z^{4N^2-2N+2}/\mathrm{Im} (\psi).$$
The map $\psi$ sends the generators of $K_0(C_\bullet(H_N^{2+}))$, coming from the classes of the form $[v\ot e_j]$ in $K_0(C^*(\Z_2)\ot \C^N)$ to the corresponding classes $[\beta^*v\ot e_j]-[\alpha^*v\ot e_j]$, where $\beta : \Z_2 \rightarrow \Z_6$ and $\alpha : \Z_2 \rightarrow \Z_4$ are the canonical embeddings. Hence, the quotient by $\mathrm{Im}(\psi)$ identifies free copies of $\Z$ and the group is isomorphic to $\Z^{8N^2-2N+2}$.
\end{proof}

\begin{remark}
We can use Proposition \ref{freeprodK} to give another proof of Corollary \ref{CORKthFree}. It is also possible to compute first the K-theory of the amalgamated free products and then apply Theorem \ref{THMD} to get the results of \ref{freeprodK}.
\end{remark}

\noindent Let us now do a K-theory computation in the context of an HNN extension. Recall that the Baumslag-Solitar group ${\rm BS}(n,m)$, for $n,m\in\Z^*$, is defined by generators and relations:
$${\rm BS}(n,m):=\langle a,b\,\vert\,ab^na^{-1}=b^m\rangle.$$
Let $\langle a\rangle<{\rm BS}(n,m)$ be the subgroup generated by $a$ and view $\widehat{\langle a\rangle}$ as a dual quantum subgroup of the compact quantum group $\widehat{{\rm BS}(n,m)}$.

\begin{proposition}
The dual of the compact quantum group $\widehat{{\BS}(n,m)}\wr_{\ast,\widehat{\langle a\rangle}} \SN$ is K-amenable for any $n,m\in \Z^*$ and $N\geq 1$. Its K-theory groups are given by the following, for any $n,m\in\Z^*$: if $N\geq 4$ and $n=m$, then
\[K_0(C_\bullet(\widehat{{\rm BS}(n,m)}\wr_{\ast,\widehat{\langle a\rangle}} \SN)) \simeq \Z^{2N^2-2N+3}\text{ and }K_1(C_\bullet(\widehat{{\rm BS}(n,m)}\wr_{\ast,\widehat{\langle a\rangle}} \SN)) \simeq\Z^{2N^2-2N+3},\]
and if $n\neq m$,
$$K_0(C_\bullet(\widehat{{\rm BS}(n,m)}\wr_{\ast,\widehat{\langle a\rangle}} \SN)) \simeq \Z^{N^2-2N+3},\,\,K_1(C_\bullet(\widehat{{\rm BS}(n,m)}\wr_{\ast,\widehat{\langle a\rangle}} \SN)) \simeq\Z^{N^2-2N+3}\oplus (\Z_{\vert n-m\vert})^{N^2}.$$
For $1\leq N\leq 3$, if $n=m$, then
\[K_0(C_\bullet(\widehat{{\rm BS}(n,m)}\wr_{\ast,\widehat{\langle a\rangle}} \SN)) \simeq \Z^{N^2+N!} \text{ and }K_1(C_\bullet(\widehat{{\rm BS}(n,m)}\wr_{\ast,\widehat{\langle a\rangle}} \SN)) \simeq \Z^{N^2+N!} \]
and if $n\neq m$,
\[K_0(C_\bullet(\widehat{{\rm BS}(n,m)}\wr_{\ast,\widehat{\langle a\rangle}} \SN)) \simeq \Z^{N!}\text{ and }K_1(C_\bullet(\widehat{{\rm BS}(n,m)}\wr_{\ast,\widehat{\langle a\rangle}} \SN)) \simeq\Z^{N!}\oplus (\Z_{\vert n-m\vert})^{N^2}.\]
\end{proposition}

\begin{proof}
Note that ${\rm BS}(n,m)$ is the HNN-extension ${\rm BS}(n,m)={\rm HNN}(\Z,\Z,\theta_n,\theta_m)$, where $\theta_l\,:\,\Z\rightarrow\Z$ is the multiplication by $l\in\{n,m\}$. In particular, ${\rm BS}(n,m)$ is K-amenable and the CQG $\widehat{{\rm BS}(n,m)}$ is also an HNN-extension. We still denote by $\theta_l\,:\, C_\bullet(\widehat{\Z}\wr_* S_N^+)\rightarrow C_\bullet(\widehat{\Z}\wr_* S_N^+)$ the map defined by $\theta_l(\nu_i(k)u_{ij})=\nu_i(lk)u_{ij}$ for all $l\in\{n,m\}$, $k\in\Z$ and $1\leq i,j\leq N$. By Example \ref{ExGraphCQG}, the loop subgroup is $\Gamma=\langle a\rangle$ and, by using Theorem \ref{THMExactSeqGraphCQG} (with $A=\C$), we get,
\begin{equation*}
\begin{tikzcd}
 K_0\left(C_\bullet\left(\widehat{\Z}\w \SN\right)\right) \arrow[r,"\theta_n^\ast-\theta_m^\ast"] & K_0\left(C_\bullet\left(\widehat{\Z}\w \SN\right)\right) \arrow[r] & K_0(\B) \arrow[d]\\
 K_1(\B)\arrow[u] & K_1\left(C_\bullet\left(\widehat{\Z}\w \SN\right)\right) \arrow[l] &  K_1\left(C_\bullet\left(\widehat{\Z}\w \SN\right)\right)\arrow[l,"\theta_n^\ast-\theta_m^\ast"]
\end{tikzcd}
\end{equation*}
where $\B:=C_\bullet(\widehat{{\rm BS}(n,m)}\wr_{\ast,\widehat{\langle a\rangle}} \SN)$. As the group $K_0(C_\bullet(\widehat{\Z}\w \SN))$ is generated by the image of $K_0(C_\bullet(\SN))$ in it, and since we have, for $l\in\{n,m\}$, $\theta_l^*([u_{ij}])=[\theta_l(u_{ij})]=[u_{ij}]$ and $\theta_l^*([1])=[\theta_l(1)]=[1]$, where $[u_{ij}]$ and $[1]$ are the generators of the $K_0$ group, the map $\theta_n^\ast-\theta_m^\ast$ is trivial at the $K_0$ level. It is however nontrivial at the $K_1$ level: $K_1(C(\widehat{\Z}\w \SN))$ is generated by the images of the generator $[1]$ of $K_1(C^*(\Z))$, the one of the generator of $K_1(C(\SN))$. The action of $\theta_n$ at the K-theory level is thus given by multiplication by $n$ on the generators coming from $K_1(C^*(\Z))$, and acts trivially on the one coming from $K_1(C(\SN))$. Thus the sequence splits differently depending on if $m$ and $n$ are equal or not. If $n=m$, then the maps are trivial and the sequence splits in two short sequences as follows:
$$0 \rightarrow \Z^{N^2-2N+2} \rightarrow K_0(\B)\rightarrow \Z^{N^2+1} \rightarrow 0\quad\text{and}\quad0 \rightarrow \Z^{N^2+1} \rightarrow K_0(\B)\rightarrow \Z^{N^2-2N+2} \rightarrow 0,$$
giving the first part of the result. If $m\neq n$ then the map is still trivial at the $K_0$-level, but acts as the multiplication by $m-n$ on the $N$ first generators of $K_1(C(\widehat{\Z}\w \SN))$ and as $0$ on the last one. The sequence splits in two short sequences:
$$0 \rightarrow \Z^{N^2-2N+2} \rightarrow K_0(\B)\rightarrow \Z \rightarrow 0\quad\text{and}\quad0 \rightarrow \Z^{N^2} \xrightarrow{\psi} \Z^{N^2+1} \rightarrow K_1(\B)\rightarrow \Z^{N^2-2N+2} \rightarrow 0,$$ the map $\psi$ in the second being multiplication by $(n-m)$ on each of the first $N^2$ terms of the sum, and the sequence becomes $0 \rightarrow (\Z_{\vert n-m\vert})^{N^2} \oplus \Z \rightarrow K_1(\B)\rightarrow \Z^{N^2-2N+2} \rightarrow 0$, and thus the result in the remaining case because the group on the right-hand-side is free.
The proof for the cases $1\leq N \leq 3$ works the same way.
\end{proof}

\begin{remark}
The computation of this case can also be achieved thanks to the six-term exact sequence for the free wreath product with amalgamation of Theorem \ref{THMD}, written in the beginning of the proof in section \ref{SectionKtheory}. The main point in the use of this exact sequence is that the maps $\sum s_v^*-r_v^*$ appearing in the sequence are injective and explicit in this special case, allowing an easy computation.
\end{remark}

We can compare it to the non-amalgamated case, using the following proposition about the K-theory of the Baumslag-Solitar groups $\BS(m,n)=\HNN\left(\Z,\Z,\theta_m,\theta_n\right)$. We include a proof as we couldn't find the result stated except in the solvable case in \cite{PV18}.

\begin{proposition}
Let $m$ and $n$ be integers, then the Baumslag-Solitar group $\BS(m,n)$ is K-amenable and its K-theory is given, if $n=m$, by
\[K_0(C^*_\bullet(\BS(m,m))) \simeq \Z^2\text{ and }K_1(C^*_\bullet(\BS(m,m))) =\Z^2,\]
and if $n\neq m$, by
\[K_0(C^*_\bullet(\BS(m,n))) \simeq \Z\text{ and }K_1(C^*_\bullet(\BS(m,n))) =\Z\oplus \Z_{\vert n-m\vert}.\]
\end{proposition}

\begin{proof}
The Baumslag-Solitar groups are the fundamental group of the graph with only one vertex and one edge, with $\Z$ on each and with $\Z$ being sent to the subgroups $n\Z$ and $m\Z$ by multiplication by $m$ and $n$. Thus, the corresponding C*-algebras are the full and vertex-reduced fundamental C*-algebra of this graph with $C^*(\Z)$ and the maps induced from the multiplications, which are K-equivalent. The exact sequence of K-theory is then 
\begin{equation*}
\begin{tikzcd}
 K_0\left(C^*(\Z)\right) \arrow[r,"\theta_m^\ast-\theta_n^\ast"] & K_0\left(C^*(\Z)\right) \arrow[r] & K_0(C^*_\bullet(\BS(m,n))) \arrow[d]\\
 K_1(C^*_\bullet(\BS(m,n)))\arrow[u] & K_1\left(C^*(\Z)\right) \arrow[l] &  K_1\left(C^*(\Z)\right),\arrow[l,"\theta_m^\ast-\theta_n^\ast"]
\end{tikzcd}
\end{equation*}
which becomes 
\begin{equation*}
\begin{tikzcd}
 \Z \arrow[r,"0"] & \Z \arrow[r] & K_0(C^*_\bullet(\BS(m,n))) \arrow[d]\\
 K_1(C^*_\bullet(\BS(m,n)))\arrow[u] & \Z \arrow[l] &  \Z,\arrow[l,"(m-n)"]
\end{tikzcd}
\end{equation*}
giving the result as the map $(m-n)$ is injective if $m\neq n$ and $0$ if $m=n$.
\end{proof}

From this we can compute the K-theory of the free wreath products using Theorem \ref{THMD}.

\begin{proposition}
Let $m$ and $n$ be integers and $N\geq 4$, then the dual of the compact quantum group $\widehat{{\BS}(n,m)}\w\SN$ is K-amenable and if $n=m$, then
\[K_0(C_\bullet(\widehat{{\BS}(m,n)}\w \SN)) \simeq \Z^{2N^2-2N+2}\text{ and }K_1(C_\bullet(\widehat{{\BS}(m,n)}\w \SN)) =\Z^{2N^2+1}.\]
If $n\neq m$,
\[K_0(C_\bullet(\widehat{{\BS}(m,n)}\w \SN)) \simeq \Z^{N^2-2N+2}\text{ and }K_1(C_\bullet(\widehat{{\BS}(m,n)}\w \SN)) =\Z^{N^2+1} \oplus (\Z_{\vert n-m\vert})^{N^2}.\]
For $1\leq N\leq 3$, if $n=m$, then
\[K_0(C_\bullet(\widehat{{\BS}(m,n)}\w \SN)) \simeq \Z^{N^2+N!} \text{ and }K_1(C_\bullet(\widehat{{\BS}(m,n)}\w \SN)) =\Z^{2N^2}.\]
If $n\neq m$,
\[K_0(C_\bullet(\widehat{{\BS}(m,n)}\w \SN)) \simeq\Z^{N!} \text{ and }K_1(C_\bullet(\widehat{{\BS}(m,n)}\w \SN)) =\Z^{N^2} \oplus (\Z_{\vert n-m\vert})^{N^2}.\]
\end{proposition}

\section*{Acknowledgements}

\noindent The authors would like to thank Kenny De Commer, Adam Skalski, Roland Vergnioux and Makoto Yamashita for their help and their suggestions during the redaction of this article.

\bibliography{ref}
\bibliographystyle{amsalpha}
\end{document}